\documentclass[11pt,a4paper]{amsart}

\usepackage{amssymb, amstext, amscd, amsmath, color}

\usepackage{hyperref}

\usepackage{url}
\usepackage{tikz}
\usepackage{epstopdf}
\usepackage{amsfonts}
\usepackage{amsthm}
\usepackage{amsfonts}
\usepackage{array}
\usepackage{epsfig}
\usepackage{eucal}
\usepackage{latexsym}
\usepackage{mathrsfs}
\usepackage{textcomp}
\usepackage{verbatim}
\usepackage{setspace}
\usepackage{bm}

\usepackage{algorithm}
\usepackage{algpseudocode}

\newtheorem{thm}{Theorem}[section]

\newtheorem{cor}[thm]{Corollary}

\newtheorem{defn}[thm]{Definition}
\newtheorem{example}[thm]{Example}

\newtheorem{lem}[thm]{Lemma}

\newtheorem{prop}[thm]{Proposition}
\newtheorem{rem}[thm]{Remark}

\numberwithin{equation}{section}

  \newcommand{\U}{{\mathcal{U}}}

\newcommand{\be}{{\bf{e}}}

\newcommand{\rank}{\operatorname{rank}}
\usepackage{blkarray}

\begin{document}

\title[Mobility of geometric constraint systems with extrusion symmetry]{Mobility of geometric constraint systems with extrusion symmetry}
\author[J. Owen]{J. Owen}
\address{Uplands, Bentley IP9 2DA, U.K.}
\email{john.owen.k33@icloud.com}
\author[B. Schulze]{B. Schulze}
\address{ Dept.\ Math.\ Stats.\\ Lancaster University\\
Lancaster LA1 4YF \\U.K.}
\email{b.schulze@lancaster.ac.uk}

\begin{abstract}
If we take a (bar-joint) framework, prepare an identical copy of this framework, translate it by some vector $\tau$, and finally join corresponding points of the two copies, then we obtain a  framework with `extrusion' symmetry in the direction of $\tau$. This process may be repeated $t$ times to obtain a framework whose underlying graph has $\mathbb{Z}_2^t$ as a subgroup of its automorphism group and which has `$t$-fold extrusion' symmetry. Extruding a framework is a widely used technique in CAD for generating a 3D model from an initial 2D sketch, and hence it is important to understand the flexibility of extrusion-symmetric frameworks. Using group representation theory, we show that while $t$-fold extrusion symmetry is not a point-group symmetry, the rigidity matrix of a framework with $t$-fold extrusion symmetry can still be transformed into a block-decomposed form  in the analogous way as for point-group symmetric frameworks. This allows us  to establish Fowler-Guest-type character counts to analyse the mobility of such frameworks. We show that this entire theory also extends to the more general point-hyperplane frameworks with $t$-fold extrusion symmetry. Moreover, we  show that under suitable regularity conditions the infinitesimal flexes we detect with our symmetry-adapted counts extend to finite (continuous) motions. Finally, we establish an algorithm that checks for finite motions via linearly displacing framework points along velocity vectors of  infinitesimal motions.
\end{abstract}

\date{}
\maketitle
\section{Introduction}\label{sec:intro0}
A  (bar-joint) framework is a finite simple graph embedded into Euclidean $d$-space, with  edges interpreted as stiff bars of fixed lengths and vertices as  universal joints that allow arbitrary rotations in the space. The rigidity and flexibility analysis of  frameworks and related geometric constraint systems is a major research area in applied discrete geometry (see \cite[Chapters 61--63]{SW1} and \cite{conguest,WW}, for example).  Over the last 15 years or so, there has been significant interest in studying the rigidity and flexibility of  geometric constraint systems that possess non-trivial point group symmetries. Since symmetry is a common feature of both man-made and natural structures, research problems in this area are often motivated by practical applications in science, technology, and design. We refer the reader to \cite[Chapter 62]{SW1} for a summary of results.

In this paper, we introduce a new type of symmetry, called \emph{extrusion symmetry}, which is different from point group symmetry. A basic example is shown in Figure~\ref{fig:ex1}. The idea is to start with a framework $(H,p)$ (a triangle in $\mathbb{R}^2$ in the case of Figure~\ref{fig:ex1}(a)), to create an identical copy of $(H,p)$, translate this copy along a vector $\tau$, and  join corresponding vertices by edges (of equal lengths). This results in the extruded framework $(G,q)$ (the triangular prism framework in the case of Figure~\ref{fig:ex1}(a)). The underlying graph $G$ of this framework is symmetric, since  the automorphism group of $G$ contains a subgroup that is isomorphic to $\mathbb{Z}_2$. However, in general, the framework $(G,q)$ does not have any point group symmetry. The above process of extruding a framework can be repeated mulitple, say $t$, times to obtain a framework with $t$-fold extrusion symmetry.

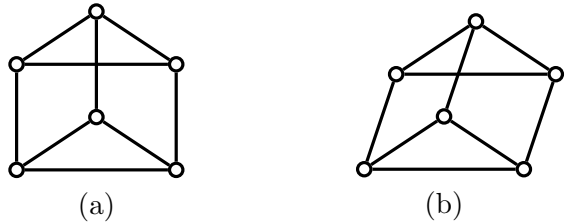
\begin{figure}[htp]
\begin{center}
\begin{tikzpicture}[very thick,scale=0.7]
\tikzstyle{every node}=[circle, draw=black, fill=white, inner sep=0pt, minimum width=5pt];
        \path (0,0) node (p1) {} ;
        \path (3,0) node (p2) {} ;
        \path (1.5,1) node (p3) {} ;
        
         \path (0,2) node (p11) {} ;
        \path (3,2) node (p22) {} ;
        \path (1.5,3) node (p33) {} ;
        
        \draw (p1)  --  (p2);
         \draw (p1)  --  (p3);
        \draw (p3)  --  (p2);
        \draw (p11)  --  (p22);
         \draw (p11)  --  (p33);
        \draw (p33)  --  (p22);
        \draw (p1)  --  (p11);
         \draw (p2)  --  (p22);
        \draw (p3)  --  (p33);
        
        \node [draw=white, fill=white] (a) at (1.5,-0.7) {(a)};
      \end{tikzpicture}
      \hspace{2cm}
      \begin{tikzpicture}[very thick,scale=0.7]
\tikzstyle{every node}=[circle, draw=black, fill=white, inner sep=0pt, minimum width=5pt];
        \path (0,0) node (p1) {} ;
        \path (3,0) node (p2) {} ;
        \path (1.5,1) node (p3) {} ;
        
         \path (0.6,1.8) node (p11) {} ;
        \path (3.6,1.8) node (p22) {} ;
        \path (2.1,2.8) node (p33) {} ;
        
        \draw (p1)  --  (p2);
         \draw (p1)  --  (p3);
        \draw (p3)  --  (p2);
        \draw (p11)  --  (p22);
         \draw (p11)  --  (p33);
        \draw (p33)  --  (p22);
        \draw (p1)  --  (p11);
         \draw (p2)  --  (p22);
        \draw (p3)  --  (p33);
        
        \node [draw=white, fill=white] (a) at (1.5,-0.7) {(b)};
      \end{tikzpicture}
\end{center}
\vspace{-0.3cm}
\caption{(a) A realisation of the triangular prism graph as a bar-joint framework in $\mathbb{R}^2$ with  extrusion symmetry. This framework is continuously flexible, as illustrated in (b). }
\label{fig:ex1}
\end{figure}

This notion of extrusion symmetry is motivated by applications in Computer-Aided Design. The 3D version of extruding a framework is a  widely used technique in CAD for generating a 3D model from an initial 2D sketch, and hence it is important to understand the flexibility of extruded frameworks. (The framework in Figure~\ref{fig:ex1}(a) can be viewed as an extruded triangle where the extrusion direction is in the plane of the triangle.) Since CAD structures in 3D often include both points and planes, as well as a mix of geometric constraints between these objects (point-point distance constraints, point-plane distance constraints, and plane-plane angle constraints) \cite{FGO}, we will develop our mobility analysis of extrusion-symmetric frameworks for both bar-joint frameworks and the more general point-hyperplane frameworks \cite{EJNSTW} in Euclidean $d$-space  (for an arbitrary $d$).

By treating the extrusion as a symmetry we can exploit the methods that have proved useful for analysing frameworks with point group symmetry. In particular, using group representation theory,  the rigidity matrices of the extruded frameworks can be transformed into a block-diagonalised form, where each block matrix corresponds  to an irreducible representation of the group.  Based on this block-decomposition, methods from character theory may be applied to derive symmetry-adapted counting rules for detecting both infinitesimal motions and self-stresses in extrusion-symmetric frameworks that are not detectable via the standard count. Although in many cases it is also possible to identify the infinitesimal (and finite) flexibility of an extruded framework using a direct parametrisation of the extrusion directions \cite{FGO},  the symmetry-adapted counting rules also provide  information on self-stresses of the framework. Note that  for applications in CAD and geometric constraint solving, for example, it is important to identify which constraints are  dependent. 

Another key aspect of the block-decomposition is that each block matrix is a symmetry-adapted rigidity matrix of a much smaller size than the standard non-symmetric rgidity matrix. This reduction in the size of the constraint system provides a signifcant computational advantage for analysing extruded frameworks.  

Finally, the results we establish in Section~\ref{sec:finiteflex} for detecting finite flexes  in frameworks which have symmetry has led us to develop a numerical algorithm which detects both mobility and local redundancy in a variety (but not all) frameworks including those with point group symmetry (see Section~\ref{sec:linpush}). The identification of constraints that are  locally redundant, in the sense of Definition~\ref{localred}, is a crucial problem in CAD, especially when the framework is under-constrained,  so this information can be conveyed to a designer. The fact that extrusion generates a framework with symmetry shows that the flexibility and local redundancy for extruded frameworks is also detected by our numerical algorithm.

The paper is organised as follows. In Sections~\ref{sec:intro1} and \ref{sec:general} we first establish the basic definitions of $t$-fold extrusion-symmetric bar-joint and point-hyperplane frameworks, respectively. Moreover, we show in these sections that the rigidity matrix of a $t$-fold extrusion-symmetric bar-joint framework (point-hyperplane framework, respectively) can be transformed into a block-diagonal form, with the caveat that if a point-hyperplane framework contains a hyperplane that lies along an extrusion direction, then the framework may have to be suitably pinned for the block-diagonalisation to apply. These results are obtained by adapting the method in   \cite{BS2} and showing that  the corresponding rigidity matrices intertwine two particular representations of the group describing the extrusion symmetry.  The block-diagonalisation of the rigidity matrix %which is implied by the extrusion symmetry 
is useful in numerical algorithms which require the inversion of the rigidity matrix since the size of each block is typically less than about half the size of the original rigidity matrix.

In Section~\ref{sec:fowler-guest} we then show how the block-decomposition of the rigidity matrix of a $t$-fold extrusion symmetric framework can be exploited to derive Fowler-Guest type character counts (see \cite{FGsymmax,OP,BS2}, for example) that give insights into the (infinitesimal) rigidity and the self-stresses of the framework. 
For the sake of clarity of the presentation, we first develop this theory for bar-joint frameworks and then generalise it to point-hyperplane frameworks. Throughout the section, we illustrate our results via some basic examples.

Section~\ref{sec:general} and Subsection~\ref{sec:countph} which describe the application of our theory to point-hyperplane frameworks are considerably more complicated than the sections which deal only with bar-joint frameworks. The reader who is interested only in bar-joint frameworks could skip these parts. However,  we have found that configurations which complicate the use of symmetry with hyperplanes,  such as when a self-symmetric hyperplane contains an extrusion direction,  do occur frequently in practical applications so this aspect of the theory is important.

Finally,  in Section~\ref{sec:finiteflex} we show that the methods which have been developed to detect \emph{finite} mobility in frameworks with point group symmetry \cite{BSfinite} may also be used to give a sufficient condition for finite  mobility in frameworks which do not have an obvious symmetry.  We 
give an Asimov-Roth type result \cite{AR} to detect finite flexbility in bar-joint or point-hyperplane frameworks whose configurations of points and hyperplanes are `regular points' of the  measurement map, restricted to a smooth manifold (such as the manifold of $t$-fold extrusion-symmetric configurations).  Based on the idea of displacing a `regular'  framework  by a linear push along a smooth manifold  
to obtain a finite motion,  in Section~\ref{sec:linpush} we give a numerical algorithm for detecting finite flexibility in a point-hyperplane framework even in the absence of any known symmetry. The algorithm requires only that we can compute the null-space of the rigidity matrix at an arbitrary point in the configuration space.  Since  there is a correspondence between finite motions and  redundant constraints,  this algorithm is also useful for detecting locally redundant constraints in systems of geometric constraint equations.

\section{Bar-joint frameworks}\label{sec:intro1}
\subsection{Rigidity of bar-joint frameworks}\label{sec:intro}

A $d$-dimensional \emph{(bar-joint) framework} $(G,p)$ consists of a simple graph $G=(V,E)$ and a map $p:V\to \mathbb{R}^d$.
 We  think of a framework as  a straight-line realisation of  $G$, where each edge of $G$ models a rigid bar of fixed length and each vertex of $G$ models a joint in Euclidean $d$-space that allows bending in any direction. It is a classical topic in discrete geometry (as well as in diverse areas of application, such as engineering, CAD, robotics or molecular dynamics) to study the rigidity and flexibility of frameworks, in the following sense.
 
Two frameworks $(G, p)$ and $(G, q)$ are said to be \emph{equivalent} if
$\|p(i) - p(j)\| = \|q(i) - q(j)\|$ for all $\{i, j\}\in E$.
Frameworks $(G, p)$ and $(G, q)$ are said to be \emph{congruent} if $p$ and $q$ are congruent.
A \emph{finite motion} of a framework $(G,p)$ is a continuous one-parameter family of frameworks
$(G, p_{T})$ with $p_0 = p$ and $(G, p_{T})$ equivalent to $(G, p)$ for all $T\in [0,1)$. A finite motion is \emph{trivial} if all $(G, p_{T})$ are congruent to $(G, p)$. A non-trivial finite motion is also called a \emph{finite flex}. A framework is \emph{flexible} if it has a finite flex, and \emph{rigid} otherwise. Since it is in general difficult to analyse a framework for rigidity, it is common to linearise the problem by differentiating the edge constraints, which leads to the concept of infinitesimal rigidity (which is equivalent to static rigidity). We briefly introduce the key definitions (see \cite{SW1,WW}, for example, for details).

An \emph{infinitesimal motion} of a framework $(G,p)$ in $\mathbb{R}^d$
is a function $\dot p: V\to \mathbb{R}^{d}$ such that
\begin{equation}
\label{infinmotioneq}
\langle p_i-p_j, \dot p_i-\dot p_j\rangle =0 \quad\textrm{ for all } \{i,j\} \in E\textrm{,}
\end{equation}
where $p_i=p(i)$ and $\dot p_i=\dot p(i)$ for each $i$.

An infinitesimal motion $\dot p$ of $(G,p)$ is a \emph{trivial infinitesimal motion}
if there exists a skew-symmetric matrix $S$
and a vector $b$ such that $\dot p_i=Sp_i+b$ for all $i\in V$.
Otherwise $\dot p$ is called an \emph{infinitesimal flex} (or \emph{non-trivial infinitesimal motion}) of $(G,p)$.
 $(G,p)$ is \emph{infinitesimally rigid} if every infinitesimal motion of $(G,p)$ is trivial.
Otherwise $(G,p)$ is said to be \emph{infinitesimally flexible}.

The matrix corresponding to the linear system in (\ref{infinmotioneq}) is known as the \emph{rigidity matrix} of $(G,p)$, and is denoted by $R(G,p)$. In other words, $R(G,p)$ is the $|E|\times d|V|$ matrix whose rows correspond to the edges of $G$, and the entries in row $e = \{i, j\}$ are zero except possibly in the column $d$-tuples for $p_i$ and $p_j$,
where the entries are the coordinates of $p_i-p_j$ and $p_j - p_i$, respectively.

It is well known that if $p$ affinely spans at least $\mathbb{R}^{d-1}$, then the space of trivial infinitesimal motions of $(G,p)$ has dimension $\binom{d+1}{2}$, and hence $(G,p)$ is infinitesimally rigid if and only if $\textrm{rank } R(G,p)=d|V|-\binom{d+1}{2}$. Thus, infinitesimal rigidity of a framework $(G,p)$ can easily be checked by computing the rank of $R(G,p)$.

If a framework is infinitesimally rigid then it is rigid. While the converse is not true in general, it was shown in \cite{AR} that if $p$ is a \emph{regular} point, i.e. $\textrm{rank } R(G,p)= \textrm{max}\{R(G,q)|\, q\in \mathbb{R}^{d|V|}\}$, then the existence of an infinitesimal flex implies the existence of a finite flex. Thus, infinitesimal rigidity and rigidity are in fact equivalent for all regular (and hence almost all) $d$-dimensional realisations $(G,p)$ of a graph $G$ as a bar-joint framework.

%%%%%%%%%%%%%%%
\subsection{Bar-joint frameworks with extrusion symmetry}\label{subsec:defex}

We denote graphs by $G = (V, E)$, where $V$ is the set of vertices and $E$ is the set of edges. In cases where $G$ is not clear from the context, we write $V(G)$ for $V$ and $E(G)$ for $E$. The complete graph on $n$ vertices is denoted by $K_n$. The automorphism group of a graph $G$ is denoted by $\textrm{Aut}(G)$.

The \emph{Cartesian product}  of the graphs $G_1$ and $G_2$ is the graph $G=G_1\square G_2$ with $V(G)=V(G_1)\times V(G_2)$, where two vertices $(u_1,v_1)$ and $(u_2,v_2)$ are adjacent in $G$ if and only if either $u_1=u_2$ and $\{v_1,v_2\}\in E(G_2)$, or $v_1=v_2$ and $\{u_1,u_2\}\in E(G_1)$. It is well known that the Cartesian product is associative and commutative (as an operation on isomorphism classes of graphs). For the $t$-fold Cartesian product of a graph $G$ with itself we use the notation $G^{\square t}$.

Let $H$ be a graph and let $G=H\square K_2^{\square t}$ for some fixed $t\geq 1$. Let $V(H)=\{v_1,\ldots, v_k\}$ and let $V(K_2)=\{0,1\}$. Then each vertex of $G$ is of the form $(v_j,\be)$, where $j\in \{1,\ldots, k\}$ and $\be\in \{0,1\}^t$. 

Now, let $\mathbb{Z}_2=\{0,1\}$ be the additive group of order 2 and $\mathbb{Z}_2^t$ be the $t$-fold direct product  of $\mathbb{Z}_2$ with itself. So any element of $\mathbb{Z}_2^t$ is of the form $\gamma=(x_1,\ldots, x_t)$ with $x_i\in \{0,1\}$ for each $i$. Then, for $G=H\square K_2^{\square t}$, there clearly exists a group homomorphism $\theta: \mathbb{Z}_2^t\to \textrm{Aut}(G)$ given by 
$\theta(\gamma)=\alpha_{\gamma}$, where $$\alpha_{\gamma}(v_j,\be)=(v_j,\be+\gamma)$$ 
for all $(v_j,\be)\in V(G)$ (with the addition in $\be+\gamma$ taken modulo $2$). 
We call the group action $\theta$ the \emph{extrusion action} on $G$.

\begin{defn}\label{def:extrdef}
For $G=H\square K_2^{\square t}$ with extrusion action $\theta$, a bar-joint framework $(G,p)$ in $\mathbb{R}^d$ is said to have \emph{$t$-fold extrusion symmetry} (or simply \emph{extrusion symmetry} if $t=1$) if
 there exist $\tau_1,\ldots, \tau_t \in \mathbb{R}^d\setminus \{0\}$ such that 
 for every vertex $(v_j,\be)$ of $G$, and every $\gamma\in \mathbb{Z}_2^t$, we have 
 $$p_{\theta(\gamma)((v_j,\be))}=p_{(v_j,\be)}+ \sum_{i\in X} \tau_i - \sum_{g\in Y} \tau_g,$$ where $X$ is the set of indices in $\{1,\ldots, t\}$ for which $\be$ has an entry of $0$ and $\gamma$ has an entry of $1$, and  $Y$ is the set of indices in $\{1,\ldots, t\}$ for which both $\be$ and $\gamma$ have an entry of $1$. The vector  $\tau_{\gamma}((v_j,\bm e))=\sum_{i\in X} \tau_i - \sum_{g\in Y} \tau_g$ is called the \emph{extrusion direction of $(v_j,\be)$ induced by $\gamma$}.
 \end{defn}

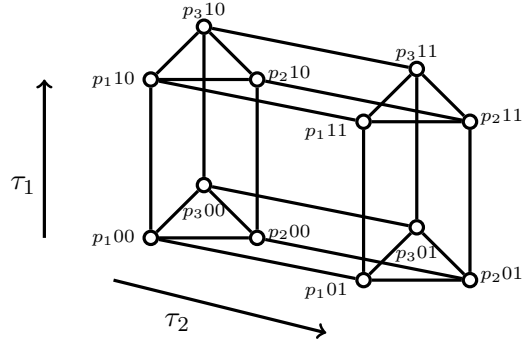
\begin{figure}[htp]
\begin{center}
\begin{tikzpicture}[very thick,scale=0.7]
\tikzstyle{every node}=[circle, draw=black, fill=white, inner sep=0pt, minimum width=5pt];
        \node [draw=white, fill=white] (a) at (-0.7,0) {\tiny $p_100$};
        \node [draw=white, fill=white] (a) at (-0.7,3) {\tiny $p_110$};
        \node [draw=white, fill=white] (a) at (2.6,0.1) {\tiny $p_200$};
        \node [draw=white, fill=white] (a) at (2.6,3.1) {\tiny $p_210$};
           \node [draw=white, fill=white] (a) at (1,0.5) {\tiny $p_300$};
        \node [draw=white, fill=white] (a) at (1,4.3) {\tiny $p_310$};
        
           \node [draw=white, fill=white] (a) at (3.3,-1) {\tiny $p_101$};
        \node [draw=white, fill=white] (a) at (3.3,2) {\tiny $p_111$};
        \node [draw=white, fill=white] (a) at (6.6,-0.7) {\tiny $p_201$};
        \node [draw=white, fill=white] (a) at (6.6,2.3) {\tiny $p_211$};
           \node [draw=white, fill=white] (a) at (5,-0.3) {\tiny $p_301$};
        \node [draw=white, fill=white] (a) at (5,3.5) {\tiny $p_311$};

        \path (0,0) node (p1) {} ;
        \path (2,0) node (p2) {} ;
        \path (1,1) node (p3) {} ;
        
         \path (0,3) node (p11) {} ;
        \path (2,3) node (p22) {} ;
        \path (1,4) node (p33) {} ;
        
        \draw (p1)  --  (p2);
         \draw (p1)  --  (p3);
        \draw (p3)  --  (p2);
        \draw (p11)  --  (p22);
         \draw (p11)  --  (p33);
        \draw (p33)  --  (p22);
        \draw (p1)  --  (p11);
         \draw (p2)  --  (p22);
        \draw (p3)  --  (p33);
        
                \path (4,-0.8) node (p1r) {} ;
        \path (6,-0.8) node (p2r) {} ;
        \path (5,0.2) node (p3r) {} ;
        
         \path (4,2.2) node (p11r) {} ;
        \path (6,2.2) node (p22r) {} ;
        \path (5,3.2) node (p33r) {} ;
        
        \draw (p1r)  --  (p2r);
         \draw (p1r)  --  (p3r);
        \draw (p3r)  --  (p2r);
        \draw (p11r)  --  (p22r);
         \draw (p11r)  --  (p33r);
        \draw (p33r)  --  (p22r);
        \draw (p1r)  --  (p11r);
         \draw (p2r)  --  (p22r);
        \draw (p3r)  --  (p33r);
        
        \draw (p1r)  --  (p1);
         \draw (p3r)  --  (p3);
        \draw (p2r)  --  (p2);
        \draw (p11r)  --  (p11);
         \draw (p33)  --  (p33r);
        \draw (p22)  --  (p22r);
       
        \draw[->](-2,0)--(-2,3);
         \node [draw=white, fill=white] (a) at (-2.4,1) {$\tau_1$};
        \draw[->](-0.7,-0.8)--(3.3,-1.8);
        \node [draw=white, fill=white] (a) at (0.5,-1.6) {$\tau_2$};
       % \node [draw=white, fill=white] (a) at (1,-0.7) {(a)};
      \end{tikzpicture}
     \end{center}
\caption{A realisation of the graph $K_3\square K_2^{\square 2}$ as a framework with $2$-fold extrusion symmetry in $\mathbb{R}^2$, where the point $p((v_j,x,y))$ is denoted by $p_jxy$. The extrusion directions of $(G,p)$ are $\tau_1$ and $\tau_2$.}
\label{fig:ex2}
\end{figure}

Note that if $(G,p)$ has $t$-fold extrusion symmetry, then  we have  $$p(v_j,\be)=p(v_j,\mathbf{0})+ \mathbf{T} \cdot \be^T,$$
where  $\mathbf{T}$ is the $d\times t$ matrix whose $i$-th column is $\tau_i$, and $\be^T$ is the $t$-dimensional column vector corresponding to the $0$-$1$ word $\be$. The vectors $\tau_1,\ldots ,\tau_t$ are called the \emph{extrusion directions} of $(G,p)$.

An example of a $2$-dimensional framework with $2$-fold extrusion symmetry is shown in Figure~\ref{fig:ex2}.

\subsection{Block-decomposing the rigidity matrix}\label{subsec:blockdecom}

In this section we show that the rigidity matrix of a  framework with $t$-fold extrusion symmetry  can be transformed into a block-diagonalised form using techniques from group representation theory. This block-decomposition is obtained in a similar way as the one for frameworks with point group symmetry, as described in \cite{KG2,BS2,OP}.
We need the following basic definitions.

For a group $\Gamma$ and a linear space $X$, a homomorphism $\rho:\Gamma\to GL(X)$ is called a \emph{(linear) representation} of $\Gamma$. The space $X$ is called the \emph{representation space} of $\rho$, and two linear representations are considered equivalent if they are similar.

For a linear representation $\rho:\Gamma\to GL(X)$, a subspace $U\subseteq X$ is called a \emph{$\rho$-invariant subspace} if  $\rho(\gamma)(U)=U$ for all $\gamma\in \Gamma$.
A linear representation that has no proper non-zero $\rho$-invariant subspaces is called an \emph{irreducible representation}.

Given two linear representations, $\rho_1:\Gamma\to GL(X)$ and   $\rho_2:\Gamma\to GL(Y)$ with representation spaces $X$ and $Y$, a linear map  $T:X \to Y$ is said to be a \emph{$\Gamma$-linear map of $\rho_1$ and $\rho_2$} if $T\circ \rho_1(\gamma)=\rho_2(\gamma)\circ T$ for all $\gamma\in \Gamma$. The vector space of all $\Gamma$-linear maps of $\rho_1$ and $\rho_2$ is denoted by ${\rm Hom}_{\Gamma}(\rho_1,\rho_2)$.

In the following we will define two particular linear representations (called the `internal' and `external' representation) of the group $\mathbb{Z}_2^t$ of a framework $(G,p)$ with $t$-fold extrusion symmetry,  and we will show that the rigidity matrix of $(G,p)$ is a $\mathbb{Z}_2^t$-linear map of the internal and external representation (see Theorem~\ref{thm:block}).

Let $G=H\square K_2^{\square t}=(V,E)$ and let $(G,p)$ be a framework in $\mathbb{R}^d$ with $t$-fold extrusion symmetry.
We let $\Gamma=\mathbb{Z}_2^t$, and for simplicity we denote $\theta(\gamma)(i)$ by $\gamma i$ for $\gamma\in \Gamma$ and $i\in V$. Let $P_V:\Gamma\to GL(\mathbb{R}^{|V|})$ be the linear representation of $\Gamma$ defined by $P_V(\gamma)=[\delta_{i,\gamma j}]_{i,j},$ where $\delta$ denotes the Kronecker delta symbol. That is, $P_V(\gamma)$ is the permutation matrix of the permutation of $V$ induced by $\theta(\gamma)$.  Similarly, we let $P_E:\Gamma\to GL(\mathbb{R}^{|E|})$ be the linear representation of $\Gamma$ consisting of the permutation matrices of the permutations of $E$ induced by $\theta(\gamma)$.

The \emph{internal representation} of $\Gamma$ (with respect to $G=H\square K_2^{\square t}$) is the linear representation $P'_E:\Gamma\to GL(\mathbb{R}^{|E|})$, which is obtained from $P_E$ as follows. Let $\gamma\in \Gamma$, and let $e=\{i,j\}$ be an edge of $G$ with $i$ and $j$ of the form $i=(v,\bm e)$ and $j=(v,\bm e')$, with $v\in V(H)$ and $\be, \bm e'\in \{0,1\}^t$. Moreover, let $\gamma e=f$. Note that $\be$ and $\bm e'$ must differ in exactly one entry, say the $h$-th entry. We replace the $1$ in the $(e,f)$-th and the $(f,e)$-th entry of $P_E(\gamma)$ by $-1$ if and only if $\gamma$ has a $1$ in the $h$-th entry. We obtain $P'_E(\gamma)$ by applying this procedure to every edge of $G$ of the form $\{i,j\}$, $i=(v,\bm e)$ and $j=(v,\bm e')$.

Note that for $\gamma\neq \textrm{id}=(0,\ldots, 0)$, an edge $e=\{i,j\}$ of $G$ makes a contribution of $1$ to the diagonal of $P_E(\gamma)$ if and only if $e$ is \emph{fixed} by $\gamma$, i.e. if $\gamma i=j$ and $\gamma j=i$. Moreover, $e=\{i,j\}$ is fixed by $\gamma$ if and only if $i$ and $j$ are of the form $i=(v,\bm e)$ and $j=(v,\bm e')$, with $v\in V(H)$ and $\be, \bm e'\in \{0,1\}^t$ differing only in the $h$-th entry, and $\gamma$ has a $1$ in the $h$-th entry and zeros everywhere else. So in particular, for each $\gamma\in \Gamma$, $\gamma\neq \textrm{id}$,  each $1$ in the diagonal of the matrix $P_E(\gamma)$ is replaced by a $-1$ in $P'_E$. It is straightforward to check that $P'_E$ is indeed a linear representation.

The \emph{external representation} of $\Gamma$ (with respect to $G=H\square K_2^{\square t}$) is the linear representation $$P_V\otimes I_d:\Gamma\to GL(\mathbb{R}^{d|V|}),$$ i.e., for any $\gamma\in \Gamma$, the matrix $(P_V\otimes I_d)(\gamma)$ is the Kronecker product of the permutation matrix $P_V(\gamma)$ and the $d \times d$ identity matrix $I_d$. So $(P_V\otimes I_d)(\gamma)$ is obtained from $P_V(\gamma)$ be replacing each zero with the $d\times d$ zero matrix and each $1$ with $I_d$.

\begin{thm}
\label{thm:block}
Let $G=H\square K_2^{\square t}$ and $(G,p)$ be  a framework in $\mathbb{R}^d$ with $t$-fold extrusion symmetry (and extrusion directions $\tau_1,\ldots, \tau_t$) and let $\theta:\Gamma\to \textrm{Aut}(G)$ be the extrusion action. 
Then $$R(G,p)\in {\rm Hom}_{\Gamma}(P_V\otimes I_d,P'_E).$$
\end{thm}
\begin{proof} Suppose we have $R(G,p)\dot p=s$. Then we need to show that $R(G,p)(P_V\otimes I_d)(\gamma) \dot p = P'_E(\gamma)s$.

Fix $\gamma\in \Gamma$ and let $e=\{i,j\}\in E$. Suppose that $\gamma i =q$ and $\gamma j =l$, and hence $P_E(\gamma)(\{i,j\})=\{q,l\}$. By definition of $P'_E$, for the $\{q,l\}$-th component $(P'_E(\gamma)s)_{\{q,l\}}$ of the $|E|\times 1$ column vector $P'_E(\gamma)s$, we have:
$$(P'_E(\gamma)s)_{\{q,l\}}
= 
\begin{cases}
   -s_{\{i,j\}} ,& \text{if } i=(v,\be) \textrm{ and } j=(v,\bm e'), \be_h\neq \bm e'_h \textrm{ and } \gamma_h=1\\
    s_{\{i,j\}},              & \text{otherwise}
\end{cases}
$$
Next we consider $(R(G,p)(P_V\otimes I_d)(\gamma) \dot p)_{\{q,l\}}$ and show that it is equal to $(P'_E(\gamma)s)_{\{q,l\}}$.

Note that $(R(G,p) \dot p)_{\{q,l\}}= \langle(p_q-p_l),\dot p_q\rangle+\langle(p_l-p_q),\dot p_l\rangle$. Thus, by the definition of $P_V\otimes I_d$, we have
$$(R(G,p)(P_V\otimes I_d)(\gamma) \dot p)_{\{q,l\}}= \langle(p_q-p_l),\dot p_i\rangle+\langle(p_l-p_q),\dot p_j\rangle.$$

Suppose first that  $i=(v,\be)$  and $j=(v,\bm e')$, where $\be_h\neq \bm e'_h$ and $\gamma_h=1$. Without loss of generality, suppose that $\be_h=0$ and  $\bm e'_h=1$. Let $p_q=p_i+\tau$ and $p_l=p_j+\tau'$, where $\tau=\tau_{\gamma}(i)$ and $\tau'=\tau_{\gamma}(j)$ are the respective extrusion directions of $i$ and $j$ induced by $\gamma$. (Recall Section~\ref{subsec:defex}.) Note that by definition of $i$ and $j$, we have $\tau- \tau'=2 \tau_h$, where $\tau_h$ is the $h$-th extrusion direction. Moreover, we have $p_j=p_i+\tau_h$. It follows that
 $$p_q-p_l=p_i-p_j+2\tau_h=p_j-p_i.$$
Therefore, 
$$(R(G,p)(P_V\otimes I_d)(\gamma) \dot p)_{\{q,l\}}=\langle(p_j-p_i),\dot p_i\rangle+\langle(p_i-p_j),\dot p_j\rangle= -s_{\{i,j\}}$$
as claimed.

Suppose next that $i$ and $j$ are of the form $i=(v,\be)$ and $j=(w,\be)$, with $v\neq w$, or $i=(v,\be)$ and $j=(v,\bm e')$, with $\be_h\neq \bm e'_h$ and $\gamma_h=0$.  Then, by the $t$-fold extrusion symmetry of $(G,p)$, in both cases we have $\tau_{\gamma}(i)=\tau_{\gamma}(j)$. So if we denote $\tau=\tau_{\gamma}(i)$, then 
 $p_q=p_i+\tau$ and $p_l=p_j+\tau$. Therefore, we have $p_q-p_l=p_i-p_j$ and $p_l-p_q=p_j-p_i$, and hence
$$(R(G,p)(P_V\otimes I_d)(\gamma) \dot p)_{\{q,l\}}=\langle(p_i-p_j),\dot p_i\rangle+\langle(p_j-p_i)\dot p_j\rangle= s_{\{i,j\}}$$
as claimed. This completes the proof.
\end{proof}

It is an immediate consequence of Theorem~\ref{thm:block}  and Schur's lemma (see \cite{Serre}, for example) that we can block-decompose the rigidity matrix of $(G,p)$, provided that $(G,p)$ has $t$-fold extrusion symmetry. We introduce some basic notation before we state the result.

Recall that any element of $\Gamma=\mathbb{Z}_2^t$ is of the form $\gamma=(x_1,\ldots, x_t)$ with $x_i\in \{0,1\}^t$ for each $i$. Let $r=|\Gamma|=2^t$. It is an elementary fact from group representation theory that $\Gamma$ has $r$ non-equivalent irreducible representations
which we denote by  $\{\rho_{\gamma}\colon {\gamma}\in \Gamma\}$ (see e.g. \cite{Serre}).
For each ${\gamma}\in \Gamma$, $\rho_{\gamma}$ is defined by
\begin{align} \label{eq:abelian_rho}
\rho_{{\gamma}}:\Gamma&\rightarrow \{1,-1\}  \nonumber\\
{\gamma'}&\mapsto (-1)^{x'_1x_1+x'_2x_2 + \ldots + x'_tx_t},
\end{align}
where $\gamma'=(x'_1,\ldots, x'_t)$.

The irreducible representation that assigns $1$ to each element of $\Gamma$ is called the \emph{trivial representation}.

\begin{cor} \label{cor:block} Let $(G,p)$ be  a framework in $\mathbb{R}^d$ with $t$-fold extrusion symmetry and let $\theta:\Gamma\to \textrm{Aut}(G)$ be the extrusion action. Then there exist non-singular matrices $A$ and $B$ such that $B^T R(G,p) A$ is block-decomposed as
\begin{equation}
\label{rigblocks}
B^{\top}R(G,p)A:=\widetilde{R}(G,p)
=\left(\begin{array}{ccc}\widetilde{R}_{0}(G,p) & & \mathbf{0}\\ & \ddots & \\\mathbf{0} &  &
\widetilde{R}_{r-1}(G,p) \end{array}\right)\textrm{,}
\end{equation}
where the submatrix block $\widetilde{R}_{\kappa}(G,p)$ corresponds to the $\kappa$-th irreducible representation of  $\Gamma$.
\end{cor}

The  symmetry-adapted coordinate systems that yield the block-decomposition of the rigidity matrix can be obtained as follows.
Recall that every linear representation of $\Gamma$ can be written uniquely, up to equivalency of the direct summands, as a direct sum of the irreducible linear representations of $\Gamma$. So we have
\begin{equation}
\label{irrep}
(P_V\otimes I_d)= \lambda_{0}\rho_{0}\oplus\ldots\oplus \lambda_{r-1}\rho_{r-1} \textrm{, where } \lambda_{0},\ldots,\lambda_{r-1}\in \mathbb{N}\cup {\{0\}} \textrm{.}
\end{equation}
For each $\kappa=0,\ldots,r-1$, there exist $\lambda_{\kappa}$ subspaces $\big(X^{(\rho_{\kappa})}\big)_{1},\ldots, \big(X^{(\rho_{\kappa})}\big)_{\lambda_{\kappa}}$ of the $\mathbb{R}$-vector space $\mathbb{R}^{d|V|}$ which correspond to the $\lambda_{\kappa}$ direct summands in (\ref{irrep}), so that
\begin{equation}
\label{dirsumofvs}
\mathbb{R}^{d|V|}=X^{(\rho_{0})}\oplus \ldots \oplus X^{(\rho_{r-1})} \textrm{,}
\end{equation}
where
\begin{equation}
\label{dirsumofv2}
X^{(\rho_{\kappa})}= \big(X^{(\rho_{\kappa})}\big)_{1}\oplus \ldots \oplus \big(X^{(\rho_{\kappa})}\big)_{\lambda_{\kappa}} \textrm{.}
\end{equation}

Similarly, for the internal representation $P'_E$ of $\Gamma$, we have the direct sum decomposition
\begin{equation}
\label{irrepHi}
P'_E= \mu_{0}\rho_{0}\oplus\ldots\oplus \mu_{r-1}\rho_{r-1} \textrm{, where } \mu_{0},\ldots,\mu_{r-1}\in \mathbb{N}\cup {\{0\}} \textrm{.}
\end{equation}
For each $\kappa=0,\ldots, r-1$, there exist $\mu_{\kappa}$ subspaces $\big(Y^{(\rho_{\kappa})}\big)_{1},\ldots, \big(Y^{(\rho_{\kappa})}\big)_{\mu_{\kappa}}$ of the $\mathbb{R}$-vector space $\mathbb{R}^{|E|}$ which correspond to the $\mu_{\kappa}$ direct summands in (\ref{irrepHi}), so that
\begin{equation}
\label{dirsumofvsHi}
\mathbb{R}^{|E|}=Y^{(\rho_{0})}\oplus \ldots \oplus Y^{(\rho_{r-1})} \textrm{,}
\end{equation}
where
\begin{equation}
\label{dirsumofv2Hi}
Y^{(\rho_{\kappa})}= \big(Y^{(\rho_{\kappa})}\big)_{1}\oplus \ldots \oplus \big(Y^{(\rho_{\kappa})}\big)_{\mu_{\kappa}} \textrm{.}
\end{equation}
If we choose bases  $\big(A^{(\rho_{\kappa})}\big)_{1},\ldots, \big(A^{(\rho_{\kappa})}\big)_{\lambda_{\kappa}}$  for the subspaces in (\ref{dirsumofv2}) and we also choose bases  $\big(B^{(\rho_{\kappa})}\big)_{1},\ldots, \big(B^{(\rho_{\kappa})}\big)_{\mu_{\kappa}}$ for the subspaces in  (\ref{dirsumofv2Hi}), then   $\bigcup_{\kappa=0}^{r-1}\bigcup_{i=1}^{\lambda_\kappa}\big(A^{(\rho_{\kappa})}\big)_{i}$ and  $\bigcup_{\kappa=0}^{r-1}\bigcup_{i=1}^{\mu_\kappa}\big(B^{(\rho_{\kappa})}\big)_{i}$ are symmetry-adapted bases with respect to which  the rigidity matrix is block-decomposed as shown in (\ref{rigblocks}). 

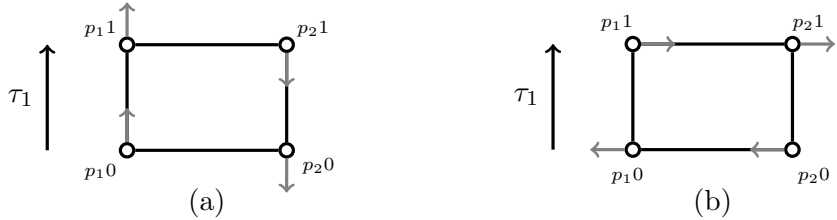
\begin{figure}[htp]
\begin{center}
 \begin{tikzpicture}[very thick,scale=0.7]
\tikzstyle{every node}=[circle, draw=black, fill=white, inner sep=0pt, minimum width=5pt];
             \path (0,0) node (p1) {} ;
        \path (3,0) node (p2) {} ;
       
         \path (0,2) node (p1o) {} ;
        \path (3,2) node (p2o) {} ;

          \node [draw=white, fill=white] (a) at (-0.5,-0.4) {\tiny $p_10$};
         \node [draw=white, fill=white] (a) at (3.6,-0.3) {\tiny $p_20$};
        \node [draw=white, fill=white] (a) at (-0.5,2.3) {\tiny $p_11$};
         \node [draw=white, fill=white] (a) at (3.5,2.3) {\tiny $p_21$};
        
        \draw (p1)  --  (p2);
         \draw (p1o)  --  (p2o);
      
        \draw (p1)  --  (p1o);
         \draw (p2)  --  (p2o);
         
               \draw[->](-1.5,0)--(-1.5,2);
         \node [draw=white, fill=white] (a) at (-2,1) {$\tau_1$};
         
              \draw[gray,->] (p1)  --  (0,0.8);
           \draw[gray,->] (p2)  --  (3,-0.8);
          \draw[gray,->] (p1o)  --  (0,2.8);
           \draw[gray,->] (p2o)  --  (3,1.2);
         
        \node [draw=white, fill=white] (a) at (1.5,-1) {(a)};
      \end{tikzpicture}
            \hspace{2cm}
\begin{tikzpicture}[very thick,scale=0.7]
\tikzstyle{every node}=[circle, draw=black, fill=white, inner sep=0pt, minimum width=5pt];
      
              \node [draw=white, fill=white] (a) at (-0.1,-0.5) {\tiny $p_10$};
         \node [draw=white, fill=white] (a) at (3.4,-0.5) {\tiny $p_20$};
        \node [draw=white, fill=white] (a) at (-0.3,2.4) {\tiny $p_11$};
         \node [draw=white, fill=white] (a) at (3.3,2.4) {\tiny $p_21$};
      
        \path (0,0) node (p1) {} ;
        \path (3,0) node (p2) {} ;
       
         \path (0,2) node (p1o) {} ;
        \path (3,2) node (p2o) {} ;

        \draw (p1)  --  (p2);
         \draw (p1o)  --  (p2o);
      
        \draw (p1)  --  (p1o);
         \draw (p2)  --  (p2o);
         
          \draw[gray,->] (p1)  --  (-0.8,0);
           \draw[gray,->] (p2)  --  (2.2,0);
          \draw[gray,->] (p1o)  --  (0.8,2);
           \draw[gray,->] (p2o)  --  (3.8,2);
         
         \draw[->](-1.5,0)--(-1.5,2);
         \node [draw=white, fill=white] (a) at (-2,1) {$\tau_1$};
        \node [draw=white, fill=white] (a) at (1.5,-1) {(b)};
      \end{tikzpicture}
\end{center}
\caption{(a) A plane framework  $(K_2\square K_2,p)$ with ($1$-fold) extrusion symmetry in the direction of $\tau_1$. (a) shows a fully-symmetric (i.e. $\rho_{0}$-symmetric) infinitesimal motion (the velocity vectors are preserved under extrusion) and (b) shows an anti-symmetric (i.e. $\rho_{1}$-symmetric) infinitesimal motion (the velocity vectors are reversed under extrusion). The motion in (b) may be obtained from the one in (a) by adding an infinitesimal rotation. (The framework may also be considered as a $2$-fold-extrusion symmetric framework $(K_1\square K_2^{\square 2},p)$.)}
\label{fig:simple}
\end{figure}

\begin{defn}\label{def:sy}
A vector $u\in \mathbb{R}^{d|V|}$ is \emph{symmetric with respect to the irreducible linear representation $\rho_{\gamma}$} of $\Gamma$ if $u\in X^{(\rho_{\gamma})}$. Similarly, a vector $s\in \mathbb{R}^{|E|}$ is \emph{symmetric with respect to  $\rho_{\gamma}$}  if $s\in Y^{(\rho_{\gamma})}$. A vector (in $\mathbb{R}^{d|V|}$ or $\mathbb{R}^{|E|}$) that is symmetric with respect to the trivial representation is also called \emph{fully-symmetric}.
\end{defn}

Note that the kernel of the block matrix $\widetilde{R}_{\gamma}(G,p)$ is isomorphic to the space of all infinitesimal motions of $(G,p)$ that are symmetric with respect to $\rho_{\gamma}$. Since the irreducible representations $\rho_{\gamma}$ of $\mathbb{Z}_2^t$ are all one-dimensional and assign either $1$ or $-1$ to each group element, the infinitesimal motions of symmetry $\rho_{\gamma}$ have a simple form: if $\dot p_i$ is the velocity vector at the point $p_i$, then the velocity vector of the point $p_{\gamma' i}$ is also $\dot p_i$ if $\rho_{\gamma}(\gamma')=1$ and it is $-\dot p_i$ if $\rho_{\gamma}(\gamma')=-1$.  In particular, a fully-symmetric infinitesimal motion assigns the same velocity vector to all  points that lie in the same orbit under $\mathbb{Z}^t_2$. See Figure~\ref{fig:simple} for an illustration.

\begin{rem} An infinitesimal motion of a bar-joint framework $(G,p)$ with $t$-fold extrusion symmetry, where none of the extrusion directions $\tau_1,\ldots, \tau_t$ are parallel, must either be fully-symmetric or $\rho_{\gamma}$-symmetric, where $\gamma\in \mathbb{Z}_2^t$ has exactly one non-zero entry. To see this, consider a  $\rho_{\gamma}$-symmetric infinitesimal motion $u\in \mathbb{R}^{d|V|}$, where $\gamma$ has a $1$ in entries $h$ and $h'$, so that $u((v,\mathbf{0}))\neq 0$ for some vertex $(v,\mathbf{0})$. Then $u$ assigns the same non-zero velocity vector to the vertices $(v,\mathbf{0})$ and $(v,\mathbf{e}_h +\mathbf{e}_{h'})$ and the negative of that vector to  $(v,\mathbf{e}_h)$, where $\mathbf{e}_h$ has a $1$ in the $h$-th entry and zeros elsewhere.  For $u$ to be a $\rho_{\gamma}$-symmetric infinitesimal motion of $(G,p)$, the vectors $u((v,\mathbf{0}))$ and $u((v,\mathbf{e}_h))$  must be perpendicular to $\tau_h$, and the vectors  $u((v,\mathbf{e}_h))$ and $u((v,\mathbf{e}_h +\mathbf{e}_{h'}))$  must be perpendicular to $\tau_{h'}$. However, since $\tau_h$ and $\tau_{h'}$ are not parallel,  this implies that $u((v,\mathbf{e}_h))$ is zero, a contradiction.
\end{rem}

\begin{rem}\label{rem:rotmot} A $d$-dimensional bar-joint framework $(G,p)$ with $G=H\square K_2^{\square t}$ and $t$-fold extrusion symmetry will always have a fully-symmetric infinitesimal flex, provided that $H\neq K_1$.
%the vertices of $H$ affinely span a space of dimension at least $d-1$.
The framework $(G,p)$ contains $2^t$ congruent realisations of the graph $H$, one for each element $\be\in \{0,1\}^t$. Let $H_{\bf e}$ be the subgraph of $G$ with vertex set $\{(v_j,\be)|\, v_j\in V(H)\}$.  Since $H\neq K_1$, there exists an infinitesimal rotation that assigns two different velocity vectors to two vertices of $H_\mathbf{0}$.
If we apply  this same infinitesimal rotation to each of the frameworks $(H_{\be}, p|_{V(H_{\be})})$ (as in Figure~\ref{fig:ex1maxwell}(a)), then the resulting infinitesimal motion $\dot p$ of $(G,p)$ is clearly non-trivial and fully-symmetric. 

A framework $(G,p)$ with  $t$-fold extrusion symmetry 
will also always have an infinitesimal flex that shears the framework orthogonally to an extrusion direction (as in Figure~\ref{fig:ex1maxwell}(d)), unless $H=K_1$ and $t=1$. These infinitesimal flexes are anti-symmetric with respect to the corresponding extrusion direction. However,
%if $H\neq K_1$, then one of them is
these may be related to the fully-symmetric infinitesimal flex described above by an infinitesimal rotation, as is illustrated  in Figure~\ref{fig:ex1maxwell}. 
\end{rem}

The block-decomposition of the rigidity matrix breaks up the problem of analysing the infinitesimal rigidity of a framework with $t$-fold extrusion symmetry into a number of independent sub-problems. In particular, it leads to refined Maxwell-type counts, as we will see in Section~\ref{sec:fowler-guest}. While a direct geometric analysis  of a $t$-fold extrusion-symmetric framework may be used to reveal its infinitesimal (and even finite) flexes, the refined Maxwell count will also provide information on self-stresses of the framework.

\section{Point-hyperplane frameworks}\label{sec:general}

\subsection{Rigidity of point-hyperplane frameworks} \label{sec:pthy}

Let  $G=(V_{\mathcal P}\cup V_{\mathcal H}, E)$ be a graph where the vertex set $V$ is partitioned into two non-empty sets $V_{\mathcal P}$ and $V_{\mathcal H}$.  This induces a partition of the edge set $E$  into the sets $E_{\mathcal P \mathcal P}, E_{\mathcal P \mathcal H}, E_{\mathcal H \mathcal H}$, where $E_{\mathcal P \mathcal P}$ consists of pairs of vertices in $V_{\mathcal P}$, $E_{\mathcal H \mathcal H}$ consists of pairs of vertices in $V_{\mathcal H}$, and $E_{\mathcal P \mathcal H}$ consists of pairs of vertices with one vertex in $V_{\mathcal P}$ and the other one in $V_{\mathcal H}$. We call such a graph $G$ a \emph{$\mathcal P \mathcal H$-graph}.

A \emph{point-hyperplane framework} in $\mathbb{R}^d$ is a triple $(G,p,\ell)$, where $G=(V_{\mathcal P}\cup V_{\mathcal H}, E)$ is a $\mathcal P \mathcal H$-graph and $p:V_{\mathcal P}\to \mathbb{R}^d$ and $\ell=(a,r):V_{\mathcal H}\to (\mathbb{R}^{d}\setminus \{0\})\times \mathbb{R}$  are maps.
%$\ell=(a,r):V_{\mathcal H}\to \mathbb{S}^{d-1}\times \mathbb{R}$  are maps. 
These maps $p$ and $\ell$ are interpreted as follows: each vertex $i$ in $V_{\mathcal P}$ is mapped to the point $p_i$ in $\mathbb{R}^d$ and each vertex $j$ in $V_{\mathcal H}$ is mapped to the hyperplane in $\mathbb{R}^d$ given by  $\{x\in \mathbb{R}^d: \langle a_j, x\rangle-r_j=0\}$ \cite{EJNSTW}.  (Note that in  \cite{EJNSTW}, a hyperplane is defined by $\ell=(a,r):V_{\mathcal H}\to \mathbb{S}^{d-1}\times \mathbb{R}$, i.e. the normal vector of each hyperplane is assumed to be a unit vector. However, the hyperplane is unchanged if $a$ and $r$ are multiplied by any $\alpha\in\mathbb{R}\setminus\{0\}$, and in this paper we find it more convenient to allow every point in $\mathbb{R}^{d+1}$ to represent a hyperplane $\ell = (a,r)$, where points $(a,r)$ and $(a',r')$ with $ra'=r'a$ correspond to the same hyperplane.) 
A point-hyperplane framework in $\mathbb{R}^2$ is also called a \emph{point-line framework} \cite{JO}. 

We say that the points and hyperplanes of $(G,p,\ell)$ in $\mathbb{R}^d$ affinely span at least  $\mathbb{R}^{d-1}$ if the coordinates of all of the point vertices in $p$ and the coordinates of all of the points which lie on the hyperplanes in $\ell$ affinely span at least $\mathbb{R}^{d-1}$.

In the following we will assume that the points $p(V_{\mathcal P})$ and hyperplanes $\ell(V_{\mathcal H})$ affinely span at least $\mathbb{R}^{d-1}$. Each edge in $E_{\mathcal P \mathcal P}, E_{\mathcal P \mathcal H}, E_{\mathcal H \mathcal H}$ indicates a point-point distance constraint, a point-hyperplane distance constraint, or a hyperplane-hyperplane
angle constraint, respectively. Moreover, there is a normalisation constraint for each hyperplane normal. The constraints are: 
\begin{align}
\| p_i - p_j \|&=\textrm{const} && (\{i,j\}\in E_{\mathcal P \mathcal P})  \\
\langle p_i, a_j\rangle-r_j&=\textrm{const} && (\{i,j\}\in E_{\mathcal P \mathcal H})\\
\langle a_i, a_j\rangle&=\textrm{const} && (\{i,j\}\in E_{\mathcal H \mathcal H})\\
\langle a_i, a_i\rangle&=\textrm{const} && (i\in V_{\mathcal H}).
\end{align}
This leads to the following system of first-order constraints:
\begin{align}
\langle p_i - p_j, \dot{p}_i-\dot{p}_j\rangle&=0 && (\{i,j\}\in E_{\mathcal P \mathcal P}) \label{eq:line_inf1_euc} \\
\langle p_i, \dot{a}_j\rangle+\langle \dot{p}_i, a_j\rangle -\dot{r}_j&=0 && (\{i,j\}\in E_{\mathcal P \mathcal H}) \label{eq:line_inf2_euc}\\
\langle a_i, \dot{a}_j\rangle+\langle \dot{a}_i, a_j\rangle&=0 && (\{i,j\}\in E_{\mathcal H \mathcal H}) \label{eq:line_inf3_euc}\\
\langle a_i, \dot{a}_i\rangle&=0 && (i\in V_{\mathcal H}).
\label{eq:a_inf_euc}
\end{align}
A map $(\dot{p},\dot\ell)$ is said to be an {\em
infinitesimal motion} of $(G,p,\ell)$ if it satisfies this system of
linear constraints. An infinitesimal motion is trivial if there exists a $d \times d$ skew-symmetric matrix $S$ and a vector $b \in \mathbb{R}^d$ such that $\dot{p}_i=Sp_i+b$ for all $i \in V_\mathcal{P}$ and $(\dot{a}_i,\dot{r}_i)=(Sa_i, r+\langle b,a_i \rangle )$ for all $i \in V_\mathcal{H}$.
$(G,p,\ell)$ is {\em infinitesimally
rigid} if every infinitesimal motion of $(G,p,\ell)$ is trivial (i.e.  the dimension of the space of its infinitesimal motions is equal to ${d+1\choose 2}$.

\subsubsection{Point-hyperplane frameworks with parallel constraints} \label{sec:parcon}

When we extrude a hyperplane in a direction away from the hyperplane, we create two hyperplanes that are parallel and constrained to remain parallel. It is therefore important to consider the representation of angle constraints between two hyperplanes which are parallel. 

If two hyperplanes are constrained to be parallel then the cosine of the angle between the hyperplanes is one and the first derivative of the angle constraint with respect to the hyperplane angle is zero, which means that the  corresponding row  in the rigidity matrix is always zero. The constraint $\{i,j\}$ that two hyperplanes  are parallel removes $d-1$ degrees of freedom and this should be represented by $d-1$ rows in the rigidity matrix. The edges in $ E_{\mathcal{HH}}$ of the form $\{i,j\}$ are therefore partitioned into two sets $E_{\mathcal{HH}\not\parallel}$ and $E_{\mathcal{HH}\parallel}$ according to whether the hyperplanes $i$ and $j$ are parallel or not. The constraints between  $a_i$ and $a_j$ for $\{i,j\} \in E_{\mathcal{H}\parallel}$ are defined as follows.  

For $d=2$ the constraint (3.3) becomes $\langle a_i,a_j^\perp \rangle=0$ where $\langle a_j, a_j^\perp \rangle=0$.  %a.xb.y-a.yb.x$.  
This choice gives the first-order constraint $\langle a_j^\perp, \dot{a}_i\rangle-\langle a_i^\perp, \dot{a}_j\rangle=0$ which has a non-zero derivative. 

For $d=3$ we introduce an orthogonal set of axes $\{a,u_1,u_2\}$, choose $a=a_i$ (or $a=a_j$) and write the parallel constraint as two polynomial constraints $\langle(a_i \wedge a_j), u_1 \rangle=0$ and $\langle(a_i \wedge  a_j), u_2\rangle=0$.  It is straightforward to verify that the zero set of the constraint equations does not depend on the precise choice of $\{u_1,u_2\}$.  This  gives two first-order constraints $\langle (a_j \wedge u_i),\dot{a_i}\rangle-\langle(a_i\wedge u_i),\dot{a_j}\rangle=0$,  $i=1,2$.

For $d >3$, let $a_i,u_1,\dots,u_{d-1}$ be an orthogonal basis for $\mathbb{R}^d$ and let $D_k$ be the $d\times d$ matrix with rows $a_i,a_j,u_1,\dots,u_{k-1}, u_{k+1},\dots,u_{d-1}$. The $d-1$ constraint equations implied by the parallel constraint are $\det(D_k)=0$ for $k=1,\dots,d-1$. For $d=2,3$ these constraint equations are the same as the equations given in the previous paragraph.

The corresponding rows in the rigidity matrix are easily found by differentiating these equations (using Jacobi's formula for the derivative of a determinant if required).

%\begin{rem} \label{rem:dec}
%The hyperplanes can be partitioned into equivalence classes (using graph connectivity by the parallel constraints) and a vector $a_{c(k)} \in \mathbb{R}^d$ (representing the angle $a_{c(k)}$ of all the parallel hyperplanes in the class) assigned to each equivalence class instead of to each hyperplane.  The parallel constraints then play no further part.  This alternative formulation easily extends to all dimensions,  leads to fewer constraints equations and a smaller rigidity matrix.   However in this alternative there is no longer a pairing between the angle variables $a_{c(k)}$ and the translation variables $r_k$ for $k \in V_{\mathcal{H}}$ and the proof of Theorem~\ref{thm:blockph} below is no longer valid for type 2 (PH) constraints.
The hyperplanes can be partitioned into equivalence classes using graph connectivity by the parallel constraints so that all hyperplanes in the same equivalence class are constrained to be parallel. We will use this partition of the vertices in $V_\mathcal{H}$ %which is induced by constraints between parallel hyperplanes in a framework 
when we consider finite flexibility in Section~\ref{sec:finiteflex}. We therefore define a \emph{decorated $\mathcal{PH}$-graph} to be a $\mathcal{PH}$-graph together with a partition of $V_\mathcal{H}$ into a set $V_c$ of parallel classes.  
The decoration of a $\mathcal{PH}$-graph induces the partition of $E_\mathcal{HH}=E_\mathcal{HH\parallel}\cup E_\mathcal{HH\not\parallel}$ where $\{i,j\} \in E_\mathcal{HH}$ is in $E_\mathcal{HH\parallel}$ if and only if $i$ and $j$ are in the same parallel class.  
Let $c:V_{\mathcal{H}} \to V_c$ be the map which takes $i$ to its equivalence class.

%\textcolor{red}{In section 5 we start with a decorated graph to define the measurment set. I don't think we need to define a decorated point-hyperplane framework, so this para could be removed?} A \emph{decorated point-hyperplane framework} $(G,p,\ell)$ is a point-hyperplane framework on a decorated point-hyperplane graph, $G$,  where vertices $i,j \in V_\mathcal{H}$ are in the same parallel class if and only if $\{i,j\} \in E_\mathcal{HH}$ and $a_i$ is parallel to $a_j$. It is clear that every point-hyperplane framework induces a unique decorated point-hyperplane framework and hence a unique decoration of $G$. 

\begin{rem} \label{rem:dec}
The partition of the hyperplanes into parallel classes could be used %to reduce the size of the rigidity matrix by assigning
to assign a vector $a_{c(i)} \in \mathbb{R}^d$ to represent the angle $a_{c(i)}$ of all the hyperplanes in the class
%assigned to each equivalence class 
instead of assigning a vector to each hyperplane.  The parallel constraints would then play no further part.  This alternative formulation easily extends to all dimensions,  leads to fewer constraints equations and a smaller rigidity matrix.   However in this alternative there is no longer a pairing between the angle variables $a_{c(i)}$ and the translation variables $r_i$ for $i \in V_{\mathcal{H}}$ and the proof of Theorem~\ref{thm:blockph} below is no longer valid for type 2 ($\mathcal{PH}$) constraints. Consequently, the block-diagonalisation of the rigidity matrix may no longer be valid.
\end{rem}
\subsection{The rigidity matrix}
The \emph{point-hyperplane rigidity matrix} $R(G,p,\ell)$ of  $(G,p,\ell)$,  i.e. the matrix corresponding to the  linear system given in (\ref{eq:line_inf1_euc}) -- (\ref{eq:a_inf_euc}) is a  $(|E|+|V_{\mathcal H}|+(d-2)(|V_\mathcal{H}|-|V_c|))\times (d|V_{\mathcal{P}}|+(d+1)|V_{\mathcal{H}}|)$ matrix of the following form:
\begin{displaymath}\bordermatrix{& & i & & j & &k& & l&\cr
 \{i,j\} & 0 \ldots  0 & (p_i-p_j) & 0 \ldots  0  & (p_j-p_i) & 0 \ldots  0 & 0 &0 \ldots  0 &  0 &0 \ldots  0 \cr
  & & &  & &  \vdots &  & &   &  \cr
\{i,k\} & 0 \ldots  0 &a_k & 0 \ldots  0  & 0& 0 \ldots  0 & (p_i, -1) &0 \ldots  0 &  0 &0 \ldots  0 \cr
 & & &  & &  \vdots &  & &   &  \cr
\{k,l\} & 0 \ldots  0 & 0 & 0 \ldots  0  & 0& 0 \ldots  0 & (a_{l},0) &0 \ldots  0 &  (a_{k},0) &0 \ldots  0 \cr
 & & &  & &  \vdots &  & &   &  \cr
k & 0 \ldots  0 & 0& 0 \ldots  0  & 0& 0 \ldots  0 &(a_{k},0)&0 \ldots  0 & 0 &0 \ldots  0 \cr
& & &  & &  \vdots &  & &   &  \cr
l & 0 \ldots  0 & 0 & 0 \ldots  0  & 0& 0 \ldots  0 &0 &0 \ldots  0 & (a_{l},0) &0 \ldots  0 \cr
& & &  & &  \vdots &  & &   &  \cr
}
\textrm{,}\end{displaymath}
where $i,j\in V_{\mathcal{P}}$ and $k,l\in V_{\mathcal{H}}$.  If $d=2$ and $\{k,l\}\in E_{\mathcal{HH}\parallel}$ then the entries in the row for $\{k,l\}$ are  $(a_l^\perp,0)$ under the columns for $k$ and $-(a_k^\perp,0)$ under the columns for $l$.  If $d=3$ and $\{k,l\} \in E_{\mathcal{HH}}$ then the row for $\{k,l\}$ becomes two rows with entries $(a_l \wedge u_i,0)$ under the columns for $k$ and $(-a_k \wedge  u_i,0)$ under the columns for $l$ for $i=1,2$. Analogously, using the constraints given in the previous subsection, the row for $\{k,l\}$ becomes $d-1$ rows for an arbitrary $d$, but $d=2,3$ are the most relevant cases for applications.

\subsection{Point-hyperplane frameworks with extrusion symmetry} \label{sec:pthy}

Extrusion symmetry for point-hyperplane frameworks may be defined analogously as for bar-joint frameworks. However, in the following we will extend the definition of extrusion symmetry to allow for the possibility that  some hyperplanes of a point-hyperplane framework contain some of the extrusion directions of the framework, so that they are `fixed' by the corresponding elements of the extrusion symmetry group.

Let $H$ be a $\mathcal P \mathcal H$-graph with $V(H)=V_{\mathcal P}(H)\cup V_{\mathcal H}(H)$, where $V_\mathcal P(H)=\{v_1,\ldots, v_a\}$ and $V_{\mathcal H}(H)=\{w_1,\ldots, w_b\}$. 
Let $G=H\square K_2=(V,E)$ be the Cartesian product of $H$ and $K_2$, where $V(K_2)=\{0,1\}$. Then $G$ is again a $\mathcal P \mathcal H$-graph, where the partition of the vertex set of $G$ is inherited from the partition of $V(H)$ (i.e. any vertex of the form $(v_j,e)$, $v_j\in V_{\mathcal P}(H)$, is in $V_{\mathcal P}$ and any vertex of the form $(w_j,e)$, $w_j\in V_{\mathcal H}(H)$, is in $V_{\mathcal H}$). Let $F$ be a subset of $V_{\mathcal H}(H)$. Then $G'=H\square_{F} K_2$ is the $\mathcal P \mathcal H$-graph  obtained from  $G$ by contracting each pair of vertices $(w,0),(w,1)$ with $w\in F$ to a single vertex, denoted $(w,\star)$, and by removing the edges  $\{(w,0),(w,1)\}$, $w\in F$.

More generally, let $F_1,\ldots , F_t$ be $t$ (not necessarily disjoint) subsets of $V_{\mathcal H}(H)$. Then the $\mathcal P \mathcal H$-graph $G=H\square_{F_1,\ldots, F_t} K_2^{\square t}$ is obtained from $H\square K_2^{\square t}$ by the following procedure: For $h=1,\ldots, t$, contract the pairs of vertices $(w,\bm e), (w,\bm e')$ with $w\in F_h$ and $\bm e$ and $\bm e'$ differing  in the $h$-th coordinate  to a single vertex, which is denoted by $(w,e_1,\ldots, \star, \ldots e_t)$, where $\star$ is in the $h$-th position. (Note that here $\bm e, \bm e'\in \{0,1,\star\}^t$.) Edges between contracted vertices are again  removed. (So if $F_1,\ldots, F_t$ are disjoint, then only pairs of vertices are contracted to a single vertex in $G$, but otherwise multiple vertices are  contracted to the same vertex.)

Then there clearly exists a group homomorphism $\theta: \mathbb{Z}_2^t\to \textrm{Aut}(G)$ given by $\theta(\gamma)=\alpha_{\gamma}$, where  
$$\alpha_{\gamma}(u,\bm e)=(u,\bm e+\gamma) \qquad \textrm{ for all } (u,\bm e)\in V $$ with the addition in $\bm e+\gamma$ taken modulo 2, and $\star+0=\star+1=\star$. %and 
We call $\theta$ the \emph{extrusion action} on $G$.

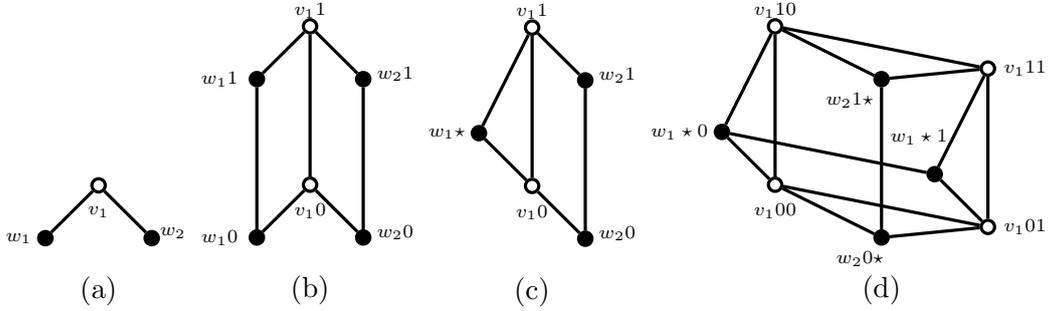
\begin{figure}[htp]
\begin{center}
  \begin{tikzpicture}[very thick,scale=0.7]
\tikzstyle{every node}=[circle, draw=black, fill=white, inner sep=0pt, minimum width=5pt];
     \node [draw=white, fill=white] (a) at (-0.5,0) {\tiny $w_1$};
              \node [draw=white, fill=white] (a) at (2.4,0.1) {\tiny $w_2$};
             \node [draw=white, fill=white] (a) at (1,0.5) {\tiny $v_1$};
                           
        \path (0,0) node[fill=black] (p1) {} ;
        \path (2,0) node[fill=black] (p2) {} ;
        \path (1,1) node (p3) {} ;
        
         \draw (p1)  --  (p3);
        \draw (p3)  --  (p2);
            
        \node [draw=white, fill=white] (a) at (1,-1) {(a)};
      \end{tikzpicture}    
            \vspace{0.3cm}
      \begin{tikzpicture}[very thick,scale=0.7]
\tikzstyle{every node}=[circle, draw=black, fill=white, inner sep=0pt, minimum width=5pt];
     \node [draw=white, fill=white] (a) at (-0.7,0) {\tiny $w_10$};
        \node [draw=white, fill=white] (a) at (-0.7,3) {\tiny $w_11$};
        \node [draw=white, fill=white] (a) at (2.6,0.1) {\tiny $w_20$};
        \node [draw=white, fill=white] (a) at (2.6,3.1) {\tiny $w_21$};
           \node [draw=white, fill=white] (a) at (1,0.5) {\tiny $v_10$};
        \node [draw=white, fill=white] (a) at (1,4.3) {\tiny $v_11$};

        \path (0,0) node[fill=black] (p1) {} ;
        \path (2,0) node[fill=black] (p2) {} ;
        \path (1,1) node (p3) {} ;
        
         \path (0,3) node[fill=black] (p11) {} ;
        \path (2,3) node[fill=black] (p22) {} ;
        \path (1,4) node (p33) {} ;
        
        %\draw (p1)  --  (p2);
         \draw (p1)  --  (p3);
        \draw (p3)  --  (p2);
       % \draw (p11)  --  (p22);
         \draw (p11)  --  (p33);
        \draw (p33)  --  (p22);
        \draw (p1)  --  (p11);
         \draw (p2)  --  (p22);
        \draw (p3)  --  (p33);

        \node [draw=white, fill=white] (a) at (1,-1) {(b)};
      \end{tikzpicture}    
            \vspace{0.3cm}
            \begin{tikzpicture}[very thick,scale=0.7]
\tikzstyle{every node}=[circle, draw=black, fill=white, inner sep=0pt, minimum width=5pt];
     \node [draw=white, fill=white] (a) at (-0.6,2) {\tiny $w_1\star$};
                 \node [draw=white, fill=white] (a) at (2.6,0.1) {\tiny $w_20$};
        \node [draw=white, fill=white] (a) at (2.6,3.1) {\tiny $w_21$};
                   \node [draw=white, fill=white] (a) at (1,0.5) {\tiny $v_10$};
        \node [draw=white, fill=white] (a) at (1,4.3) {\tiny $v_11$};

            \path (2,0) node[fill=black] (p2) {} ;
                \path (2,3) node[fill=black] (p22) {} ;
                        \path (0,2) node[fill=black] (p1) {} ;
          \path (1,1) node (p3) {} ;
        \path (1,4) node (p33) {} ;

                 \draw (p1)  --  (p3);
        \draw (p3)  --  (p2);
    \draw (p22)  --  (p2);
         \draw (p1)  --  (p33);
        \draw (p33)  --  (p22);
      
        \draw (p3)  --  (p33);

        \node [draw=white, fill=white] (a) at (1,-1) {(c)};
      \end{tikzpicture}
      \vspace{0.3cm}
\begin{tikzpicture}[very thick,scale=0.7]
\tikzstyle{every node}=[circle, draw=black, fill=white, inner sep=0pt, minimum width=5pt];
        \node [draw=white, fill=white] (a) at (-0.8,2) {\tiny $w_1\star 0$};
                 \node [draw=white, fill=white] (a) at (2.6,-0.4) {\tiny $w_20\star$};
        \node [draw=white, fill=white] (a) at (2.4,2.6) {\tiny $w_21\star$};
                   \node [draw=white, fill=white] (a) at (1,0.5) {\tiny $v_100$};
        \node [draw=white, fill=white] (a) at (1,4.3) {\tiny $v_110$};
               
                          \node [draw=white, fill=white] (a) at (3.7,1.9) {\tiny $w_1\star 1$};
        \node [draw=white, fill=white] (a) at (5.7,3.2) {\tiny $v_111$};
                   \node [draw=white, fill=white] (a) at (5.7,0.2) {\tiny $v_101$};

            \path (3,0) node[fill=black] (p2) {} ;
                \path (3,3) node[fill=black] (p22) {} ;
                        \path (0,2) node[fill=black] (p1) {} ;
          \path (1,1) node (p3) {} ;
        \path (1,4) node (p33) {} ;
        
         \path (5,3.2) node (p3rt) {} ;
                    \path (5,0.2) node (p3rb) {} ;
                \path (4,1.2) node[fill=black] (p11) {} ;

                 \draw (p1)  --  (p3);
        \draw (p3)  --  (p2);
    \draw (p22)  --  (p2);
         \draw (p1)  --  (p33);
        \draw (p33)  --  (p22);
      
        \draw (p3)  --  (p33);
         \draw (p11)  --  (p3rb);
         \draw (p11)  --  (p3rt);
           \draw (p3rb)  --  (p3rt);
             \draw (p1)  --  (p11);
             \draw (p22)  --  (p3rt);
               \draw (p2)  --  (p3rb);
                 \draw (p3rt)  --  (p33);
                 \draw (p3rb)  --  (p3);
       \node [draw=white, fill=white] (a) at (3,-1) {(d)};
      \end{tikzpicture}
     \end{center}
     \vspace{-0.5cm}
\caption{(a) A $\mathcal P \mathcal H$-graph $H$ with $V_{\mathcal P}(H)=\{v_1\}$ (shown in white) and  $V_{\mathcal H}(H)=\{w_1,w_2\}$ (shown in black). (b) shows the $\mathcal P \mathcal H$-graph $G=H\square K_2$. (c) is  the $\mathcal P \mathcal H$-graph  $G'=H\square_{F} K_2$, where $F=\{w_1\}$. Finally, (d) shows the $\mathcal P \mathcal H$-graph $H\square_{F_1,F_2} K_2^{\square 2}$, where $F_1=\{w_1\}$ and $F_2=\{w_2\}$.}
\label{fig:exptlngr}
\end{figure}

\begin{defn}\label{def:extrsymph}
For $G=H\square_{F_1,\ldots, F_t} K_2^{\square t}$ with extrusion action $\theta$, a point-hyperplane framework $(G,p,\ell)$ in $\mathbb{R}^d$ is said to have \emph{$t$-fold extrusion symmetry} if
 there exist $\tau_1,\ldots, \tau_t \in \mathbb{R}^d\setminus \{0\}$, so that 
  \begin{itemize}
  \item[(i)] for every $(v_j,\bm e)\in V_{\mathcal{P}}$, and every $\gamma\in \mathbb{Z}_2^t$ we have 
 $$p_{\theta(\gamma)((v_j,\bm e))}=p_{(v_j,\bm e)}+ \sum_{i\in X} \tau_i - \sum_{g\in Y} \tau_g,$$ where $X$ is the set of indices in $\{1,\ldots, t\}$ for which $\be$ has an entry of $0$ and $\gamma$ has an entry of $1$, and  $Y$ is the set of indices in $\{1,\ldots, t\}$ for which both $\be$ and $\gamma$ have an entry of $1$. 
 \item[(ii)] for every $(w_j,\bm e)\in V_{\mathcal{H}}$ and every $\gamma\in \mathbb{Z}_2^t$, we have $a_{\theta(\gamma)((w_j,\bm e))}=a_{(w_j,\bm e)}$.
 \item[(iii)] for every  $(w_j,\bm e)\in V_{\mathcal{H}}$, and every $h=1,\ldots, t$, the vector $\tau_h$ lies within the hyperplane associated with $(w_j,\bm e)$ (i.e., $\langle\tau_h, a_{(w_j,\bm e)}\rangle = 0$ for the  normal vector $a_{(w_j,\bm e)}$) if  $w_j\in F_h$. 
 \item[(iv)] for every $(w_j,\bm e)\in V_{\mathcal{H}}$  and every $\gamma\in \mathbb{Z}_2^t$ we have
 $$r_{\theta(\gamma)((w_j,\bm e))}=r_{(w_j,\bm e)}+\sum_{i\in X} \langle a_{(w_j,\bm e)},\tau_i\rangle - \sum_{g\in Y} \langle a_{(w_j,\bm e)},\tau_g\rangle,$$
 where $X$ is the set of indices in $\{1,\ldots, t\}$ for which $\be$ has an entry of $0$ and $\gamma$ has an entry of $1$, and  $Y$ is the set of indices in $\{1,\ldots, t\}$ for which both $\be$ and $\gamma$ have an entry of $1$. 
 \end{itemize}
As in Definition~\ref{def:extrdef}, for $(v_j,\bm e)\in V_{\mathcal{P}}$, the vector  $\tau_{\gamma}((v_j,\bm e))=\sum_{i\in X} \tau_i - \sum_{g\in Y} \tau_g$ is called the \emph{extrusion direction of $(v_j,\be)$ induced by $\gamma$}. Similarly, by (iv), for $(w_j,\bm e)\in V_{\mathcal{H}}$, the vector $\tau_{\gamma}((w_j,\bm e))=\sum_{i\in X} \tau_i - \sum_{g\in Y} \tau_g$ is called the \emph{extrusion direction of $(w_j,\be)$ induced by $\gamma$}. The vectors $\tau_1,\ldots, \tau_t$ are called the \emph{extrusion directions} of $(G,p,\ell)$.
 \end{defn}
 
 Examples of $2$-dimensional point-hyperplane frameworks (i.e., point-line frameworks) with extrusion symmetry are shown in Figure~\ref{fig:exptlngrfw}.

\begin{figure}[htp]
\begin{center}
  \begin{tikzpicture}[very thick,scale=0.7]
\tikzstyle{every node}=[circle, draw=black, fill=white, inner sep=0pt, minimum width=4pt];
     \node [draw=white, fill=white] (a) at (0.3,0.3) {\tiny $p_1$};

        \path (0,0) node (p1) {} ;
       
           \draw(-2,-0.5)--(1.5,-0.5);    
             \draw(-2,-1)--(0,1);   
             
               \draw[dotted](p1)--(0,-0.5);  
                 \draw[dotted](p1)--(-0.5,0.5);  
                 
             \node [draw=white, fill=white] (a) at (0.5,-0.8) {\tiny $\ell_2$};   
                \node [draw=white, fill=white] (a) at (-1.2,0.3) {\tiny $\ell_1$}; 
        \node [draw=white, fill=white] (a) at (0,-1.8) {(a)};
      \end{tikzpicture} 
      \hspace{0.5cm}   
    \begin{tikzpicture}[very thick,scale=0.7]
\tikzstyle{every node}=[circle, draw=black, fill=white, inner sep=0pt, minimum width=4pt];
        \node [draw=white, fill=white] (a) at (-0.7,-0.8) {\tiny $\ell_20$};   
                \node [draw=white, fill=white] (a) at (-1.3,0.2) {\tiny $\ell_10$};
   \node [draw=white, fill=white] (a) at (1.8,1.8) {\tiny $\ell_21$};   
                \node [draw=white, fill=white] (a) at (1.8,0.2) {\tiny $\ell_11$};
   
     \node [draw=white, fill=white] (a) at (0.1,0.4) {\tiny $p_10$};
              \path (0,0) node (p1) {} ;
       
           \draw(-2,-0.5)--(3.5,-0.5);    
             \draw(-2,-1)--(2,3);   
             
               \draw[dotted](p1)--(0,-0.5);  
                 \draw[dotted](p1)--(-0.5,0.5);  
                           
                \node [draw=white, fill=white] (a) at (4.3,2.3) {\tiny $p_11$};                
        \path (4,2) node (p2) {} ;
       
           \draw(0,1.5)--(5.5,1.5);    
             \draw(0,-1)--(4,3);   
             
               \draw[dotted](p2)--(4,1.5);  
                 \draw[dotted](p2)--(3.5,2.5);  
                 
                      \draw[dotted](p1)--(p2); 
                      
             % \draw[dotted](0.2,1.2)--(1.2,0.2);      
                      \draw[dotted](0.2,1.2)--(2.2,1.2);       
                 \draw[dotted](2.2,-0.5)--(4.2,1.5);

        \node [draw=white, fill=white] (a) at (2,-1.8) {(b)};
      \end{tikzpicture} 
      \hspace{0.1cm}   
    \begin{tikzpicture}[very thick,scale=0.7]
\tikzstyle{every node}=[circle, draw=black, fill=white, inner sep=0pt, minimum width=4pt];
   \node [draw=white, fill=white] (a) at (-0.7,-0.8) {\tiny $\ell_20$};   
                \node [draw=white, fill=white] (a) at (-1.3,0.2) {\tiny $\ell_1\star$};
   \node [draw=white, fill=white] (a) at (2.8,1.8) {\tiny $\ell_21$};   

     \node [draw=white, fill=white] (a) at (0.1,0.4) {\tiny $p_10$};
              \path (0,0) node (p1) {} ;
       
           \draw(-2,-0.5)--(1.5,-0.5);    
             \draw(-2,-1)--(2,3);   
             
               \draw[dotted](p1)--(0,-0.5);  
                 \draw[dotted](p1)--(-0.5,0.5);  
                           
                   \node [draw=white, fill=white] (a) at (2.3,2.3) {\tiny $p_11$};          
        \path (2,2) node (p2) {} ;
       
           \draw(-0,1.5)--(3.5,1.5);    
              
              \draw[dotted](p1)--(p2); 
               \draw[dotted](p2)--(2,1.5);  
                 \draw[dotted](p2)--(1.5,2.5);  
                 
               \draw[dotted](0.2,-0.5)--(2.2,1.5);   
        \node [draw=white, fill=white] (a) at (0,-1.8) {(c)};
      \end{tikzpicture} 
      
           \vspace{0.2cm}   
    \begin{tikzpicture}[very thick,scale=0.7]
\tikzstyle{every node}=[circle, draw=black, fill=white, inner sep=0pt, minimum width=4pt];
     \node [draw=white, fill=white] (a) at (0.7,-0.8) {\tiny $\ell_20\star$};   
                \node [draw=white, fill=white] (a) at (-1.5,0.2) {\tiny $\ell_1\star 0$};
   \node [draw=white, fill=white] (a) at (3.8,1.7) {\tiny $\ell_21\star$};  
      \node [draw=white, fill=white] (a) at (2.7,0.2) {\tiny $\ell_1\star 1$};
  
     \node [draw=white, fill=white] (a) at (0.1,0.4) {\tiny $p_100$};
              \path (0,0) node (p1) {} ;
       
           \draw(-2,-0.5)--(6,-0.5);    
             \draw(-2,-1)--(2,3);   
             
               \draw[dotted](p1)--(0,-0.5);  
                 \draw[dotted](p1)--(-0.5,0.5);  
                           
                   \node [draw=white, fill=white] (a) at (2.3,2.3) {\tiny $p_110$};          
        \path (2,2) node (p2) {} ;
       
           \draw(-0,1.5)--(7,1.5);    
              
              \draw[dotted](p1)--(p2); 
               \draw[dotted](p2)--(2,1.5);  
                 \draw[dotted](p2)--(1.5,2.5);  
                 
               \draw[dotted](1.2,-0.5)--(3.2,1.5);

          \node [draw=white, fill=white] (a) at (4.8,0.1) {\tiny $p_101$};
              \path (4,0) node (p3) {} ;

             \draw(2,-1)--(6,3);   
             
               \draw[dotted](p3)--(4,-0.5);  
                 \draw[dotted](p3)--(3.5,0.5);  
                           
                   \node [draw=white, fill=white] (a) at (6.3,2.3) {\tiny $p_111$};          
        \path (6,2) node (p4) {} ;
       
           \draw(-0,1.5)--(3.5,1.5);    
              
              \draw[dotted](p3)--(p4); 
               \draw[dotted](p4)--(6,1.5);  
                 \draw[dotted](p4)--(5.5,2.5);  
                 
             \draw[dotted](p3)--(p1);        
                \draw[dotted](p2)--(p4); 
                
                 \draw[dotted](0.2,1.2)--(4.2,1.2);   
        \node [draw=white, fill=white] (a) at (2,-1.8) {(d)};
      \end{tikzpicture}
     \end{center}
     \vspace{-0.5cm}
\caption{Point-line frameworks  in the plane, with all types of constraints shown as dotted lines. The underlying graph of the framework in (a)-(d) is the respective graph in (a)-(d) shown in Figure~\ref{fig:exptlngr}. The frameworks in (b) and (c) have ($1$-fold) extrusion symmetry and are obtained by extruding the framework in (a). The framework in (c) is obtained by extruding the framework in (a) along the direction of the line $\ell_1$. The framework in (d)  has $2$-fold extrusion symmetry and is obtained from the one in (c) by extruding in the direction of the lines $\ell_20$ and $\ell_21$.}
\label{fig:exptlngrfw}
\end{figure}
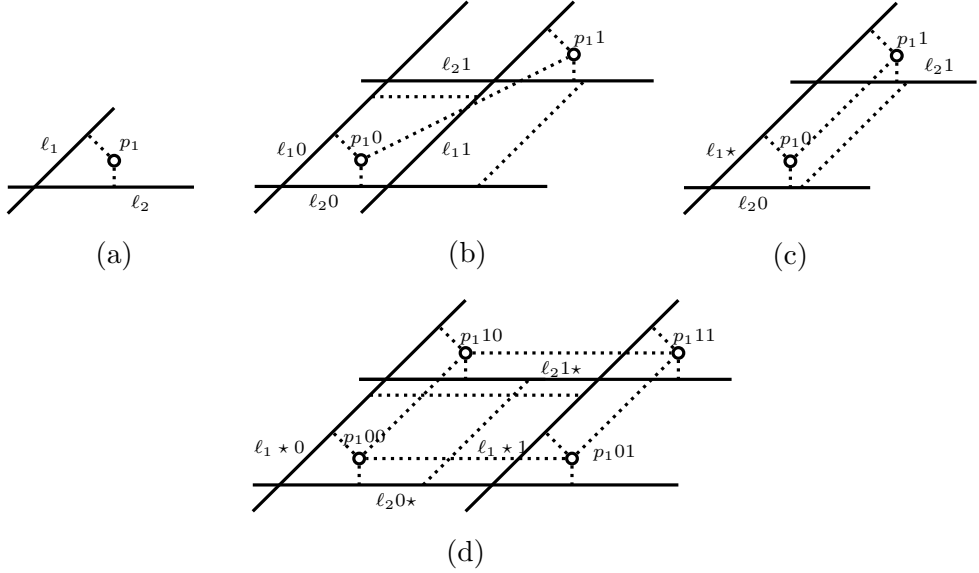

\subsection{Point-hyperplane frameworks with no edge $\{i,j\}$ in $E_{\mathcal P \mathcal H}$ for which $j$ is fixed by an extrusion}
\label{subsec:blockdecompthyper}
In this section we show that the rigidity matrix of a point-hyperplane framework with $t$-fold extrusion symmetry  can be transformed into a block-diagonalised form using the same approach as in Section~\ref{subsec:blockdecom}, provided that there is no point-hyperplane distance constraint where the corresponding hyperplane is fixed by an extrusion.

Let $G=H\square_{F_1,\ldots, F_t} K_2^{\square t}=(V,E)$ and let $(G,p,\ell)$ be a point-hyperplane framework in $\mathbb{R}^d$ with $t$-fold extrusion symmetry and extrusion directions $\tau_1,\ldots, \tau_t$. Throughout this subsection we assume that whenever there is a vertex $(w,\be)\in V_{\mathcal H}$ with $w\in F_h$ for some $h\in\{1,\ldots, t\}$ (i.e., by Definition~\ref{def:extrsymph}, the hyperplane corresponding to $(w,\be)$ contains the extrusion direction $\tau_h$), then there is no edge in $E_{\mathcal P \mathcal H}$ that is incident to $(w,\be)$ (i.e. there is no point whose distance to the hyperplane corresponding to $(w,\be)$ is fixed).
We let $\Gamma=\mathbb{Z}_2^t$, and as before we denote $\theta(\gamma)(i)$ by $\gamma i$ for $\gamma\in \Gamma$ and $i\in V$.

 Recall from Section~\ref{sec:parcon} that a hyperplane-hyperplane parallel constraint is different from  an angle constraint.% angle constraint becomes singular and hence takes on a different form if the hyperplanes are parallel to each other. 
 Thus, in the following, we will distinguish between edges corresponding to angle constraints and edges corresponding to (sets of) parallel constraints and denote them (as before) by $E_{\mathcal{HH}\not\parallel}$ and $E_{\mathcal H \mathcal H\parallel}$, respectively.
 
 Let $\gamma\in \Gamma$, and let $e=\{j,l\}$ be an edge of $G$ in  $E_{\mathcal H \mathcal H\parallel}$ with $j$ and $l$ of the form $j=(w,\bm e)$ and $l=(w,\bm e')$, where $w\in V_{\mathcal{H}}(H)$ and $\be, \bm e'\in \{0,1,\star\}^t$. Note that $\be$ and $\bm e'$ must differ in exactly one entry, say the $h$-th entry. Suppose $\be_h=0$ and $\be'_h=1$. Then we use the convention that any row in the rigidity matrix corresponding to the parallel constraint given by $e$ (there are $d-1$ of them in dimension $d$) will have  $(\alpha,0)$ for some $\alpha\in \mathbb{R}^d$ in the $d+1$ columns corresponding to $j$ and $(-\alpha,0)$ in the $d+1$ columns corresponding to $l$, and zeros elsewhere.
 
 In the following we assume that the vertices and edges of $G$ are ordered so that in  the rigidity matrix $R(G,p,\ell)$ the columns corresponding to the vertices in $V_{\mathcal P}$ are followed by the columns corresponding to the vertices  in $V_{\mathcal H}$, and the rows corresponding to the edges in $E_{\mathcal P \mathcal P}$ are followed by the rows corresponding to the edges in $E_{\mathcal P \mathcal H}$, which  in turn are followed by the rows corresponding to the edges in $E_{\mathcal H \mathcal H\not\parallel}$ and $E_{\mathcal H \mathcal H\parallel}$. The final set of columns of $R(G,p,\ell)$ correspond to the vertices in $V_{\mathcal H}$.

We need to define the analogues of the internal and external representation of $\Gamma$. 

The \emph{internal representation} of $\Gamma$ (with respect to $G=H\square_{F_1,\ldots, F_t} K_2^{\square t}$) is the linear representation $P'_E:\Gamma\to GL(\mathbb{R}^{|E|})$ defined by 
$$P'_E=P'_{E_{\mathcal P \mathcal P}}\oplus P_{E_{\mathcal P \mathcal H}}\oplus P_{E_{\mathcal H \mathcal H \not\parallel}}\oplus (P'_{E_{\mathcal H \mathcal H \parallel}}\otimes I_{d-1})\oplus P_{V_{\mathcal H}},$$
where $P'_{E_{\mathcal P \mathcal P}}$ is the internal representation of the subgraph of $G$ induced by $V_{\mathcal P}$, as defined in Section~\ref{subsec:blockdecom}, and $P_{E_{\mathcal P \mathcal H}}$ and $P_{E_{\mathcal H \mathcal H\not\parallel}}$ are the permutation representations for the edge sets $E_{\mathcal P \mathcal H}$ and  $E_{\mathcal H \mathcal H\not\parallel}$, respectively.

The representation $P'_{E_{\mathcal H \mathcal H\parallel}}$ is obtained from the permutation representation $P_{E_{\mathcal H \mathcal H\parallel}}$ of the edge set $E_{\mathcal H \mathcal H\parallel}$ in the analogous way as $P'_{E_{\mathcal P \mathcal P}}$ is obtained from $P_{E_{\mathcal P \mathcal P}}$. 
Let $\gamma\in \Gamma$, and let $e=\{j,l\}$ be an edge of $G$ in  $E_{\mathcal H \mathcal H\parallel}$ with $j$ and $l$ of the form $j=(w,\bm e)$ and $l=(w,\bm e')$, where $w\in V_{\mathcal{H}}(H)$ and $\be, \bm e'\in \{0,1,\star\}^t$. Moreover, let $\gamma e=f$. The vectors $\be$ and $\bm e'$ differ in exactly one entry, say the $h$-th entry. We replace the $1$ in the $(e,f)$-th and the $(f,e)$-th entry of $P_{E_{\mathcal H \mathcal H\parallel}}(\gamma)$ by $-1$ if and only if $\gamma$ has a $1$ in the $h$-th entry. We obtain $P'_{E_{\mathcal H \mathcal H\parallel}}(\gamma)$ by applying this procedure to every edge of  $E_{\mathcal H \mathcal H\parallel}$.

Note that for $\gamma\neq \textrm{id}=(0,\ldots, 0)$, an edge $e=\{j,l\}$ of $E_{\mathcal H \mathcal H\parallel}$ makes a contribution of $1$ to the diagonal of $P_{E_{\mathcal H \mathcal H\parallel}}(\gamma)$ if and only if $e$ is \emph{fixed} by $\gamma$, i.e. if either $\gamma j=l$ and $\gamma l=j$ or $\gamma j=j$ and $\gamma l=l$. 
 Moreover, $e=\{j,l\}$ is fixed by $\gamma$ if and only if $j$ and $l$ are of the form $j=(w,\bm e)$ and $l=(w,\bm e')$, with $w\in V_{\mathcal{H}}(H)$ and $\be, \bm e'\in \{0,1,\star\}^t$ differing only in the $h$-th entry, and $\gamma$ has a zero in all entries, except possibly in the $h$-th entry and the entries where $\bm e$ and $\bm e'$ have a $\star$. 
If $\gamma$ has a $1$ in the $h$-th entry (i.e. the hyperplanes of $j$ and $l$ are swapped by $\gamma$), then the
$1$ in the diagonal of the matrix $P_{E_{\mathcal H \mathcal H\parallel}}(\gamma)$ is replaced by a $-1$ in $P'_{E_{\mathcal H \mathcal H\parallel}}$. If $\gamma$ has a $0$ in the $h$-th entry (i.e. the hyperplanes of $j$ and $l$ are fixed by $\gamma$), then the
$1$ in the diagonal of the matrix $P_{E_{\mathcal H \mathcal H\parallel}}(\gamma)$ is kept in $P'_{E_{\mathcal H \mathcal H\parallel}}(\gamma)$.

 It is straightforward to check that $P'_{E_{\mathcal H \mathcal H\parallel}}$ is indeed a linear representation. In $P'_E$, we take the Kronecker product of  $P'_{E_{\mathcal H \mathcal H\parallel}}$ and the identity matrix $I_{d-1}$, since each parallel constraint between two hyperplanes in $d$-space gives $d-1$ rows in the rigidity matrix, as described in Section~\ref{sec:parcon}.

Finally, $P_{V_{\mathcal H}}$ is the  vertex permutation representation of $V_{\mathcal H}$.

The \emph{external representation} of $\Gamma$ (with respect to $G=H\square_{F_1,\ldots, F_t} K_2^{\square t}$) is the linear representation $$P_V'=(P_{V_{\mathcal P}}\otimes I_d)\oplus (P_{V_{\mathcal H}}'):\Gamma\to GL(\mathbb{R}^{d|V_{\mathcal P}|+(d+1)|V_{\mathcal H}|}),$$
where $P_{V_{\mathcal P}}$ is the vertex permutation representation of $V_{\mathcal P}$, and $P_{V_{\mathcal H}}':\Gamma \to  GL(\mathbb{R}^{(d+1)|V_{\mathcal H}|})$ is defined as follows.

For $\gamma\in \Gamma$, the matrix $P_{V_{\mathcal H}}'(\gamma)$ is obtained from the vertex permutation matrix $P_{V_{\mathcal H}}(\gamma)=[\delta_{i,\gamma j}]_{i,j}$ (where $\delta$ denotes the Kronecker delta symbol) by replacing each $0$ with a $(d+1)\times (d+1)$ zero-matrix, and each $1$ in entry $(i,j)$ by the matrix 
$$ 
\begin{pmatrix}
  \begin{matrix}
   &  \\
   & I_d  \\
   &\\
  \end{matrix}
  &  \vline  \mathbf{0}   \\
\hline 
   \begin{matrix}
   & -\tau_{\gamma}(i)^T  \\
  \end{matrix}
   &  \vline  1   \\
\end{pmatrix}
$$
It is straightforward to check that $P_{V_{\mathcal H}}'$ is a group representation of $\Gamma$.

Let $G=H\square_{F_1,\ldots, F_t} K_2^{\square t}=(V,E)$ and let $(G,p,\ell)$ be a point-hyperplane framework in $\mathbb{R}^d$ with $t$-fold extrusion symmetry and extrusion directions $\tau_1,\ldots, \tau_t$. In the following we assume that whenever there is a vertex $(w,\be)\in V_{\mathcal H}$ with $w\in F_h$ for some $h\in\{1,\ldots, t\}$ (i.e. the hyperplane corresponding to $(w,\be)$ contains the extrusion direction $\tau_h$), then there is no edge in $E_{\mathcal P \mathcal H}$ that is incident to $(w,\be)$ (i.e. there is no point whose distance to the hyperplane corresponding to $(w,\be)$ is fixed).

\begin{thm}
\label{thm:blockph}
Let $G=H\square_{F_1,\ldots, F_t} K_2^{\square t}$ and $(G,p,\ell)$ be  a point-hyperplane framework in $\mathbb{R}^d$ with $t$-fold extrusion symmetry and extrusion directions $\tau_1,\ldots, \tau_t$. Further, let $\theta:\Gamma\to \textrm{Aut}(G)$ be the extrusion action. Suppose there is no edge $\{i,j\}\in E_{\mathcal P \mathcal H}$ where the hyperplane corresponding to $j$ contains an extrusion direction.  
Then $$R(G,p,\ell)\in {\rm Hom}_{\Gamma}(P_{V}',P'_E).$$
\end{thm}
\begin{proof}  Suppose we have $R(G,p,\ell)(\dot p, \dot \ell)=s$, where $\dot \ell=(\dot a,\dot r)$. Then we need to show that $$R(G,p,\ell)\big((P_{V_{\mathcal P}}\otimes I_d)\oplus (P_{V_{\mathcal H}}')\big)(\gamma) (\dot p,\dot \ell) = P'_E(\gamma)s.$$

Fix $\gamma\in \Gamma$ and let $e=\{i,j\}\in E$. Suppose that $\gamma i =q$ and $\gamma j =l$, and hence $P_E(\gamma)(\{i,j\})=\{q,l\}$.
We now consider each type of edge in turn. 

\textbf{Type 1 ($\mathcal{PP}$):}  If $\{i,j\}\in E_{\mathcal P \mathcal P}$ then so is $\{q,l\}$ and we have 
 $$(P'_E(\gamma)s)_{\{q,l\}}=\big(R(G,p)(P_V')(\gamma) (\dot p, \dot \ell)\big)_{\{q,l\}}$$
 by the proof of Theorem~\ref{thm:block}.

\textbf{Type 2 ($\mathcal{PH}$):}
Suppose  that $\{i,j\}\in E_{\mathcal P \mathcal H}$, with $i\in V_{\mathcal P}$ and $j\in V_{\mathcal H}$. Then we have  $ a_l= a_j$ and $p_q=p_i+\tau_{\gamma}(i)$.   
Moreover, note that if   $j=(w_j,\bm e)$ and $i=(v_i,\bm e')$, then, since $\{i,j\}\in E_{\mathcal P \mathcal H}$ and since $w_j\neq F_h$ for all $h=1,\ldots, t$, the vectors $\bm e $ and $\bm e'$ are equal. Thus, by Definition~\ref{def:extrsymph}, we have $\tau_\gamma(j)=\tau_\gamma(i)$.

By definition of $P'_E$ and the definition of the point-hyperplane rigidity matrix, for the $\{q,l\}$-th component $(P'_E(\gamma)s)_{\{q,l\}}$ of the column vector $P'_E(\gamma)s$, we have:
$$(P'_E(\gamma)s)_{\{q,l\}}= s_{\{i,j\}}= \langle a_j,\dot p_i\rangle+ \langle(p_i,-1),\dot \ell_j \rangle.$$

Next we consider $\big(R(G,p,\ell)(P_V')(\gamma) ((\dot p, \dot \ell))\big)_{\{q,l\}}$ and show that it is equal to $(P'_E(\gamma)s)_{\{q,l\}}$.

Note that $(R(G,p,\ell)((\dot p, \dot \ell)))_{\{q,\ell\}}=\langle a_l,\dot p_q\rangle+ \langle (p_q,-1),\dot \ell_l\rangle $. Thus, by the definition of $P_{V}'$, we have 
\begin{eqnarray*}\big(R(G,p,\ell)(P_V')(\gamma) ((\dot p, \dot \ell))\big)_{\{q,l\}}&=& \langle a_l, \dot p_i\rangle + \langle (p_q,-1),(\dot a_j,\langle -\tau_{\gamma}(l),\dot a_j\rangle+\dot r_j)\rangle\\
 &=& \langle a_j, \dot p_i \rangle+ \langle p_q, \dot a_j \rangle + \langle \tau_{\gamma}(l), \dot a_j\rangle-\dot r_j\\
  &=& \langle  a_j, \dot p_i \rangle + \langle p_q+\tau_{\gamma}(l), \dot a_j \rangle -\dot r_j\\
  &=& \langle  a_j, \dot p_i \rangle + \langle p_i+\tau_{\gamma}(i)+\tau_{\gamma}(l), \dot a_j \rangle -\dot r_j\\
  \end{eqnarray*}
 Now, since $\gamma j=l$ and $\gamma l=j$ we have $\tau_\gamma(l)=-\tau_\gamma(j)$.  Thus, since $\tau_\gamma(j)=\tau_\gamma(i)$, 
  the above equals
$$  \langle  a_j, \dot p_i \rangle + \langle p_i, \dot a_j \rangle -\dot r_j= \langle a_j,\dot p_i\rangle+ \langle(p_i,-1),\dot \ell_j \rangle,
$$
as desired.
 
 \textbf{Type 3 ($\mathcal{HH}\not\parallel)$:}
Suppose that $\{i,j\}\in E_{\mathcal H \mathcal H}$. Then we have $ a_q= a_i$ and $a_l= a_j$. By definition of $P'_E$ and the definition of the point-hyperplane rigidity matrix, for the $\{q,l\}$-th component $(P'_E(\gamma)s)_{\{q,l\}}$ of the  column vector $P'_E(\gamma)s$, we have:
$$(P'_E(\gamma)s)_{\{q,l\}}= s_{\{i,j\}}= \langle (a_j,0),\dot \ell_i\rangle+ \langle (a_i,0),\dot \ell_j\rangle.$$
Next we consider $\big(R(G,p,\ell)(P_V')(\gamma) ((\dot p, \dot \ell))\big)_{\{q,l\}}$ and show that it is equal to $(P'_E(\gamma)s)_{\{q,l\}}$.

Note that $(R(G,p,\ell)((\dot p, \dot \ell)))_{\{q,l\}}= \langle (a_l,0),\dot \ell_q\rangle+ \langle (a_q,0),\dot \ell_l\rangle$. Thus, by definition of  $P_{V}'$, we have 
\begin{eqnarray*}
(R(G,p,\ell)(P_V')(\gamma)((\dot p, \dot \ell)))_{\{q,l\}}&=&  \langle (a_l,0),(\dot a_i,\langle-\tau_{\gamma}(q),\dot a_i \rangle+\dot r_i)\rangle+ \langle (a_q,0),(\dot a_j,\langle-\tau_{\gamma}(l),\dot a_j \rangle+\dot r_j)\rangle\\
&=& \langle (a_l,0),\dot \ell_i\rangle+ \langle (a_q,0),\dot \ell_j\rangle\\
&=& \langle (a_j,0),\dot \ell_i\rangle+ \langle (a_i,0),\dot \ell_j\rangle,
  \end{eqnarray*}
  as desired.
  
 \textbf{Type 4 ($\mathcal{HH}\parallel$):} 
  Suppose that $\{i,j\}\in E_{\mathcal H \mathcal H\parallel}$. Then there are $d-1$ rows in the rigidity matrix corresponding to this edge. Since these rows  all have the same basic form, namely $$(0,\ldots, 0, (\alpha,0) , 0 ,\ldots, 0, (-\alpha,0), 0, \ldots, 0)$$ for some $\alpha\in \mathbb{R}^d$, without loss of generality we will just focus on the first of these rows. We have $ a_i= a_j=a_q= a_l$. By definition of $P'_E$, for the $\{q,l\}$-th component $(P'_E(\gamma)s)_{\{q,l\}}$ of the  column vector $P'_E(\gamma)s$, we have:
$$(P'_E(\gamma)s)_{\{q,l\}}=  
\begin{cases}
   -s_{\{i,j\}} ,& \text{if } i=(w,\be) \textrm{ and } j=(w,\bm e'), \be_h\neq \bm e'_h \textrm{ and } \gamma_h=1\\
    s_{\{i,j\}},              & \text{otherwise}
\end{cases}
$$
Let  $i=(w,\be)$  and $j=(w,\bm e')$, where $\be_h\neq \bm e'_h$. Without loss of generality we assume  that  $\be_h=0$ and  $\bm e'_h=1$. 

Suppose first that   $\gamma_h=1$. 
 Then 
$$(P'_E(\gamma)s)_{\{q,l\}}= -s_{\{i,j\}}= \langle (-\alpha,0),\dot \ell_i\rangle+ \langle (\alpha,0),\dot \ell_j\rangle.$$

Next we consider $(R(G,p)(P_V\otimes I_d)(\gamma) (\dot p,\dot \ell))_{\{q,l\}}$ and show that it is equal to $(P'_E(\gamma)s)_{\{q,l\}}$.
We have
 $(R(G,p) (\dot p,\dot \ell))_{\{q,l\}}= \langle(-\alpha,0),\dot \ell_q\rangle+\langle(\alpha,0),\dot \ell_l\rangle$. Therefore, by the definition of $P_V\otimes I_d$, we have
\begin{eqnarray*}(R(G,p)(P_V\otimes I_d)(\gamma)  (\dot p,\dot \ell))_{\{q,l\}}&=&  \langle (-\alpha,0),(\dot a_i,\langle-\tau_{\gamma}(q),\dot a_i \rangle+\dot r_i)\rangle+ \langle (\alpha,0),(\dot a_j,\langle-\tau_{\gamma}(l),\dot a_j \rangle+\dot r_j)\rangle\\
&=&\langle(-\alpha,0),\dot \ell_i\rangle+\langle(\alpha,0),\dot \ell_j\rangle\end{eqnarray*}
as claimed.

Suppose next that  $\gamma_h=0$.  Then 
$$(P'_E(\gamma)s)_{\{q,l\}}= s_{\{i,j\}}= \langle (\alpha,0),\dot \ell_i\rangle+ \langle (-\alpha,0),\dot \ell_j\rangle.$$
We have
 $(R(G,p) (\dot p,\dot \ell))_{\{q,l\}}= \langle(\alpha,0),\dot \ell_q\rangle+\langle(-\alpha,0),\dot \ell_l\rangle$. Therefore, by the definition of $P_V\otimes I_d$, we have
\begin{eqnarray*}(R(G,p)(P_V\otimes I_d)(\gamma)  (\dot p,\dot \ell))_{\{q,l\}}&=&  \langle (\alpha,0),(\dot a_i,\langle-\tau_{\gamma}(q),\dot a_i \rangle+\dot r_i)\rangle+ \langle (-\alpha,0),(\dot a_j,\langle-\tau_{\gamma}(l),\dot a_j \rangle+\dot r_j)\rangle\\
&=&\langle(\alpha,0),\dot \ell_i\rangle+\langle(-\alpha,0),\dot \ell_j\rangle\end{eqnarray*}
as claimed.

\textbf{Type 5 ($V_{\mathcal H}$):} Finally, we consider the rows of the point-hyperplane rigidity matrix that correspond to the vertices in $V_{\mathcal H}$.  Suppose that $i\in V_{\mathcal H}$. Then $ a_q= a_i$. By the definition of $P'_E$, for the  component $(P'_E(\gamma)s)_{q}$ of the  column vector $P'_E(\gamma)s$ corresponding to $q$, we have:
$$(P'_E(\gamma)s)_{q}= s_{i}=\langle (a_i,0),\dot \ell_i\rangle.$$
Note that $(R(G,p,\ell)((\dot p, \dot \ell)))_{q}= \langle (a_q,0),\dot \ell_q\rangle$. Thus, by definition of  $P_{V}'$, we have
$$
(R(G,p,\ell)(P_V')(\gamma)((\dot p, \dot \ell)))_{q}=\langle (a_q,0),(\dot a_i,\langle-\tau_{\gamma}(q),\dot a_i \rangle+\dot r_i)\rangle=\langle (a_i,0),\dot \ell_i)\rangle
$$
This completes the proof.
\end{proof}

By the same reasoning as in Section~\ref{subsec:blockdecom}, it follows from Theorem~\ref{thm:blockph} that the rigidity matrix of a point-hyperplane framework can be block-decomposed, with each block corresponding to an irreducible representation of the group.

 \begin{cor} \label{cor:blockph} Let $G=H\square_{F_1,\ldots, F_t} K_2^{\square t}$ and let $(G,p,\ell)$ be  a point-hyperplane framework in $\mathbb{R}^d$ with $t$-fold extrusion symmetry. Suppose there is no edge $\{i,j\}\in E_{\mathcal P \mathcal H}$ where the hyperplane corresponding to $j$ contains an extrusion direction.  
 Then there exist non-singular matrices $A$ and $B$ such that $B^T R(G,p,\ell) A$ is block-decomposed as
\begin{equation}
\label{rigblocks2}
B^{\top}R(G,p,\ell)A:=\widetilde{R}(G,p,\ell)
=\left(\begin{array}{ccc}\widetilde{R}_{0}(G,p,\ell) & & \mathbf{0}\\ & \ddots & \\\mathbf{0} &  &
\widetilde{R}_{r-1}(G,p,\ell) \end{array}\right)\textrm{,}
\end{equation}
where the submatrix block $\widetilde{R}_i(G,p,\ell)$ corresponds to the $i$-th irreducible representation of  $\Gamma$.
\end{cor}

This block-decomposition corresponds to a decomposition of the space $\mathbb{R}^{|E|+|V_{\mathcal H}|}$ and the space $\mathbb{R}^{d|V_{\mathcal{P}}|+(d+1)|V_{\mathcal{H}}|}$ into a direct sum of subspaces. We use the analogous terminology for these subspaces and their elements as in Section~\ref{subsec:blockdecom}.

\subsection{Point-hyperplane frameworks containing an edge $\{i,j\}$ in $E_{\mathcal P \mathcal H}$ for which $j$ is fixed by an extrusion}\label{pinh} Let $G=H\square_{F_1,\ldots, F_t} K_2^{\square t}=(V,E)$ %, where $F_i\neq \emptyset$  for some $i\in\{1,\ldots, t\}$ 
and let $(G,p,\ell)$ be a point-hyperplane framework in $\mathbb{R}^d$ with $t$-fold extrusion symmetry and extrusion directions $\tau_1,\ldots, \tau_t$. Then,  by Definition~\ref{def:extrsymph},  the extrusion direction  $\tau_h$ lies within the hyperplane associated with the vertex $(w_j,\bm e)\in V_{\mathcal{H}}$ if   $w_j\in F_h$.

Let $w_j\in F_{h_1}\cap\cdots \cap F_{h_c}$ for $h_1,\ldots, h_c\in \{1,\ldots t\}$, and suppose there exists a point-hyperplane constraint between the vertex $(w_j,\bm e)\in V_{\mathcal{H}}$   and the vertex $(v_i,\bm e')\in V_{\mathcal{P}}$. Then Theorem~\ref{thm:blockph} no longer holds, and hence we also do not obtain the desired block-diagonalisation of the rigidity matrix in this case. This is because of the following issue.

For $(w_j,\bm e)$   and $(v_i,\bm e')$, the vectors $\bm e $ and $\bm e'$ are the same except in the positions $h_1,\ldots, h_c$, where $\bm e$ has a $\star$.
 So we have the following relationship for the extrusion directions of $(w_j,\bm e)$   and  $(v_i,\bm e')$ induced by $\gamma$:  $$\tau_{\gamma}((v_i,\bm e'))=\tau_{\gamma}((w_j,\bm e))+\sum_{z\in Z_1}\tau_z-\sum_{z'\in Z_2}\tau_{z'},$$
 where $Z_1$ is the set of indices in $\{1,\ldots, t\}$ for which $\bm e$ has a $\star$, $\bm e' $ has a $0$, and $\gamma$ has a $1$, and   $Z_2$ is the set of indices in $\{1,\ldots, t\}$ for which $\bm e $ has a $\star$, $\bm e' $ has a $1$, and $\gamma$ also has a $1$.  
So we no longer have that $\tau_\gamma((w_j,\bm e))=\tau_\gamma((v_i,\bm e'))$, and hence we do not obtain the desired equality in the  part of the proof of Theorem~\ref{thm:blockph} that deals with point-hyperplane distance constraints (see the final part of the proof for type 2 ($\mathcal{PH}$) edges).

However, we can work around this issue by ``pinning'' hyperplanes in such a way that they are no longer able to rotate.

For a point-line framework with $t$-fold extrusion symmetry in the plane, for example, we may pin one of the lines, say $\ell$, that lies along an extrusion direction by deleting  all three columns for  $\ell$, as well as the row for $\ell$ corresponding to the normalisation constraint (\ref{eq:a_inf_euc}) from the  rigidity matrix. Moreover, for each line that is parallel to $\ell$ we may delete the two columns and one row for the normal vector, but keep the third column, so that each line only retains one degree of freedom --  a parallel displacement. With this pinning in place, the rows in the rigidity matrix that correspond to the parallel constraints of these lines are redundant and can be removed. The resulting pinned point-line framework has ($1$-fold) extrusion symmetry in the direction of $\ell$, and it is easy to see that its rigidity matrix can be block-decomposed as described in Corollary~\ref{cor:blockph}. This is because  the issue in the proof of Theorem~\ref{thm:blockph} described above  has been resolved, as the problematic term involving $\dot{a}_j$ no longer exists for any line $\ell_j$ that lies along the extrusion direction. We will illustrate this in Example~\ref{ex:3}.

A similar pinning procedure can also be employed for point-hyperplane frameworks in higher dimensions, as illustrated in Example~\ref{ex:4}.

\section{Symmetry-adapted mobility counts} \label{sec:fowler-guest}

\subsection{Character counts for bar-joint frameworks}\label{sec:fowler-guest1}

 A map $\omega:E\to \mathbb{R}$ is called a \emph{self-stress} of $(G,p)$ if $$\sum_{j: \{i,j\}\in E} \omega_{ij}(p_i-p_j)=0,$$ where $\omega_{ij}=\omega(\{i,j\})$ for all $\{i,j\}\in E$. In other words, a self-stress is a linear dependence of the rows of the rigidity matrix of $(G,p)$, in the sense that $\omega^TR(G,p)=0$. If a framework does not have any non-zero self-stress, then it is said to be \emph{independent}. A framework $(G,p)$ is \emph{isostatic} if it is infinitesimally rigid and independent.

In the following we assume that $p$ affinely spans at least $\mathbb{R}^{d-1}$.
We let $\Omega(G,p)$ be the space of self-stresses of $(G,p)$, and $\mathcal{F}(G,p)$ be the space of non-trivial infinitesimal motions (or infinitesimal flexes) of $(G,p)$. Moreover, we denote $s=\textrm{dim } \big(\Omega(G,p)\big)$ and $m=\textrm{dim } \big(\mathcal F(G,p)\big)$. It was shown by Maxwell that $$m-s=d|V|-|E|-\binom{d+1}{2}.$$
In particular, if $(G,p)$ is isostatic, then $m=s=0$ and we must have $|E|=d|V|-\binom{d+1}{2}$. Our goal in this section is to refine Maxwell's formula for a  framework $(G,p)$ with $t$-fold extrusion symmetry  using the results from Section~\ref{subsec:blockdecom}. 

We first recall another basic definition from group representation theory. Let $\rho:\Gamma\to GL(X)$ be a linear representation of a group $\Gamma$ and let $Y$ be a $\rho$-invariant subspace of $X$. If for all $\gamma\in \Gamma$, we restrict the automorphism $\rho(\gamma)$ of $X$ to the subspace $Y$, then we obtain a new linear representation $\rho^{(Y)}$ of $\Gamma$ with representation space $Y$. $\rho^{(Y)}$ is said to be a \emph{subrepresentation} of $\rho$.

\begin{prop}\label{proptrrt}
    Let $(G,p)$ be a  $d$-dimensional framework affinely spanning at least $\mathbb{R}^{d-1}$ with $t$-fold extrusion symmetry. Then the space of infinitesimal translations $\mathcal{T}$ of $(G,p)$ (i.e. the space of trivial infinitesimal motions $\dot p:V\to \mathbb{R}^d$ of $(G,p)$ satisfying $\dot p_i=b$ for some $b\in \mathbb{R}^d$ and all $i\in V$)  is a $(P_V\otimes I_d)$-invariant subspace. On the other hand, the space of infinitesimal rotations $\mathcal{R}$ of $(G,p)$ is not a $(P_V\otimes I_d)$-invariant subspace. 
\end{prop}
\begin{proof} By definition  of the external representation, every element in $\mathcal{T}$ is clearly fully-symmetric (recall Definition~\ref{def:sy}).
%By definition of the external representation, it is easy to see that the space of infinitesimal translations $\mathcal{T}$ of $(G,p)$ (i.e. the space of trivial infinitesimal motions $\dot p:V\to \mathbb{R}^d$ of $(G,p)$ satisfying $\dot p_i=t$ for some $t\in \mathbb{R}^d$ and all $i\in V$)  is a $(P_V\otimes I_d)$-invariant subspace.  In fact, by definition, every element in $\mathcal{T}$ is clearly fully-symmetric (recall Definition~\ref{def:sy}). Thus we can form the subrepresentation  $(P_V\otimes I_d)^{(\mathcal{T})}$ with representation space $\mathcal{T}$. 

To see that $\mathcal{R}$ is not a $(P_V\otimes I_d)$-invariant subspace, let $\be_i$ be the $i$-th canonical basis vector of $\mathbb{R}^d$ and consider the basis $\{r_{ij}|\,1\leq i<j\leq d\}$ of $\mathcal{R}$, where for $1\leq i<j\leq d$, $r_{ij}:V\to \mathbb{R}^{d}$ is defined by $r_{ij}(k)=(p_{k})_{i}\be_{j}-(p_{k})_{j}\be_{i}$ for all $k=1,\ldots, |V|$. Since translating a framework does not affect its infinitesimal rigidity properties, we may assume without loss of generality that $p_1$ is the origin and hence every infinitesimal rotation in $\mathcal R$ maps the vertex $1$ to the zero vector (i.e. fixes $p_1$). However, by the definition of the external representation, for every nontrivial $\gamma\in \mathbb{Z}_2^t$, there clearly exists an $r\in \mathcal{R}$ so that $r':=(P_V\otimes I_d)(\gamma)(r)$ satisfies $r'(1)\neq 0$, and hence $r'\notin \mathcal R$. In fact, we just need to choose $r$ so that its rotational axis does not include the extrusion direction of the vertex $1$ induced by $\gamma$.
%On the other hand, the space of infinitesimal rotations $\mathcal{R}$ of $(G,p)$ is not a $(P_V\otimes I_d)$-invariant subspace. To see this, let $e_i$ be the $i$-th canonical basis vector of $\mathbb{R}^d$ and consider the basis $\{r_{ij}|\,1\leq i<j\leq d\}$ of $\mathcal{R}$, where for $1\leq i<j\leq d$, $r_{ij}:V\to \mathbb{R}^{d}$ is defined by $r_{ij}(k)=(p_{k})_{i}e_{j}-(p_{k})_{j}e_{i}$ for all $k=1,\ldots, |V|$. Since translating a framework does not affect its infinitesimal rigidity properties, we may assume without loss of generality that $p_1$ is the origin and hence every infinitesimal rotation in $\mathcal R$ maps the vertex $1$ to the zero vector (i.e. fixes $p_1$). However, by the definition of the external representation, for every nontrivial $\gamma\in \mathbb{Z}_2^t$, there clearly exists an $r\in \mathcal{R}$ so that $r':=(P_V\otimes I_d)(\gamma)(r)$ satisfies $r'(1)\neq 0$, and hence $r'\notin \mathcal R$. In fact, we just need to choose $r$ so that its rotational axis does not include the extrusion direction of the vertex $1$ induced by $\gamma$.
\end{proof}

By Proposition~\ref{proptrrt},  we can form the subrepresentation  $(P_V\otimes I_d)^{(\mathcal{T})}$ with representation space $\mathcal{T}$. 

Since  the space of infinitesimal rotations $\mathcal{R}$ of $(G,p)$ is not a $(P_V\otimes I_d)$-invariant subspace, it follows that the space $\mathcal{J}$ of trivial infinitesimal motions of $(G,p)$ is also not $(P_V\otimes I_d)$-invariant, since $\mathcal{J}=\mathcal{T}\oplus \mathcal{R}$. By Theorem~\ref{thm:block}, the space $\mathcal{M}(G,p)$ of \emph{all} infinitesimal motions of $(G,p)$, however, is a $(P_V\otimes I_d)$-invariant subspace, since it is the kernel of the rigidity matrix of $(G,p)$. So, since $$\mathcal{M}(G,p)= \mathcal{J}\oplus \mathcal{F}(G,p)$$ it follows that $\mathcal{F}(G,p)$ is non-trivial. In other words, as previously observed (recall Remark~\ref{rem:rotmot}), a $d$-dimensional bar-joint framework with $t$-fold extrusion symmetry whose points  affinely span at least $\mathbb{R}^{d-1}$ must always have an infinitesimal flex. In particular, it can never be isotatic.

Similar to the approach in \cite{FGsymmax,BS2}, using the block-decomposition of the rigidity matrix and some basic character theory, we may obtain further information about the infinitesimal flexibility and the self-stresses  of $(G,p)$ via a simple symmetry-adapted count as follows.

The \emph{character} of a linear representation $\rho:\Gamma \to GL(X)$ is the row vector $\chi(\rho)$ whose $i$-th component is the trace of $\rho(\gamma_i)$   for some fixed ordering  of the elements of $\Gamma$.  Then it is well-known from group representation theory that if  $\rho= \alpha_{0}\rho_0\oplus\ldots\oplus \alpha_{r-1}\rho_{r-1}$, where $\alpha_{0},\ldots,\alpha_{r-1}\in \mathbb{N}\cup {\{0\}}$, then for the character of $\rho$ we have $\chi(\rho)= \alpha_{0}\chi(\rho_{0})+\ldots+ \alpha_{r-1}\chi(\rho_{r-1}).$

Now, recall that  $\Gamma=\mathbb{Z}_2^t$  has $r=2^t$ irreducible (one-dimensional) representations (one for each $\gamma\in \Gamma$), defined by
\begin{align*} \label{eq:abelian_rho}
\rho_{{\gamma}}:\Gamma&\rightarrow \{1,-1\}  \nonumber\\
{\gamma'}&\mapsto (-1)^{x'_1x_1+x'_2x_2 + \ldots + x'_tx_t},
\end{align*}
where $\gamma'=(x'_1,\ldots, x'_t)$.
By  (\ref{irrep}), we have $$\chi(P_V\otimes I_d)= \lambda_{0}\chi(\rho_{0})+\ldots+ \lambda_{r-1}\chi(\rho_{r-1}),$$ and by (\ref{irrepHi}) we have
$$\chi(P'_E)= \mu_{0}\chi(\rho_{0})+\ldots + \mu_{r-1}\chi(\rho_{r-1}).$$ Suppose further that 
$$\chi(P_V\otimes I_d)^{(\mathcal{T})}= \nu_{0}\chi(\rho_{0})+\ldots+\nu_{r-1}\chi(\rho_{r-1}).$$ 

Then the following result holds:
\begin{thm}\label{thm:symcount}
    Let $(G,p)$ be a $d$-dimensional bar-joint framework affinely spanning  at least $\mathbb{R}^{d-1}$ with $t$-fold extrusion symmetry. Then, with the notation above, we have the following.
    \begin{itemize}
    \item[(i)] If $\lambda_{\kappa}-\nu_{\kappa}>\mu_{\kappa}$ then $(G,p)$ has an infinitesimal motion that is symmetric with respect to $\rho_{\kappa}$. If $d=2$, this motion must be non-trivial.
    \item[(ii)] If $\lambda_{\kappa}-\nu_{\kappa}<\mu_{\kappa}$, then  $(G,p)$ has a non-zero self-stress that is symmetric with respect to $\rho_{\kappa}$. 
    \end{itemize}
\end{thm}
\begin{proof}
    Since  each irreducible representation of $\Gamma=\mathbb{Z}_2^t$ has a representation space of dimension $1$, for each $0\leq \kappa\leq r-1$, the numbers $\lambda_{\kappa}$, $\mu_{\kappa}$ and $\nu_{\kappa}$ are equal to the dimensions of the  $(P_V\otimes I_d)$-invariant, $P'_{E}$-invariant and $(P_V\otimes I_d)^{(\mathcal{T})}$-invariant vector spaces  corresponding to $\rho_{\kappa}$, respectively. In other words, $\lambda_{\kappa}$ is the dimension of the space of $\rho_{\kappa}$-symmetric vectors in $\mathbb{R}^{d|V|}$, $\mu_{\kappa}$ is the dimension of the space of $\rho_{\kappa}$-symmetric vectors in $\mathbb{R}^{|E|}$ and $\nu_{\kappa}$ is the dimension of the space of $\rho_{\kappa}$-symmetric infinitesimal translations of $(G,p)$. Thus, the   block matrix $\widetilde{R}_{\kappa}(G,p)$ in Corollary~\ref{cor:block} is of size $\mu_{\kappa}\times \lambda_{\kappa}$ and if $\lambda_{\kappa}-\nu_{\kappa}>\mu_{\kappa}$, then there is a non-zero element in the kernel of  $\widetilde{R}_{\kappa}(G,p)$, which is either an infinitesimal rotation or a non-trivial infinitesimal motion. 

    If $(G,p)$ is a framework in $\mathbb{R}^2$, this infinitesimal motion must be non-trivial, since in this case there is only a $1$-dimensional space of infinitesimal rotations, and  the vectors in this space are not $\rho_{\kappa}$-symmetric for any $\kappa$, by Proposition~\ref{proptrrt}.
(For $\mathbb{R}^d$ with $d\geq 3$, this is not true in general, so a more careful analysis of the detected infinitesimal motions is required.)

    Similarly, if $\lambda_{\kappa}-\nu_{\kappa}<\mu_{\kappa}$, then we may conclude that $(G,p)$ has a non-zero self-stress that is symmetric with respect to $\rho_{\kappa}$.    
\end{proof}

%Then, if $\lambda_{\kappa}-\nu_{\kappa}>\mu_{\kappa}$, we may conclude, by Corollary~\ref{cor:block}, that $(G,p)$ has an infinitesimal motion that is symmetric with respect to $\rho_{\kappa}$. 

%If $(G,p)$ is a framework in $\mathbb{R}^2$, this infinitesimal motion must be non-trivial, since none of the infinitesimal rotations are $\rho_{\kappa}$-symmetric for any $\kappa$ in this case. (For $\mathbb{R}^d$ with $d\geq 3$, this is not true in general, so a more careful analysis of the detected infinitesimal motions is required.)

Note that if we detect an infinitesimal flex of symmetry $\rho_{\kappa}$ and another infinitesimal flex of symmetry $\rho_{\kappa'}$, then these flexes may actually be related by an infinitesimal rotation, so that only a one-dimensional space of infinitesimal flexes is detected. (Recall Figure~\ref{fig:simple}.)

%If $\lambda_{\kappa}-\nu_{\kappa}<\mu_{\kappa}$, then we may conclude that $(G,p)$ has a non-zero self-stress that is symmetric with respect to $\rho_{\kappa}$. 

We illustrate the counting method given in Theorem~\ref{thm:symcount} by means of some examples.

\begin{example}\label{ex:1} Consider the framework $(G,p)$ with extrusion symmetry in $\mathbb{R}^2$ shown in Figure~\ref{fig:ex1maxwell}, where $G=K_3\square K_2$. The group $\mathbb{Z}_2=\{0,1\}$ has two irreducible representations $\rho_0$ and $\rho_1$ of dimension 1 whose characters are $\chi(\rho_0)=(1,1)$ and $\chi(\rho_1)=(1,-1)$. Note that $G$ has six vertices and nine edges. Exactly three edges of $G$ are fixed by the non-trivial element $1$ of $\mathbb{Z}_2$, and hence make a contribution of $-1$ to the trace of $P'_E(1)$. Moreover, as mentioned earlier, the infinitesimal translations are all fully-symmetric and form a vector space of dimension $2$. Thus, we have the following character table for $(G,p)$:

\begin{table}[htp]
\begin{center}
\begin{tabular}{l||l|l}
$\mathbb{Z}_{2}$   &   $0$  &  $1$   \\\hline\hline
$\chi(P_V)$  &    6  &  0 \\\hline
$\chi(P_V\otimes I_2)$  &    12  &  0\\\hline
$\chi(P'_E)$  &    9  &  -3 \\\hline
$\chi(P_V\otimes I_2)^{(\mathcal{T})}$  &    2  &  2 \\
\end{tabular}
\end{center}
\end{table}

It follows that $$\chi(P_V\otimes I_2)- \chi(P_V\otimes I_2)^{(\mathcal{T})}= (10,-2)= 4\chi(\rho_0)+6\chi(\rho_1) $$
and 
$$\chi(P'_E)= (9,-3)= 3\chi(\rho_0)+6\chi(\rho_1). $$

\begin{figure}[htp]
\begin{center}
\begin{tikzpicture}[very thick,scale=0.7]
\tikzstyle{every node}=[circle, draw=black, fill=white, inner sep=0pt, minimum width=5pt];
        \path (0,0) node (p1) {} ;
        \path (3,0) node (p2) {} ;
        \path (1.5,1) node (p3) {} ;
        
         \path (0,2) node (p11) {} ;
        \path (3,2) node (p22) {} ;
        \path (1.5,3) node (p33) {} ;
        
        \draw (p1)  --  (p2);
         \draw (p1)  --  (p3);
        \draw (p3)  --  (p2);
        \draw (p11)  --  (p22);
         \draw (p11)  --  (p33);
        \draw (p33)  --  (p22);
        \draw (p1)  --  (p11);
         \draw (p2)  --  (p22);
        \draw (p3)  --  (p33);
        
         \draw[gray,->] (p1)  --  (0.2,-0.6);
        \draw[gray,->] (p2)  --  (3.2,0.6);
        \draw[gray,->] (p3)  --  (1,1);
        
         \draw[gray,->] (p11)  --  (0.2,1.4);
        \draw[gray,->] (p22)  --  (3.2,2.6);
        \draw[gray,->] (p33)  --  (1,3);
        
        \node [draw=white, fill=white] (a) at (1.5,-0.7) {(a)};
      \end{tikzpicture}
      \hspace{0.5cm}
\begin{tikzpicture}[very thick,scale=0.7]
\tikzstyle{every node}=[circle, draw=black, fill=white, inner sep=0pt, minimum width=5pt];
        \path (0,0) node (p1) {} ;
        \path (3,0) node (p2) {} ;
        \path (1.5,1) node (p3) {} ;
        
         \path (0,2) node (p11) {} ;
        \path (3,2) node (p22) {} ;
        \path (1.5,3) node (p33) {} ;
        
        \draw (p1)  --  (p2);
         \draw (p1)  --  (p3);
        \draw (p3)  --  (p2);
        \draw (p11)  --  (p22);
         \draw (p11)  --  (p33);
        \draw (p33)  --  (p22);
        \draw (p1)  --  (p11);
         \draw (p2)  --  (p22);
        \draw (p3)  --  (p33);
        
         \draw[gray,->] (p1)  --  (0.2,-0.6);
        \draw[gray,->] (p2)  --  (3.2,0.6);
        \draw[gray,->] (p3)  --  (1,1);
        
           \draw[gray,->] (p11)  --  (-0.7,1.3);
        \draw[gray,->] (p22)  --  (2.3,2.7);
        \draw[gray,->] (p33)  --  (0.4,3);
        
        \node [draw=white, fill=white] (a) at (1.5,-0.7) 
               {(b)};
               \filldraw[black] (1.5,0.5) circle (1pt); 
      \end{tikzpicture}      
      \hspace{0.5cm}
   \begin{tikzpicture}[very thick,scale=0.7]
\tikzstyle{every node}=[circle, draw=black, fill=white, inner sep=0pt, minimum width=5pt];
        \path (0,0) node (p1) {} ;
        \path (3,0) node (p2) {} ;
        \path (1.5,1) node (p3) {} ;
        
         \path (0,2) node (p11) {} ;
        \path (3,2) node (p22) {} ;
        \path (1.5,3) node (p33) {} ;
        
        \draw (p1)  --  (p2);
         \draw (p1)  --  (p3);
        \draw (p3)  --  (p2);
        \draw (p11)  --  (p22);
         \draw (p11)  --  (p33);
        \draw (p33)  --  (p22);
        \draw (p1)  --  (p11);
         \draw (p2)  --  (p22);
        \draw (p3)  --  (p33);
        
         \draw[gray,->] (p11)  --  (0.9,2);
        \draw[gray,->] (p22)  --  (3.9,2);
        \draw[gray,->] (p33)  --  (2.4,3);
        \node [draw=white, fill=white] (a) at (1.5,-0.7) {(c)};
      \end{tikzpicture}   
            \hspace{0.5cm}
   \begin{tikzpicture}[very thick,scale=0.7]
\tikzstyle{every node}=[circle, draw=black, fill=white, inner sep=0pt, minimum width=5pt];
        \path (0,0) node (p1) {} ;
        \path (3,0) node (p2) {} ;
        \path (1.5,1) node (p3) {} ;
        
         \path (0,2) node (p11) {} ;
        \path (3,2) node (p22) {} ;
        \path (1.5,3) node (p33) {} ;
        
        \draw (p1)  --  (p2);
         \draw (p1)  --  (p3);
        \draw (p3)  --  (p2);
        \draw (p11)  --  (p22);
         \draw (p11)  --  (p33);
        \draw (p33)  --  (p22);
        \draw (p1)  --  (p11);
         \draw (p2)  --  (p22);
        \draw (p3)  --  (p33);
        
         \draw[gray,->] (p11)  --  (0.9,2);
        \draw[gray,->] (p22)  --  (3.9,2);
        \draw[gray,->] (p33)  --  (2.4,3);
          \draw[gray,->] (p1)  --  (-0.9,0);
        \draw[gray,->] (p2)  --  (2.1,0);
        \draw[gray,->] (p3)  --  (0.6,1);
        \node [draw=white, fill=white] (a) at (1.5,-0.7) {(d)};
        \end{tikzpicture}
\end{center}
\caption{(a) A fully-symmetric infinitesimal flex of the framework $(G,p)$ from Figure~\ref{fig:ex1}. This infinitesimal flex may be written as the sum of an infinitesimal rotation about the indicated point (b) and an infinitesimal flex fixing the bottom triangle (c), neither of which is symmetric with respect to any irreducible representation of $\mathbb{Z}_2$. By adding a suitable infinitesimal rotation to (a) we may also obtain the anti-symmetric infinitesimal flex shown in (d). That this change of symmetry type for the infinitesimal flex is  possible is a consequence of Proposition~\ref{proptrrt}: the $1$-dimensional space of infinitesimal rotations in the plane is not  $(P_V\otimes I_d)$-invariant.} %, because if the $1$-dimensional space of infinitesimal rotations in the plane, $\mathcal{R}$, was $(P_V\otimes I_d)$-invariant, then all vectors in $\mathcal{R}$ would be either fully- or anti-symmetric.}
\label{fig:ex1maxwell}
\end{figure}
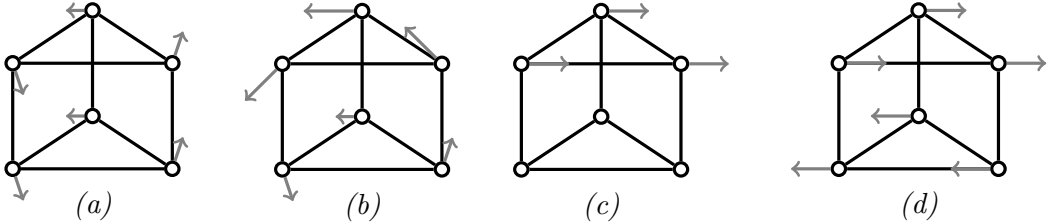

Thus, by comparing the coefficients of $\chi(\rho_0)$, we conclude from Theorem~\ref{thm:symcount}  that $(G,p)$ has %two translations plus 
a fully-symmetric infinitesimal motion.   We will show in Section~\ref{sec:finiteflex} that this infinitesimal motion extends to a finite motion.  (Note that  $(G,p)$ also has a trivial infinitesimal  motion corresponding to a rotation,  which is neither fully-symmetric nor anti-symmetric.) 
Maxwell's original rule would not have detected the infinitesimal flex, because $|E(G)|=2|V(G)|-3=9$.

The decompositions of $\mathbb{R}^{2|V|}$ and $\mathbb{R}^{|E|}$ into invariant subspaces corresponding to $\rho_0$ and $\rho_1$ give a block-diagonalisation of the rigidity matrix $\widetilde R(G,p)$ when it is written in a symmetry-adapted coordinate system. The  decompositions $\chi(P_V\otimes I_2)=6\chi(\rho_0)+6\chi(\rho_1)$ and  $\chi(P'_E)=3\chi(\rho_0)+6\chi(\rho_1)$ indicate that  $\widetilde R(G,p)$ has a block-diagonalised form with a block $\widetilde R_0(G,p)$ for $\rho_0$ of size $3 \times 6$ and a block $\widetilde R_1(G,p)$ for $\rho_1$ of size $6 \times 6$.

\begin{figure}[htp]
\begin{center}
   \begin{tikzpicture}[very thick,scale=0.7]
\tikzstyle{every node}=[circle, draw=black, fill=white, inner sep=0pt, minimum width=5pt];
        \path (0,0) node[fill=black] (p1) {} ;
        \path (3,0) node (p2) {} ;
        \path (1.5,1) node (p3) {} ;
        
         \path (0,2) node[fill=black] (p11) {} ;
        \path (3,2) node (p22) {} ;
        \path (1.5,3) node (p33) {} ;
        
        \draw (p1)  --  (p2);
         \draw (p1)  --  (p3);
        \draw (p3)  --  (p2);
        \draw (p11)  --  (p22);
         \draw (p11)  --  (p33);
        \draw (p33)  --  (p22);
       % \draw (p1)  --  (p11);
         \draw (p2)  --  (p22);
        \draw (p3)  --  (p33);        
\node [draw=white, fill=white] (a) at (1.5,-0.7) {(a)};
      \end{tikzpicture} 
      \hspace{2cm}
         \begin{tikzpicture}[very thick,scale=0.7]
\tikzstyle{every node}=[circle, draw=black, fill=white, inner sep=0pt, minimum width=5pt];
        \path (0,0) node[fill=black] (p1) {} ;
        \path (3,0) node (p2) {} ;
        \path (1.5,1) node (p3) {} ;
        
         \path (0,2) node[fill=black] (p11) {} ;
        \path (3,2) node (p22) {} ;
        \path (1.5,3) node (p33) {} ;
        
        \draw (p1)  --  (p2);
         \draw (p1)  --  (p3);
        \draw (p3)  --  (p2);
        \draw (p11)  --  (p22);
         \draw (p11)  --  (p33);
        \draw (p33)  --  (p22);
       % \draw (p1)  --  (p11);
         \draw (p2)  --  (p22);
        \draw (p3)  --  (p33);

        \draw[gray,->] (p22)  --  (3,3);
        \draw[gray,->] (p33)  --  (1.2,3.6);
         
        \draw[gray,->] (p2)  --  (3,1);
        \draw[gray,->] (p3)  --  (1.2,1.6);
\node [draw=white, fill=white] (a) at (1.5,-0.7) {(b)};
      \end{tikzpicture} 
\end{center}
\caption{(a) A pinned version of the framework in Figure~\ref{fig:ex1maxwell}, where pinned vertices are shown in black. Its fully-symmetric infinitesimal flex is shown in (b).}
\label{fig:ex1maxwellpin}
\end{figure}
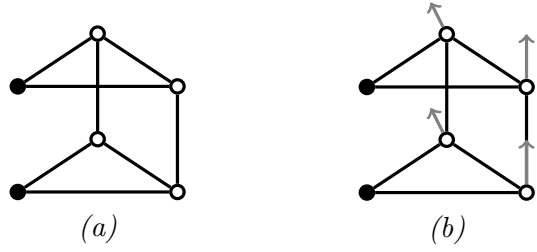

We get a more direct identification of the symmetric flex if we analyse a pinned version of the framework in Figure~\ref{fig:ex1maxwell} (as shown in Figure~\ref{fig:ex1maxwellpin}).  If we pin the two leftmost vertices of the framework  in Figure~\ref{fig:ex1maxwell}, so that they can no longer be displaced (i.e. we remove the four columns corresponding to these two vertices from the rigidity matrix) and we also delete the edge between them, then overall 3 degrees of freedom have been removed and the resulting pinned framework still has extrusion symmetry. Since the pinned framework no longer has trivial infinitesimal motions, the character counts simplify as shown in the table below. 
\begin{table}[htp]
\begin{center}
\begin{tabular}{l||l|l}
$\mathbb{Z}_{2}$   &   $0$  &  $1$   \\\hline\hline
$\chi(P_V)$  &    4  &  0 \\\hline
$\chi(P_V\otimes I_2)$  &    8  &  0\\\hline
$\chi(P'_E)$  &    8  &  -2 \\\hline
\end{tabular}
\end{center}
\end{table}\label{tab:pin}

It follows from the table that $\chi(P_V\otimes I_2)- \chi(P'_E)= (0,2)= \chi(\rho_0)-\chi(\rho_1)$ and hence one fully-symmetric infinitesimal flex (and one anti-symmetric self-stress) is detected.  The pinning preserves the block-diagonalisation of $\widetilde R(G,p)$ and removes two columns and a row from $\widetilde R_1(G,p)$ and two columns from $\widetilde R_0(G,p)$. There is an anti-symmetric self-stress due to the column deficiency in $\widetilde R_1(G,p)$  and a fully-symmetric flex  due to the row deficiency in $\widetilde R_0(G,p)$.  We note that the block-diagonalisation is preserved by any change in the coordinates which preserves the extrusion symmetry and that the self-stress can be relieved by removing any of the edges,  but only the removal of one of the edges that are fixed under extrusion symmetry %$v_{i0}v_{i1}$ which 
(whose rows are in $\widetilde R_1(G,p)$) will preserve the extrusion symmetry and hence the block-diagonalisation. 
\end{example}

\begin{example}\label{ex:2} Consider the framework $(G,p)$ with $2$-fold extrusion symmetry in $\mathbb{R}^2$ shown in Figure~\ref{fig:ex2}, where $G=K_3\square K_2^{\square t}$. The group $\mathbb{Z}_2^2$ has four irreducible representations of dimension 1 whose characters are $\chi(\rho_{(0,0)})=(1,1,1,1)$, $\chi(\rho_{(1,0)})=(1,-1,1,-1)$, $\chi(\rho_{(0,1)})=(1,1,-1,-1)$ and $\chi(\rho_{(1,1)})=(1,-1,-1,1)$.
We have the following character table for $(G,p)$:

\begin{table}[htp]
\begin{center}
\begin{tabular}{l||l|l|l|l}
$\mathbb{Z}_{2}^2$   &   $(0,0)$  &  $(1,0)$ & $(0,1)$ & $(1,1)$  \\\hline\hline
$\chi(P_V)$  &    12  &  0 & 0 & 0 \\\hline
$\chi(P_V\otimes I_2)$  &    24  &  0& 0 & 0\\\hline
$\chi(P'_E)$  &    24  &  -6 &  -6 & 0 \\\hline
$\chi(P_V\otimes I_2)^{(\mathcal{T})}$  &    2  &  2 & 2 & 2 \\
\end{tabular}
\end{center}
\end{table}

It follows that $$\chi(P_V\otimes I_2)- \chi(P_V\otimes I_2)^{(\mathcal{T})}= (22,-2,-2,-2)=  4 \chi(\rho_{(0,0)})+ 6 \chi(\rho_{(1,0)})+ 6 \chi(\rho_{(0,1)})+ 6\chi(\rho_{(1,1)})$$
and 
$$\chi(P'_E)= (24,-6,-6,0)= 3\chi(\rho_{(0,0)}) + 6\chi(\rho_{(1,0)}) + 6 \chi(\rho_{(0,1)})+ 9\chi(\rho_{(1,1)}). $$

Thus, by comparing the coefficients of the characters of the irreducible representations, we may conclude from   Theorem~\ref{thm:symcount} that $(G,p)$ has a fully-symmetric infinitesimal flex and a three-dimensional space of self-stresses that are symmetric with respect to $\rho_{(1,1)}$. The fully-symmetric flex can again be shown to be finite using the results in Section~\ref{sec:finiteflex}.

As in the previous example, we may again analyse pinned versions of this framework. If we pin $p_100$ and $p_110$, for example, and remove the edge between those two vertices, then the framework still has extrusion symmetry in the direction of $\tau_1$. Similarly,  if we pin $p_100$ and $p_1 01$ and remove the edge between those two vertices, then the framework has extrusion symmetry in the direction of $\tau_2$. In both cases, the character counts for $\mathbb{Z}_2$ detect a fully-symmetric infinitesimal flex (and four anti-symmetric self-stresses) since $\chi(P_V\otimes I_2)- \chi(P'_E)= (-3,5)= \chi(\rho_0)-4\chi(\rho_1)$.
\end{example}

We conclude this subsection with some general remarks.

\begin{figure}[htp]
\begin{center}
\begin{tikzpicture}[very thick,scale=0.5]
\tikzstyle{every node}=[circle, draw=black, fill=white, inner sep=0pt, minimum width=5pt];
        \path (0,-0.3) node (p1) {} ;
        \path (3,-0.3) node (p2) {} ;
        \path (1.5,0.7) node (p3) {} ;
        
         \path (0,2.3) node (p11) {} ;
        \path (3,2.3) node (p22) {} ;
        \path (1.5,3.3) node (p33) {} ;
        
        \path (5,1) node (p111) {} ;
        \path (8,1) node (p222) {} ;
        \path (6.5,2) node (p333) {} ;

        \draw (p1)  --  (p2);
         \draw (p1)  --  (p3);
        \draw (p3)  --  (p2);
        \draw (p11)  --  (p22);
         \draw (p11)  --  (p33);
        \draw (p33)  --  (p22);
         \draw (p111)  --  (p222);
         \draw (p111)  --  (p333);
        \draw (p333)  --  (p222);
        \draw (p1)  --  (p11);
         \draw (p2)  --  (p22);
        \draw (p3)  --  (p33);
        \draw (p1)  --  (p111);
         \draw (p2)  --  (p222);
        \draw (p3)  --  (p333);
         \draw (p11)  --  (p111);
         \draw (p22)  --  (p222);
        \draw (p33)  --  (p333);
        
         \draw[gray,->] (p1)  --  (0.2,-0.9);
        \draw[gray,->] (p2)  --  (3.2,0.3);
        \draw[gray,->] (p3)  --  (0.8,0.7);
        
          \draw[gray,->] (p11)  --  (0.2,1.7);
        \draw[gray,->] (p22)  --  (3.2,2.9);
        \draw[gray,->] (p33)  --  (0.8,3.3);
        
           \draw[gray,->] (p111)  --  (5.2,0.4);
        \draw[gray,->] (p222)  --  (8.2,1.6);
        \draw[gray,->] (p333)  --  (5.8,2);
        
        \node [draw=white, fill=white] (a) at (3,-1.4) {(a)};
      \end{tikzpicture}
         \hspace{2cm}
      \begin{tikzpicture}[very thick,scale=0.5]
\tikzstyle{every node}=[circle, draw=black, fill=white, inner sep=0pt, minimum width=5pt];
        \path (-1,0) node (p1) {} ;
        \path (1,0) node (p2) {} ;
        \path (0,1.33333) node (p3) {} ;
        
          \draw[gray,->] (p1)  --  (-0.7,-0.75);
        \draw[gray,->] (p2)  --  (1.3,0.75);
        \draw[gray,->] (p3)  --  (-0.8,1.33333);
        
         \path (-3,-0.66666) node (p11) {} ;
        \path (3,-0.66666) node (p22) {} ;
        \path (0,3.333333) node (p33) {} ;

        \draw (p1)  --  (p2);
         \draw (p1)  --  (p3);
        \draw (p3)  --  (p2);
        \draw (p11)  --  (p22);
         \draw (p11)  --  (p33);
        \draw (p33)  --  (p22);
         \draw (p1)  --  (p11);
         \draw (p2)  --  (p22);
        \draw (p3)  --  (p33);
        
            \node [draw=white, fill=white] (a) at (0,-1.4) {(b)};
      \end{tikzpicture}
   \end{center}
\caption{(a) A framework in $\mathbb{R}^2$ consisting of a cycle of extruded triangles. The infinitesimal flex shown in (a) extends to a finite flex, as we will see in Section~\ref{sec:finiteflex}. (b) A framework in $\mathbb{R}^2$ which is infinitesimally flexible, since the three extended edges connecting the two triangles meet in a point. }
\label{fig:ex3}
\end{figure}
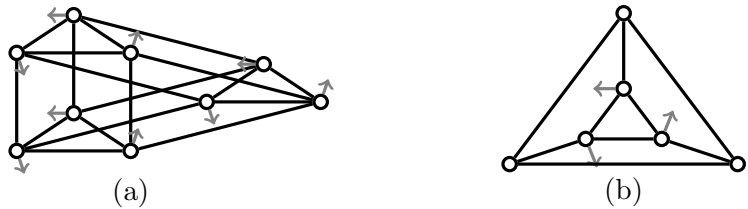

If we take $s\geq 2$ isomorphic copies of a graph $H$, say $H_1,\ldots, H_s$, where $H_{i}$ has vertex set $V(H_{i})=\{(v_j,i)|\, v_j\in V(H)\}$, and construct a framework by taking $s$ congruent  frameworks $(H_1,p_1)$,\ldots, $(H_s,p_s)$ and  adding \emph{any} edges of the from $\{(v_j,i),(v_j,i')\}$ for $v_j\in V(H)$ and $1\leq i<i'<s$, then the infinitesimal flex which rotates each $(H_i,p_i)$ (recall Remark~\ref{rem:rotmot}) is still present. This is because the velocity vectors assigned to  $(v_j,i)$ and $(v_j,i')$ are the same. The framework shown in Figure~\ref{fig:ex3}(a) gives an example. We will show in Section~\ref{sec:finiteflex} that this infinitesimal flex also extends to a finite flex.

As previously mentioned, it is well known that infinitesimal rigidity of bar-joint frameworks is invariant under projective transformations (see \cite{nsw}, for example). 
 It follows that infinitesimal flexibility transfers from a bar-joint framework $(H\square K_2,p)$ with extrusion symmetry to a modified bar-joint framework, where one copy of $H$ is realised as a dilation of the other, and the extrusion edges are no longer parallel, but they all meet in a finite point instead. See Figure~\ref{fig:ex3}(b), for example.
Note, however, that finite flexibility does not transfer in this process.

\subsection{Character counts for point-hyperplane frameworks} \label{sec:countph}

The symmetry-adpated counting method for analysing the rigidity of bar-joint frameworks given in Theorem~\ref{thm:symcount} can be extended to point-hyperplane frameworks in a straightforward fashion. (See \cite{FGO} for similar counts for CAD frameworks with point group symmetry.)

Given a point-hyperplane framework $(G,p,\ell)$, a map $\omega:E\cup V\to \mathbb{R}$ is called a \emph{self-stress} of $(G,p,\ell)$ if it is a linear dependence of the rows of the point-hyperplane  rigidity matrix of $(G,p,\ell)$, i.e., if $\omega^TR(G,p,\ell)=0$. If $(G,p,\ell)$ does not have any non-zero self-stress, then it is said to be \emph{independent}. A framework $(G,p,\ell)$ is \emph{isostatic} if it is infinitesimally rigid and independent.

In the following we assume that $|V|\geq d$ and the points and hyperplanes affinely span at least $\mathbb{R}^{d-1}$.
We let $\Omega(G,p,\ell)$ be the space of self-stresses of $(G,p,\ell)$, and $\mathcal{F}(G,p,\ell)$ be the space of non-trivial infinitesimal motions of $(G,p,\ell)$. Moreover, we denote $s=\textrm{dim } \big(\Omega(G,p,\ell)\big)$ and $m=\textrm{dim } \big(\mathcal F(G,p,\ell)\big)$. We have the Maxwell-type count \begin{eqnarray*}m-s&=&d|V_{\mathcal{P}}|+(d+1)|V_{\mathcal{H}}|-|E_{\mathcal P \mathcal P}|-|E_{\mathcal P \mathcal H}| -  |E_{\mathcal H \mathcal H\not \parallel}|  -  (d-1)|E_{\mathcal H \mathcal H\parallel}|  -|V_{\mathcal{H}}|-\binom{d+1}{2}\\& = &d|V|-|E|-  (d-2)|E_{\mathcal H \mathcal H\parallel}| -\binom{d+1}{2}.\end{eqnarray*} 
In particular, if $(G,p,\ell)$ is isostatic, then $m=s=0$ and we must have $|E|+(d-2)|E_{\mathcal H \mathcal H\parallel}| =d|V|-\binom{d+1}{2}$.

Let us now assume that $(G,p,\ell)$ is  a point-hyperplane framework in $\mathbb{R}^d$ with $t$-fold extrusion symmetry with the property that it has no edge $\{i,j\}\in E_{\mathcal P \mathcal H}$ where the hyperplane corresponding to $j$ contains an extrusion direction. 

By definition of the external representation $P_V'$, it is easy to see that the space of infinitesimal translations $\mathcal{T}$ of $(G,p,\ell)$   is a $(P_V')$-invariant subspace. (An infinitesimal translation of $(G,p,\ell)$ assigns the same vector $b$ to each point $p_i$ of $(G,p,\ell)$ and assigns the vector $(0,\ldots 0,\langle a_k,b\rangle)^T$ to each hyperplane $\ell_k=(a_k,r_k)$ of $(G,p,\ell)$.)  In fact, by definition, every element in $\mathcal{T}$ is clearly fully-symmetric.
Thus we can form the subrepresentation  $(P_V')^{(\mathcal{T})}$ with representation space $\mathcal{T}$. On the other hand, as was the case for bar-joint frameworks,   the space of infinitesimal rotations $\mathcal{R}$ of $(G,p,\ell)$ is not a $(P_V')$-invariant subspace.

By Theorem~\ref{thm:blockph}, the space $\mathcal{M}(G,p,\ell)$ of \emph{all} infinitesimal motions of $(G,p,\ell)$  is a $(P_V')$-invariant subspace, since it is the kernel of the rigidity matrix of $(G,p,\ell)$.   Unlike in the situation for bar-joint frameworks, however, the representation $P_V'$ is in general not unitary (recall the definition of $P_{V_{\mathcal H}}'$), and hence the complement space of a $(P_V')$-invariant subspace is not necessarily $(P_V')$-invariant. Thus, we cannot conclude that a $t$-fold extrusion-symmetric  point-hyperplane framework always has an infinitesimal flex. (See Figure~\ref{fig:simpleex} for an example.)

Analogous to the approach described in Section~\ref{sec:fowler-guest1} we may obtain information about the flexibility and stressability of $t$-fold extrusion-symmetric point-hyperplane frameworks by considering the characters of the representations $P_V'$, $(P_V')^{(\mathcal{T})}$, and $P_E'$. We illustrate this with the following examples. 

\begin{example}\label{ex:3}  Consider the point-line framework $(G,p,\ell)$ with $2$-fold extrusion symmetry in $\mathbb{R}^2$ shown in Figures~\ref{fig:exptlngrfw21}(a) and
\ref{fig:exptlngrfw}(d). The underlying graph $G$ is the graph $H\square_{F_1,F_2} K_2^{\square 2}$, where $F_1=\{w_1\}$ and $F_2=\{w_2\}$. (Recall Figure~\ref{fig:exptlngr}(d).) Since $|E|=14$, $ |V_{\mathcal P}|=|V_{\mathcal H}|=4$, this graph has the  Maxwell count $m-s=2|V_{\mathcal P}|+2|V_{\mathcal H}|-|E|-3=-1$ which only implies the existence of a self-stress.  However, we can use the symmetry-adapted count to reveal that the framework $(G,p,\ell)$  has an infinitesimal flex and two independent self-stresses. 

Note that while  $(G,p,\ell)$ has  $2$-fold extrusion symmetry, each pair of parallel lines lies along one of the extrusion directions. Thus, as discussed in Section~\ref{pinh}, we need to construct a pinned version of this framework to apply the symmetry-adapted counts. So we pin the line $\ell=\ell_1\star 0$; that is, we remove the corresponding three columns for $\ell=(a,r)$ and the row for the normalisation of the normal vector $a$ from the rigidity matrix. Moreover, for the parallel line $\ell'=\ell_1\star 1$, we remove the two columns for the normal vector $a'$ and the row for the normalisation of $a'$ in the rigidity matrix, but we keep the column for $r'$ so that the line can still be displaced in a parallel fashion. The parallel constraint between $\ell$ and $\ell'$ is then of course deleted, since the pinning already forces the lines  to stay parallel. The extrusion symmetry of the framework has now been reduced from $\mathbb{Z}_2^2$ to $\mathbb{Z}_2$, where the extrusion is in the direction of the pinned line. See Figure~\ref{fig:exptlngrfw21}(b). Note that the pinning has also deleted all trivial infinitesimal motions, except for the translation along the direction of the pinned line.

\begin{figure}[htp]
\begin{center}
       \begin{tikzpicture}[very thick,scale=0.7]
\tikzstyle{every node}=[circle, draw=black, fill=white, inner sep=0pt, minimum width=4pt];
     \node [draw=white, fill=white] (a) at (0.7,-0.8) {\tiny $\ell_20\star$};   
                \node [draw=white, fill=white] (a) at (-1.5,0.2) {\tiny $\ell_1\star 0$};
   \node [draw=white, fill=white] (a) at (3.8,1.7) {\tiny $\ell_21\star$};  
      \node [draw=white, fill=white] (a) at (2.7,0.2) {\tiny $\ell_1\star 1$};
  
     \node [draw=white, fill=white] (a) at (0.1,0.4) {\tiny $p_100$};
              \path (0,0) node (p1) {} ;
       
           \draw(-2,-0.5)--(6,-0.5);    
             \draw(-2,-1)--(2,3);   
             
               \draw[dotted](p1)--(0,-0.5);  
                 \draw[dotted](p1)--(-0.5,0.5);  
                           
                   \node [draw=white, fill=white] (a) at (2.3,2.3) {\tiny $p_110$};          
        \path (2,2) node (p2) {} ;
       
           \draw(-0,1.5)--(7,1.5);    
              
              \draw[dotted](p1)--(p2); 
               \draw[dotted](p2)--(2,1.5);  
                 \draw[dotted](p2)--(1.5,2.5);  
                 
               \draw[dotted](1.2,-0.5)--(3.2,1.5);

          \node [draw=white, fill=white] (a) at (4.8,0.1) {\tiny $p_101$};
              \path (4,0) node (p3) {} ;

             \draw(2,-1)--(6,3);   
             
               \draw[dotted](p3)--(4,-0.5);  
                 \draw[dotted](p3)--(3.5,0.5);  
                           
                   \node [draw=white, fill=white] (a) at (6.3,2.3) {\tiny $p_111$};          
        \path (6,2) node (p4) {} ;
       
           \draw(-0,1.5)--(3.5,1.5);    
              
              \draw[dotted](p3)--(p4); 
               \draw[dotted](p4)--(6,1.5);  
                 \draw[dotted](p4)--(5.5,2.5);  
                 
             \draw[dotted](p3)--(p1);        
                \draw[dotted](p2)--(p4); 
                
                 \draw[dotted](0.2,1.2)--(4.2,1.2);   
        \node [draw=white, fill=white] (a) at (2,-1.8) {(a)};
             \end{tikzpicture}
         \hspace{0.4cm}   
     % \begin{tikzpicture}[very thick,scale=0.7]
    %             \node [draw=white, fill=white] (a) at (0,0.8) {
%\begin{tabular}{l||l|l|l|l}
%$\mathbb{Z}_{2}^2$   &   $(0,0)$  &  $(1,0)$ & $(0,1)$ & $(1,1)$  \\\hline\hline
% $\chi(\rho_{(0,0)})$ &    1 &  1 &  1 & 1\\\hline
% $\chi(\rho_{(1,0)})$ &    1  &  -1 &  1 & -1 \\\hline
% $\chi(\rho_{(0,1)})$  &    1  &  1 & -1 & -1 \\\hline
%$\chi(\rho_{(1,1)})$  &    1  &  -1 & -1 & 1 \\
%\end{tabular}
%};
% \node [draw=white, fill=white] (a) at (0,-1.8) {(b)};
%      \end{tikzpicture}
      \begin{tikzpicture}[very thick,scale=0.7]
\tikzstyle{every node}=[circle, draw=black, fill=white, inner sep=0pt, minimum width=4pt];
  
         \node [draw=white, fill=white] (b) at (2,-0.5) {\large{$\circlearrowleft$}};
\node [draw=white, fill=white] (b) at (3,1.5) {\large{$\circlearrowleft$}};
     \node [draw=white, fill=white] (a) at (0.7,-0.8) {};   
                \node [draw=white, fill=white] (a) at (-1.3,0.2) {};
   \node [draw=white, fill=white] (a) at (3.8,1.7) {};  
      \node [draw=white, fill=white] (a) at (2.9,0.2) {};
  
     \node [draw=white, fill=white] (a) at (0.1,0.4) {};
              \path (0,0) node (p1) {} ;
       
           \draw(-2,-0.5)--(6,-0.5);    
             \draw[line width=1mm](-2,-1)--(2,3);   
             
               \draw[dotted](p1)--(0,-0.5);  
                 \draw[dotted](p1)--(-0.5,0.5);  
                           
                   \node [draw=white, fill=white] (a) at (2.3,2.3) {};          
        \path (2,2) node (p2) {} ;
       
           \draw(-0,1.5)--(7,1.5);    
              
              \draw[dotted](p1)--(p2); 
               \draw[dotted](p2)--(2,1.5);  
                 \draw[dotted](p2)--(1.5,2.5);  
                 
               \draw[dotted](1.2,-0.5)--(3.2,1.5);

          \node [draw=white, fill=white] (a) at (4.8,0.1) {};
              \path (4,0) node (p3) {} ;

             \draw(2,-1)--(6,3);   
             
               \draw[dotted](p3)--(4,-0.5);  
                 \draw[dotted](p3)--(3.5,0.5);  
                           
                   \node [draw=white, fill=white] (a) at (6.3,2.3) {};          
        \path (6,2) node (p4) {} ;
       
           \draw(-0,1.5)--(3.5,1.5);    
              
              \draw[dotted](p3)--(p4); 
               \draw[dotted](p4)--(6,1.5);  
                 \draw[dotted](p4)--(5.5,2.5);  
                 
             \draw[dotted](p3)--(p1);        
                \draw[dotted](p2)--(p4); 
                
                 \draw[dotted](0.2,1.2)--(4.2,1.2);

                 \draw[thick,gray,->] (3.7,0.7)  --  (3.4,1);

                 \draw[gray,->] (p3)  --  (4,0.8);
                 \draw[gray,->] (p4)  --  (6,2.8);

        \node [draw=white, fill=white] (a) at (2,-1.8) {(b)};
             \end{tikzpicture}
           \end{center}
     \vspace{-0.5cm}
\caption{(a) A point-line framework with $2$-fold extrusion symmetry. The framework in (b) is obtained from the one in (a) by pinning the line $\ell_1\star 0$. The pinned line is shown with increased thickness and can be considered the ``ground''. The pinned framework has $1$-fold extrusion symmetry (in the direction of the pinned line) and a fully-symmetric infinitesimal flex (indicated by gray arrows).}
\label{fig:exptlngrfw21}
\end{figure}
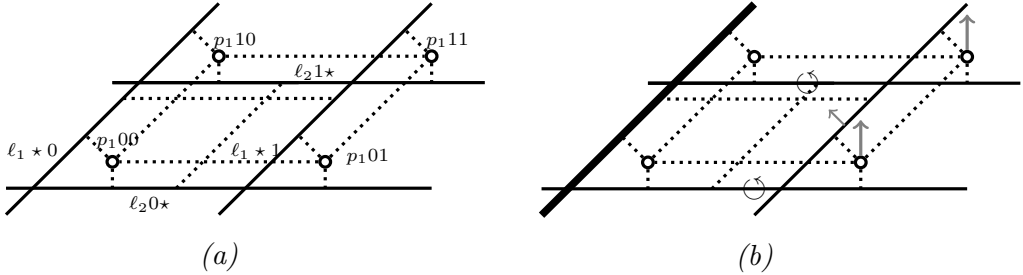

Let us apply the symmetry-adapted counts to the pinned framework in Figure~\ref{fig:exptlngrfw21}(b). None of the four vertices in   $V_{\mathcal{P}}$ are fixed by the non-trivial group element in $\mathbb{Z}_2$, and hence $\chi(P_{V_{\mathcal P}})=(4,0)$. 
The two horizontal lines each have three columns in the rigidity matrix. The fixed line $\ell$ on the left has no columns and the line $\ell'$ on the right only has one column in the rigidity matrix. Thus, the entry of $\chi(P'_{V_{\mathcal H}})$ corresponding to the identity is $3+3+1=7$. Since the two horizontal lines are images of each other under the extrusion, and the other two lines are fixed by the extrusion, the entry of $\chi(P'_{V_{\mathcal H}})$ corresponding to the non-trivial group element is $0+0+1=1$.

There are four edges in $E_{\mathcal{P}\mathcal{P}}$ and eight edges in $E_{\mathcal{P}\mathcal{H}}$. Moreover, there is one edge in  $E_{\mathcal{H}\mathcal{H}\parallel}$, namely the one that keeps the horizontal lines parallel. (Recall that the parallel constraint between $\ell$ and $\ell'$ has been removed.) Two edges in $E_{\mathcal{P}\mathcal{P}}$ are fixed by the non-identity element in $\mathbb{Z}_2$, and each of those edges contributes $-1$ to the trace of $P'_{E_{\mathcal P \mathcal P}}(1)$. None of the edges in  $E_{\mathcal{P}\mathcal{H}}$ are fixed by the non-identity element. The edge in $E_{\mathcal{H}\mathcal{H}\parallel}$ is fixed by each element in $\mathbb{Z}_2$ (with the lines being swapped under the non-trivial group element) and hence, by the definition of $P'_{E_{\mathcal H \mathcal H}\parallel}$, we have $\chi(P'_{E_{\mathcal H \mathcal H}\parallel})=(1,-1)$.  This gives  the  character counts for $(G,p,\ell)$ shown in Table~\ref{tab:expl1}.

\begin{table}[htp]
\begin{center}
\begin{tabular}{l||l|l}
$\mathbb{Z}_{2}^2$   &   $0$  &  $1$   \\\hline\hline
$\chi(P_{V_{\mathcal P}})$  &    4  &  0  \\\hline
$\chi(P_{V_{\mathcal P}}\otimes I_2)$  &    8  &  0 \\\hline
$\chi(P'_{V_{\mathcal H}})$  &    7  &  1 \\\hline
$\chi(P'_{E_{\mathcal P \mathcal P}})$  &    4  &  -2  \\\hline
 $\chi(P_{E_{\mathcal P \mathcal H}})$ &    8  &  0 \\\hline
 $\chi(P'_{E_{\mathcal H \mathcal H}\parallel})$ &    1  &  -1  \\\hline
 $\chi(P_{V_{\mathcal H}})$  &    2  &  0  \\\hline 
$\chi(P_V')^{(\mathcal{T})}$  &    1  &  1  \\
\end{tabular}
\end{center}
\vspace{0.2cm}\caption{Characters for the symmetry-adapted rigidity analysis of the pinned framework in Figure~\ref{fig:exptlngrfw21}(b).}\label{tab:expl1}
\end{table}

In this table the $\chi(P_{V_{\mathcal P}}\otimes I_2)$ character in the second row counts the degrees of freedom of the points.
The $\chi(P'_{V_{\mathcal H}})$ character in the third row counts the degrees of freedom of the lines.   Note that one of these freedoms for each of the horizontal lines is removed in the penultimate row of the table (for the $\chi(P_{V_{\mathcal H}})$ character) by the normalisation constraint (3.8).
The $\chi(P'_{E_{\mathcal P \mathcal P}})$,  $\chi(P_{E_{\mathcal P \mathcal H}})$ and $\chi(P'_{E_{\mathcal H \mathcal H}\parallel})$ characters in the fourth,  fifth and sixth rows correspond to the edge constraints.  
Finally, the last line corresponds to the translational rigid body motion in the direction of the pinned line.

It follows that $$\chi(P_{V}') -\chi(P_V')^{(\mathcal{T})}=  \chi(P_{V_{\mathcal P}}\otimes I_2) + \chi(P'_{V_{\mathcal H}})-\chi(P_V')^{(\mathcal{T})}= (14,0)= 7\chi(\rho_0) + 7\chi(\rho_1).$$
and 
$$\chi(P'_E)= \chi(P'_{E_{\mathcal P \mathcal P}})+\chi( P_{E_{\mathcal P \mathcal H}}) + \chi(P'_{E_{\mathcal H \mathcal H \parallel}})+ \chi(P_{V_{\mathcal H}})=(15,-3)= 6\chi(\rho_0) + 9\chi(\rho_1). $$

Thus, by comparing the coefficients of the characters of the irreducible representations, we see that $(G,p,\ell)$ has a fully-symmetric infinitesimal flex  (see Figure~\ref{fig:exptlngrfw21}(b)) and two anti-symmetric self-stresses.  Using the results in Section~\ref{sec:finiteflex}, the fully-symmetric infinitesimal flex can  be shown to be finite.

The infinitesimal flex shown is fully-symmetric with respect to the group representation and is clearly also fully-symmetric with respect to the extrusion direction for the translations of the points and the rotations of the lines. However the corresponding translations of the horizontal lines are not equal and hence not `fully-symmetric' with respect to the extrusion direction. This discrepancy is due to the off-diagonal term $-\tau_\gamma(i)^T$ which we introduced into the matrix representation $P'_{V_{\mathcal H}} (\gamma)$ of the non-trivial  group element $\gamma$ acting on the coordinates of a hyperplane.

The irreducible representations for the freedoms $\chi(P'_V)$ and the constraints $\chi(P'_E)$ give a block-decomposition of the rigidity matrix $\widetilde R(G,p,\ell)$. Since  $\chi(P_{V}')=8\chi(\rho_0) + 7\chi(\rho_1)$ and $\chi(P'_E)=  6\chi(\rho_0) + 9\chi(\rho_1) $, this block-decomposed matrix has a block $\widetilde R_0(G,p,\ell)$ (of size $6 \times 8$) corresponding to $\rho_0$ and a block $\widetilde R_1(G,p,\ell)$ (of size $9 \times 7$) corresponding to $\rho_1$. The $\rho_1$-symmetric self-stresses (row dependencies in the block $\widetilde R_1(G,p,\ell)$) can be relieved -- without destroying the block-decomposition of the rigidity matrix -- by deleting any two of the three rows corresponding to the parallel constraint between the horizontal lines and the two constraints in $E_{\mathcal{PP}}$ which are fixed by the non-trivial group element, as all of these constraints are $\rho_1$-symmetric.

Note that we may also pin one of the horizontal lines instead and consider the framework with extrusion symmetry in the direction of the $x$-axis. This will again detect the  fully-symmetric infinitesimal flex.
\end{example}

\begin{example}\label{ex:4}
Let's consider a $3$-dimensional point-hyperplane framework defined as follows. We start with a bar-joint framework modelling the $1$-skeleton of a unit cube, that is, the framework has  8 points with coordinates $\{(x,y,z)|\,x,y,z\in\{0,1\}\}$, and the lengths of the twelve edges of the cube are constrained with twelve point-point distance constraints (see Figure~\ref{fig:cube1}b). In addition, the point-hyperplane framework  has 4 hyperplanes which are parallel to the  bottom, top, left and right faces of the cube respectively. All 4 points of the bottom face of the cube have a point-hyperplane distance constraint to the bottom hyperplane. Similarly, the 4 top points, the 4 left points, and the 4 right points of the cube have point-hyperplane distance constraints to the top, left and right hyperplanes, respectively. This forces the points of those 4 faces of the cube to stay coplanar in any motion of the framework. (Note that if we only consider the bar-joint framework of the $1$-skeleton of the cube, without the point-hyperplane distance constraints, then this framework can be moved so that planarity of any face of the cube is destroyed.) Finally, we impose the additional constraints that the two pairs of opposite hyperplanes  have to remain parallel.
%Let's consider a $3$-dimensional point-hyperplane framework that models a symmetrically constrained unit cube. It has  8 points with coordinates $\{(x,y,z)|\,x,y,z\in\{0,1\}\}$, and the lengths of the twelve edges of the cube are constrained with twelve point-point distance constraints. Moreover, the points of the top, bottom, left and right faces of the cube are constrained to be coplanar by forcing each set of four points to have a fixed distance to a plane that is parallel to the cube's face. Finally, we impose the constraint that the two pairs of opposite planes  have to remain parallel. 
Note that this point-hyperplane framework, which we will denote by $(G,p,\ell)$, has $3$-fold extrusion symmetry. This is illustrated in Figure~\ref{fig:cube1}, where for simplicity we assume without loss of generality that the point-hyperplane distances are all zero.

\begin{figure}[htp]
\begin{center}
\begin{tikzpicture}[very thick,scale=0.6]
\tikzstyle{every node}=[circle, draw=black, fill=white, inner sep=0pt, minimum width=5pt];
        \draw[dashed] (0,0)--(3,0);
       \draw[dashed] (0,0)--(0,3);
       \draw[dashed] (0,0)--(1.5,1.5);
         \draw[dashed] (0,3)--(1.5,4.5);
       \draw[dashed] (3,0)--(4.5,1.5);
       \draw[dashed] (1.5,1.5)--(4.5,1.5);
          \draw[dashed] (1.5,1.5)--(1.5,4.5);
       % \draw[dashed] (0,3)--(3,3);
       %\draw[dashed] (3,0)--(3,3);
        \path (0,0) node (p1) {} ;

          \node [draw=white, fill=white] (a) at (-0.4,-0.2) {\tiny $p$};
          
          \node [draw=white, fill=white] (a) at (2.7,0.7) {\tiny $\ell_1$};
           % \node [draw=white, fill=white] (a) at (1.8,1.1) {\tiny $\ell_2$};
            \node [draw=white, fill=white] (a) at (0.8,2.7) {\tiny $\ell_2$};
          
                  \node [draw=white, fill=white] (a) at (1.7,-0.7) {(a)}; 
      \end{tikzpicture}
      \hspace{1cm}
\begin{tikzpicture}[very thick,scale=0.6]
\tikzstyle{every node}=[circle, draw=black, fill=white, inner sep=0pt, minimum width=5pt];
   
            \node [draw=white, fill=white] (a) at (2,0.7) {\tiny $\ell_1\star \star 0$};
           % \node [draw=white, fill=white] (a) at (1.8,0.8) {\tiny $\ell_2 \star 0\star $};
            \node [draw=white, fill=white] (a) at (-1,1.5) {\tiny $\ell_2  0 \star\star $};
            \node [draw=white, fill=white] (a) at (5.3,2.2) {\tiny $\ell_2  1 \star\star $};
   % \node [draw=white, fill=white] (a) at (2.3,2) {\tiny $\ell_2\star  1\star $};
   \node [draw=white, fill=white] (a) at (2.8,4) {\tiny $\ell_1\star \star 1$};
   
        \path (0,0) node (p1) {} ;
        \path (3,0) node (p2) {} ;
        \path (0,3) node (p3) {} ;
         \path (3,3) node (p4) {} ;
        \path (1.5,1.5) node (p5) {} ;
        \path (4.5,1.5) node (p6) {} ;
      \path (1.5,4.5) node (p7) {} ;
        \path (4.5,4.5) node (p8) {} ;
        
        \draw (p1)  --  (p2);
         \draw (p1)  --  (p3);
    \draw (p4)  --  (p2);
         \draw (p4)  --  (p3);
         
           \draw (p5)  --  (p7);
         \draw (p7)  --  (p8);
    \draw (p6)  --  (p5);
         \draw (p6)  --  (p8);
         
           \draw (p1)  --  (p5);
         \draw (p2)  --  (p6);
    \draw (p4)  --  (p8);
         \draw (p7)  --  (p3);
         
          \node [draw=white, fill=white] (a) at (-0.8,-0.2) {\tiny $p000$};
           \node [draw=white, fill=white] (a) at (3.8,-0.2) {\tiny$p100$};
                     \node [draw=white, fill=white] (a) at (0.7,1.5) {\tiny $p010$};
           \node [draw=white, fill=white] (a) at (5.3,1.5) {\tiny$p110$};
           
            \node [draw=white, fill=white] (a) at (-0.8,3) {\tiny $p001$};
           \node [draw=white, fill=white] (a) at (3.8,3) {\tiny$p101$};
                     \node [draw=white, fill=white] (a) at (0.7,4.5) {\tiny $p011$};
           \node [draw=white, fill=white] (a) at (5.3,4.5) {\tiny$p111$};

                  \node [draw=white, fill=white] (a) at (2,-0.7) {(b)}; 
      \end{tikzpicture}
   \end{center}
\caption{(a) A point-hyperplane framework with one point $p$ and two planes, where $\ell_1$ and $\ell_2$ are the bottom and left plane, respectively, and where the point $p$ is constrained to lie on each of the two planes. This framework  can be extruded three times, first to the right along the $x$-axis, then along the $y$-axis, and finally vertically up the $z$-axis, to obtain the symmetrically constrained cube framework in (b). For the point-hyperplane framework in (b), only point-point distances are shown, but we also assume that each set of four points corresponding to the bottom, top, left and right faces of the cube are constrained to remain coplanar and that the top and bottom faces and  the left and right faces remain parallel.}
\label{fig:cube1}
\end{figure}
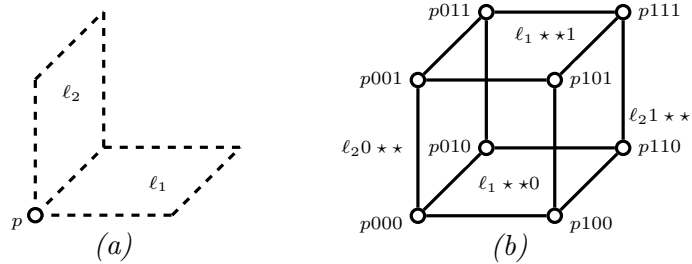

While  $(G,p,\ell)$ has  $3$-fold extrusion symmetry, for each extrusion direction there exists at least one pair of parallel planes that contains this direction. Thus, as discussed in Section~\ref{pinh}, we need to construct a pinned version of this framework to apply the symmetry-adapted counts. We pin the plane $\ell=\ell_1\star\star 0$ by removing the corresponding four columns for $\ell=(a,r)$ and the row for the normalisation of the normal vector $a$ from the rigidity matrix. Moreover, for the parallel plane $\ell'=\ell_1\star \star 1$, we remove the three columns for the normal vector $a'$ and the row for the normalisation of $a'$ in the rigidity matrix, but we keep the column for $r'$ so that the plane can still be displaced in a parallel fashion. The parallel constraints between $\ell$ and $\ell'$ are deleted since they have become redundant. The extrusion symmetry of the framework has now been reduced from $\mathbb{Z}_2^3$ to $\mathbb{Z}_2^2$, where the extrusions are in the directions of the $x$- and $y$-axis. However, if we extrude along the $y$-axis, then the left and right planes are still fixed under this extrusion. Thus, to avoid further pinning, we only analyse the framework as a framework with single extrusion symmetry in the direction of the $x$-axis.

We now apply the symmetry-adapted counts to this pinned framework with $\mathbb{Z}_2$ extrusion symmetry. None of the 8 vertices in   $V_{\mathcal{P}}$ are fixed by the non-trivial group element in $\mathbb{Z}_2$, and hence $\chi(P_{V_{\mathcal P}})=(8,0)$. 
The left and right planes each have four columns in the rigidity matrix. The fixed bottom plane $\ell$ has no columns and the top plane $\ell'$ only has one column in the rigidity matrix. Thus, the entry of $\chi(P'_{V_{\mathcal H}})$ corresponding to the identity is $4+4+1=9$. Since the left and right planes are images of each other under the extrusion, and the other two planes are fixed by the extrusion, the entry of $\chi(P'_{V_{\mathcal H}})$ corresponding to the non-trivial group element is $0+0+1=1$. Note that one of the freedoms for each of the left and right planes is removed in the penultimate row of the table by the normalisation constraint (3.8).

There are 12 edges in $E_{\mathcal{P}\mathcal{P}}$ and 16 edges in $E_{\mathcal{P}\mathcal{H}}$. Moreover, there is one edge in  $E_{\mathcal{H}\mathcal{H}\parallel}$, namely the one for the parallel constraint for the left and right plane.     Four edges in $E_{\mathcal{P}\mathcal{P}}$ are fixed by the non-identity element in $\mathbb{Z}_2$, and each of those edges contributes $-1$ to the trace of $P'_{E_{\mathcal P \mathcal P}}(1)$. None of the edges in  $E_{\mathcal{P}\mathcal{H}}$ are fixed by the non-identity element. The edge in $E_{\mathcal{H}\mathcal{H}\parallel}$ is fixed by each element in $\mathbb{Z}_2$ (with the planes being swapped under the non-trivial group element) and hence, by the definition of $P'_{E_{\mathcal H \mathcal H}\parallel}$, we have $\chi(P'_{E_{\mathcal H \mathcal H}\parallel}\otimes I_2)=(2,-2)$. (Recall that an edge in  $E_{\mathcal{H}\mathcal{H}\parallel}$ corresponds to two rows in the rigidity matrix, as described in Section~\ref{sec:parcon}, and hence we consider $\chi(P'_{E_{\mathcal H \mathcal H}\parallel}\otimes I_2)$ instead of just  $\chi(P'_{E_{\mathcal H \mathcal H}\parallel})$.) 

Finally, note that the pinning has deleted the infinitesimal translation in the $z$-direction, but there still is a  $2$-dimensional space of infinitesimal translations in the $xy$-plane, so that $\chi(P_V')^{(\mathcal{T})}=(2,2).$

 This gives  the  character counts for $(G,p,\ell)$ shown in Table~\ref{tab:expl2}.

\begin{table}[htp]
\begin{center}
\begin{tabular}{l||l|l}
$\mathbb{Z}_{2}$   &   $0$  &  $1$  \\\hline\hline
$\chi(P_{V_{\mathcal P}})$   &   8  &  0 \\\hline
$\chi(P_{V_{\mathcal P}}\otimes I_3)$  &   24 &  0 \\\hline
$\chi(P'_{V_{\mathcal H}})$  &   9  &  1  \\\hline
$\chi(P'_{E_{\mathcal P \mathcal P}})$  &   12  &  -4  \\\hline
$\chi(P_{E_{\mathcal P \mathcal H}})$  &   16  &  0 \\\hline
 $\chi(P'_{E_{\mathcal H \mathcal H}\parallel}\otimes I_2)$ &   2  &  -2\\\hline
$\chi(P_{V_{\mathcal H}})$  &   2  &  0 \\\hline
$\chi(P_V')^{(\mathcal{T})}$ &   2  &  2 \\\hline
\end{tabular}
\end{center}
\vspace{0.3cm}\caption{Characters for the symmetry-adapted rigidity analysis of the pinned version of the framework $(G,p,\ell)$.}\label{tab:expl2}
\end{table}
It follows that $$\chi(P_{V}') -\chi(P_V')^{(\mathcal{T})}=  \chi(P_{V_{\mathcal P}}\otimes I_3) + \chi(P'_{V_{\mathcal H}})-\chi(P_V')^{(\mathcal{T})}= (31,-1)= 15\chi(\rho_0) + 16\chi(\rho_1).$$
and 
$$\chi(P'_E)= \chi(P'_{E_{\mathcal P \mathcal P}})+\chi( P_{E_{\mathcal P \mathcal H}}) + \chi(P'_{E_{\mathcal H \mathcal H}\parallel}\otimes I_2)+ \chi(P_{V_{\mathcal H}})=(32,-6)= 13\chi(\rho_0) + 19\chi(\rho_1). $$
Thus, by comparing the coefficients of the characters of the irreducible representations, we see that $(G,p,\ell)$ has two fully-symmetric infinitesimal flexes  and three anti-symmetric self-stresses.

The two detected infinitesimal flexes are shown in Figure~\ref{fig:cube2}(a) and (b). Note that there exists a third independent fully-symmetric infinitesimal flex, as shown in Figure~\ref{fig:cube2}(c). This one is not detected with our method, as it is a consequence of the special geometry of the configuration of the framework. If we place the points `generically' with respect to the extrusion in the direction of the $x$-axis, i.e., we change the position of the four points on the left face of the cube (and, by extrusion, the four points on the right face of the cube) so that they form a trapezoid rather than a square, then the infinitesimal flex shown in (c) disappears, as the top and bottom face can no longer be kept parallel if we try to shear the framework in the direction of the $y$-axis.

\begin{figure}[htp]
\begin{center}
\begin{tikzpicture}[very thick,scale=0.4]
\tikzstyle{every node}=[circle, draw=black, fill=white, inner sep=0pt, minimum width=3pt];
      
        \path (0,0) node (p1) {} ;
        \path (3,0) node (p2) {} ;
        \path (0,3) node (p3) {} ;
         \path (3,3) node (p4) {} ;
        \path (1.5,1.5) node (p5) {} ;
        \path (4.5,1.5) node (p6) {} ;
      \path (1.5,4.5) node (p7) {} ;
        \path (4.5,4.5) node (p8) {} ;
        
        \draw (p1)  --  (p2);
         \draw (p1)  --  (p3);
    \draw (p4)  --  (p2);
         \draw (p4)  --  (p3);
         
           \draw (p5)  --  (p7);
         \draw (p7)  --  (p8);
    \draw (p6)  --  (p5);
         \draw (p6)  --  (p8);
         
           \draw (p1)  --  (p5);
         \draw (p2)  --  (p6);
    \draw (p4)  --  (p8);
         \draw (p7)  --  (p3);

          \draw[gray,->] (p1)  --  (-0.8,0);
              \draw[gray,->] (p2)  --  (2.2,0);
             \draw[gray,->] (p3)  --  (0.8,3);
              \draw[gray,->] (p4)  --  (3.8,3);
           \draw[gray,->] (p5)  --  (0.7,1.5);
              \draw[gray,->] (p6)  --  (3.7,1.5);
             \draw[gray,->] (p7)  --  (2.3,4.5);
              \draw[gray,->] (p8)  --  (5.3,4.5);
          
                  \node [draw=white, fill=white] (a) at (2,-1.1) {(a)}; 
      \end{tikzpicture}
           \hspace{0.3cm}
      \begin{tikzpicture}[very thick,scale=0.4]
\tikzstyle{every node}=[circle, draw=black, fill=white, inner sep=0pt, minimum width=3pt];
      
        \path (0,0) node (p1) {} ;
        \path (3,0) node (p2) {} ;
        \path (0,3) node (p3) {} ;
         \path (3,3) node (p4) {} ;
        \path (1.5,1.5) node (p5) {} ;
        \path (4.5,1.5) node (p6) {} ;
      \path (1.5,4.5) node (p7) {} ;
        \path (4.5,4.5) node (p8) {} ;
        
        \draw (p1)  --  (p2);
         \draw (p1)  --  (p3);
    \draw (p4)  --  (p2);
         \draw (p4)  --  (p3);
         
           \draw (p5)  --  (p7);
         \draw (p7)  --  (p8);
    \draw (p6)  --  (p5);
         \draw (p6)  --  (p8);
         
           \draw (p1)  --  (p5);
         \draw (p2)  --  (p6);
    \draw (p4)  --  (p8);
         \draw (p7)  --  (p3);

          \draw[gray,->] (p1)  --  (-0.8,0);
              \draw[gray,->] (p2)  --  (2.2,0);
             \draw[gray,->] (p3)  --  (-0.8,3);
              \draw[gray,->] (p4)  --  (2.2,3);
           \draw[gray,->] (p5)  --  (2.3,1.5);
              \draw[gray,->] (p6)  --  (5.3,1.5);
             \draw[gray,->] (p7)  --  (2.3,4.5);
              \draw[gray,->] (p8)  --  (5.3,4.5);
          
                  \node [draw=white, fill=white] (a) at (2,-1.1) {(b)}; 
                  \end{tikzpicture}
       \hspace{0.3cm}
              \begin{tikzpicture}[very thick,scale=0.4]
\tikzstyle{every node}=[circle, draw=black, fill=white, inner sep=0pt, minimum width=3pt];
      
        \path (0,0) node (p1) {} ;
        \path (3,0) node (p2) {} ;
        \path (0,3) node (p3) {} ;
         \path (3,3) node (p4) {} ;
        \path (1.5,1.5) node (p5) {} ;
        \path (4.5,1.5) node (p6) {} ;
      \path (1.5,4.5) node (p7) {} ;
        \path (4.5,4.5) node (p8) {} ;
        
        \draw (p1)  --  (p2);
         \draw (p1)  --  (p3);
    \draw (p4)  --  (p2);
         \draw (p4)  --  (p3);
         
           \draw (p5)  --  (p7);
         \draw (p7)  --  (p8);
    \draw (p6)  --  (p5);
         \draw (p6)  --  (p8);
         
           \draw (p1)  --  (p5);
         \draw (p2)  --  (p6);
    \draw (p4)  --  (p8);
         \draw (p7)  --  (p3);

          \draw[gray,->] (p1)  --  (0.6,0.6);
              \draw[gray,->] (p2)  --  (3.6,0.6);
             \draw[gray,->] (p3)  --  (-0.6,2.4);
              \draw[gray,->] (p4)  --  (2.4,2.4);
           \draw[gray,->] (p5)  --  (2.1,2.1);
              \draw[gray,->] (p6)  --  (5.1,2.1);
             \draw[gray,->] (p7)  --  (0.9,3.9);
              \draw[gray,->] (p8)  --  (3.9,3.9);
          
                  \node [draw=white, fill=white] (a) at (2,-1.1) {(c)}; 
      \end{tikzpicture}
       \end{center}
\caption{Fully-symmetric infinitesimal flexes of the constrained cube. (a) shears the front and back face; (b) shears the top and bottom face; (c) shears the left and right face. The flexes shown in (a) and (b) are detected via the symmetry-adapted analysis. The flex shown in (c) is not detected, since it disappears if the points are placed `generically' with respect to the extrusion symmetry in the direction of the $x$-axis.}
\label{fig:cube2}
\end{figure}
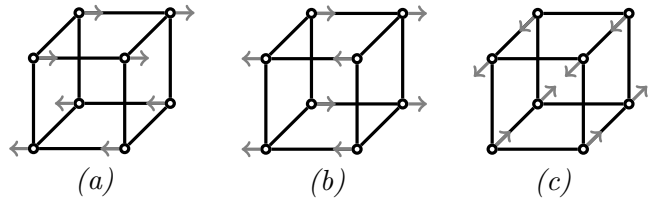

If instead of the front and back face of the cube we omit the coplanarity constraint for the points of the bottom and top or the left and right face (and keep the coplanarity constraints of the points of the other 4 faces), then the same basic analysis as above can be applied to those frameworks. In each case we obtain the same character counts and hence detect two fully-symmetric infinitesimal flexes. This double counts each infinitesimal flex and hence we find three independent infinitesimal flexes. They are exhibited by the physical model built of Polydron shown in Figure~\ref{fig:cubepoly}.  By taking 6 rings of 4  panels each and attaching them to each other along the edges of a cube all point-point distances are preserved and all four points of each face are forced to remain coplanar (and opposite faces parallel). This simple structure has been used as a fundamental building block of certain meta-materials \cite{nature}, and the flexes detected with our method (see Figure~\ref{fig:cube2}) allow the highly versatile reconfigurability of these meta-materials.

\begin{figure}[htb]
\begin{center} 
\includegraphics[scale=0.025,angle=-90]{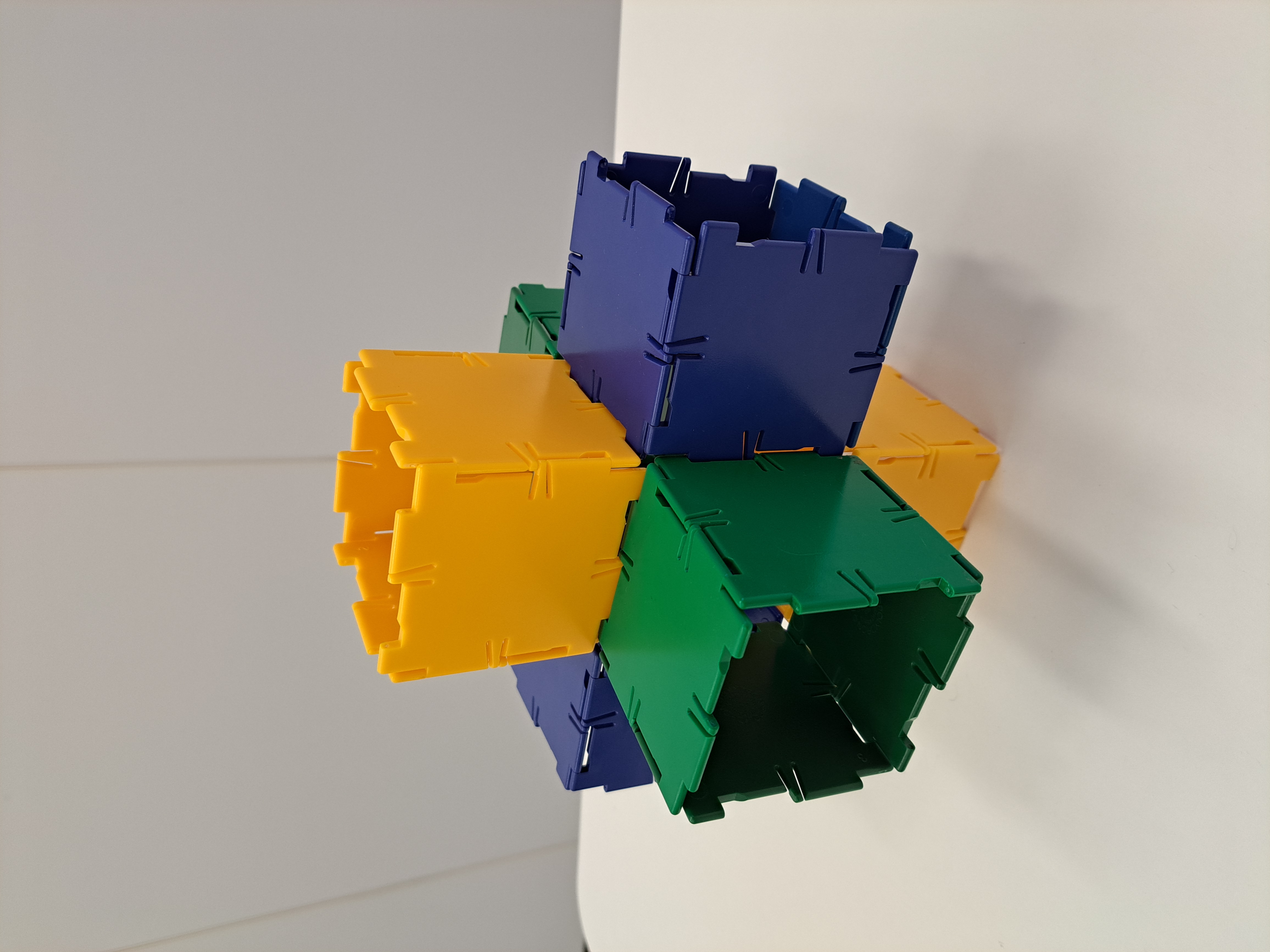}\hspace{0.4cm}  %FIGURE  
\includegraphics[scale=0.025,angle=-90]{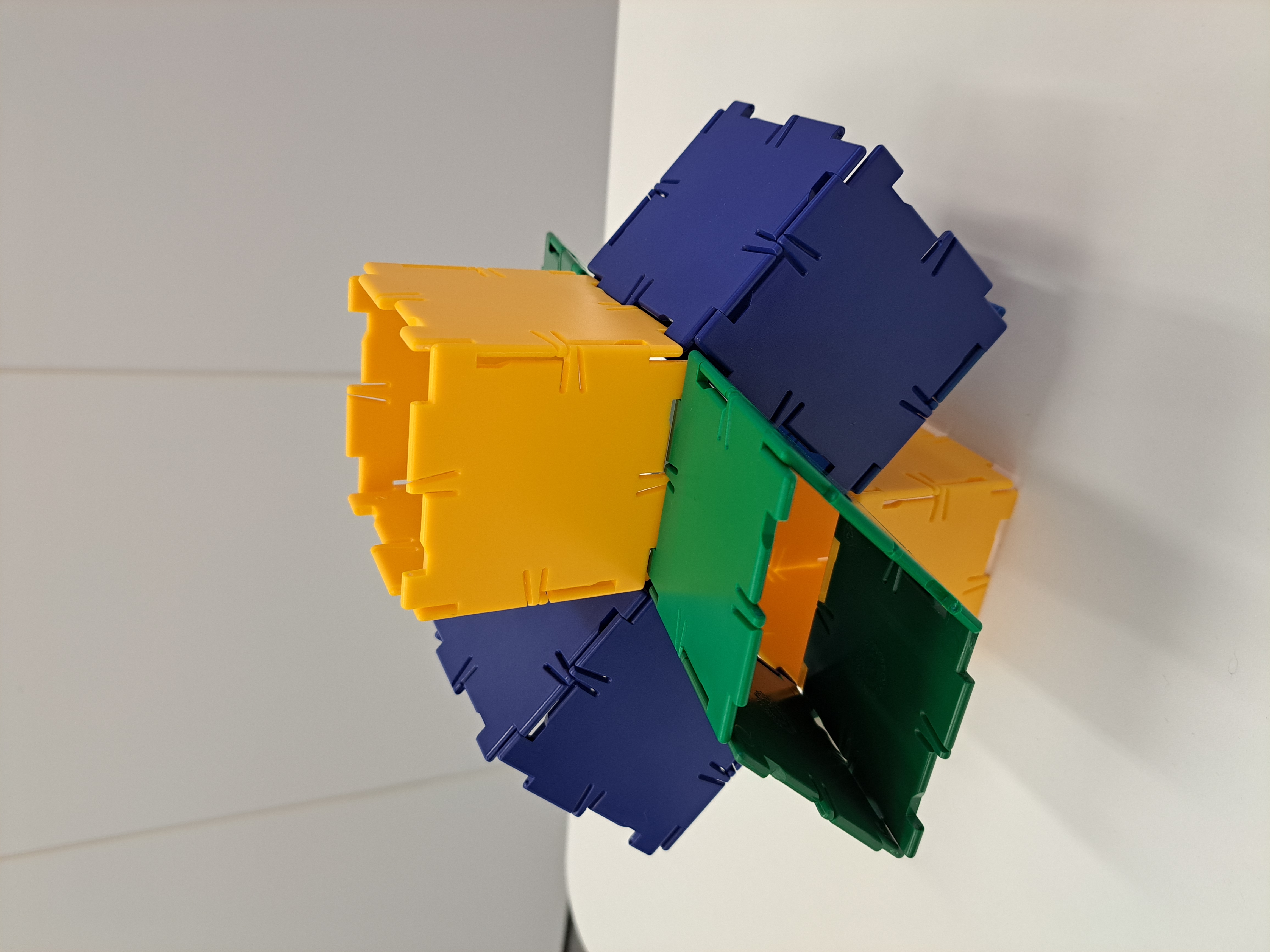}\hspace{0.4cm}   %FIGURE      
\includegraphics[scale=0.025,angle=-90]{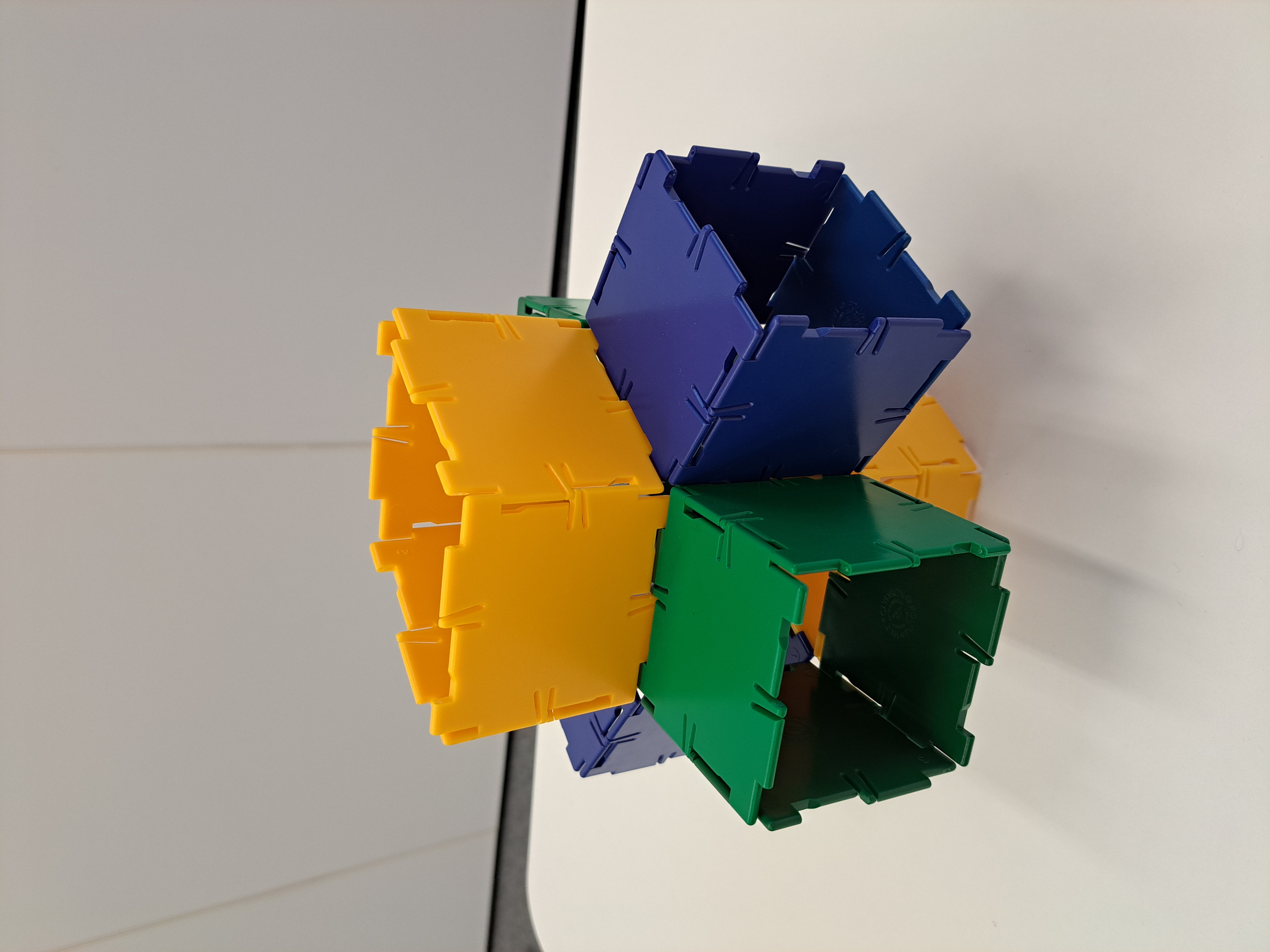}\hspace{0.4cm}  %FIGURE      
\includegraphics[scale=0.025,angle=-90]{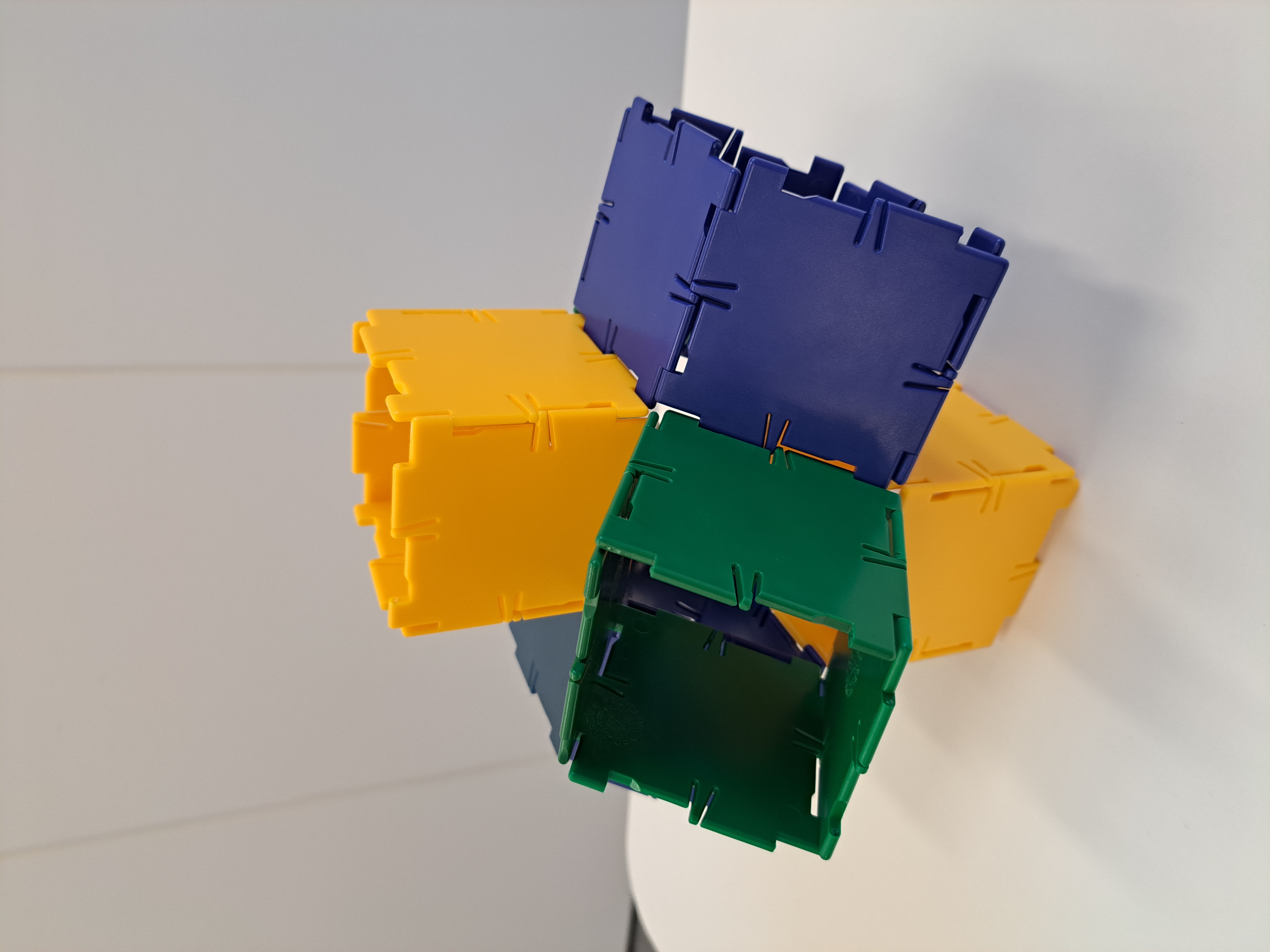}  %FIGURE          
\caption{Physical model of a constrained cube which exhibits the three shearing flexes shown in Figure~\ref{fig:cube2}.}
\label{fig:cubepoly}
\end{center}
\end{figure}

\end{example}

A general observation about the infinitesimal flexibility of point-hyperplane frameworks with $t$-fold extrusion symmetry is the following. 

If the point-hyperplane framework has no hyperplane that is fixed by an extrusion (i.e. $F_h=\emptyset$ for all $h=1,\ldots, t$), then it clearly has an infinitesimal flex which shears the framework orthogonally to an extrusion direction (as was the case for bar-joint frameworks -- recall the discussion in Remark~\ref{rem:rotmot}). 

If the point-hyperplane framework has a hyperplane that is fixed by an extrusion, but there is no point whose distance is fixed to this hyperplane, then there also exists  an infinitesimal flex, because that hyperplane can be displaced on its own in a parallel fashion preserving all of its angle and parallel constraints.

 If however, there exists a hyperplane that is fixed by an extrusion symmetry and there also exists a point-hyperplane distance constraint for this hyperplane, then the  point-hyperplane framework may be infinitesimally rigid.  For example, if we take $H$ to be a $3$-cycle with one point vertex and two line vertices and extrude a framework on $H$ in $\mathbb{R}^2$ along one of the lines, then we obtain an infinitesimally rigid point-line framework. See Figure~\ref{fig:simpleex} for an illustration.

\begin{figure}[htp]
\begin{center}
\begin{tikzpicture}[very thick,scale=0.7]
\tikzstyle{every node}=[circle, draw=black, fill=white, inner sep=0pt, minimum width=5pt];
     \node [draw=white, fill=white] (a) at (-0.5,1) {\tiny $w_1$};
              \node [draw=white, fill=white] (a) at (-0.5,-0.3) {\tiny $w_2$};
             \node [draw=white, fill=white] (a) at (1.5,0.5) {\tiny $v_1$};
                           
        \path (0,-0.3) node[fill=black] (p1) {} ;
        \path (0,1) node[fill=black] (p2) {} ;
        \path (1,0.5) node (p3) {} ;
        
         \draw (p1)  --  (p3);
        \draw (p3)  --  (p2);
 \draw (p1)  --  (p2);
            
        \node [draw=white, fill=white] (a) at (0.5,-1.6) {(a)};
      \end{tikzpicture}    
            \hspace{0.5cm}
  \begin{tikzpicture}[very thick,scale=0.7]
\tikzstyle{every node}=[circle, draw=black, fill=white, inner sep=0pt, minimum width=4pt];
     \node [draw=white, fill=white] (a) at (0.3,0.3) {\tiny $p_1$};

        \path (0,0) node (p1) {} ;
       
           \draw(-2,-0.5)--(1.5,-0.5);    
             \draw(-2,-1)--(0,1);   
             
               \draw[dotted](p1)--(0,-0.5);  
                 \draw[dotted](p1)--(-0.5,0.5);  
                 
             \node [draw=white, fill=white] (a) at (0.5,-0.8) {\tiny $\ell_2$};   
                \node [draw=white, fill=white] (a) at (-1.2,0.3) {\tiny $\ell_1$}; 

\draw[dotted] (-0.7,-0.5) arc
    [
        start angle=0,
        end angle=75,
        x radius=0.5cm,
        y radius =0.5cm
    ] ;
                
        \node [draw=white, fill=white] (a) at (0,-1.6) {(b)};
      \end{tikzpicture} 
      \hspace{0.5cm}   
       \begin{tikzpicture}[very thick,scale=0.7]
\tikzstyle{every node}=[circle, draw=black, fill=white, inner sep=0pt, minimum width=4pt];
   \node [draw=white, fill=white] (a) at (-0.7,-0.8) {\tiny $\ell_20$};   
                \node [draw=white, fill=white] (a) at (-1.3,0.2) {\tiny $\ell_1\star$};
   \node [draw=white, fill=white] (a) at (2.8,1.8) {\tiny $\ell_21$};   

     \node [draw=white, fill=white] (a) at (0.1,0.4) {\tiny $p_10$};
              \path (0,0) node (p1) {} ;
       
           \draw(-2,-0.5)--(1.5,-0.5);    
             \draw(-2,-1)--(2,3);   
             
               \draw[dotted](p1)--(0,-0.5);  
                 \draw[dotted](p1)--(-0.5,0.5);  
                           
                   \node [draw=white, fill=white] (a) at (2.3,2.3) {\tiny $p_11$};          
        \path (2,2) node (p2) {} ;
       
           \draw(-0,1.5)--(3.5,1.5);    
              
              \draw[dotted](p1)--(p2); 
               \draw[dotted](p2)--(2,1.5);  
                 \draw[dotted](p2)--(1.5,2.5);  
                 
               \draw[dotted](0.2,-0.5)--(2.2,1.5); 

\draw[dotted] (-0.7,-0.5) arc
    [
        start angle=0,
        end angle=75,
        x radius=0.5cm,
        y radius =0.5cm
    ] ;

\draw[dotted] (1.37,1.5) arc
    [
        start angle=0,
        end angle=75,
        x radius=0.5cm,
        y radius =0.5cm
    ] ;
               
        \node [draw=white, fill=white] (a) at (0,-1.6) {(c)};
      \end{tikzpicture} 
      \end{center}
     \vspace{-0.5cm}
\caption{A $\mathcal{PH}$-graph with  one point vertex and two line vertices (a), and its realisation as a point-line framework in the plane (b). If the framework in (b) is extruded along the line $\ell_1$, then we obtain the framework in (c), which is infinitesimally rigid.}
\label{fig:simpleex}
\end{figure}

We conclude this section with a remark about the effect of an affine transformation on the infinitesimal rigidity of a point-hyperplane framework.  In contrast to a bar-joint framework,  even an affine transformation can alter its infinitesimal rigidity properties. Consider, for example, a $3$-cycle which is realised as a point-line framework in the plane with two points and a line, where the edge between the two points is perpendicular to the line direction. This framework has an infinitesimal  flex, which disappears under any affine transformation which does not conserve the perpendicularity. This is illustrated in Figure~\ref{fig:simpleex2} and can easily be verfied by computing the ranks of the corresponding rigidity matrices.

\begin{figure}[htp]
\begin{center}
\begin{tikzpicture}[very thick,scale=0.7]
\tikzstyle{every node}=[circle, draw=black, fill=white, inner sep=0pt, minimum width=5pt];
     \node [draw=white, fill=white] (a) at (-0.4,1) {\tiny $v_1$};
              \node [draw=white, fill=white] (a) at (-0.4,-0.3) {\tiny $v_2$};
             \node [draw=white, fill=white] (a) at (1.5,0.5) {\tiny $w_1$};
                           
        \path (0,-0.3) node (p1) {} ;
        \path (0,1) node (p2) {} ;
        \path (1,0.5) node[fill=black] (p3) {} ;
        
         \draw (p1)  --  (p3);
        \draw (p3)  --  (p2);
 \draw (p1)  --  (p2);
            
        \node [draw=white, fill=white] (a) at (0.5,-1.6) {(a)};
      \end{tikzpicture}    
            \hspace{0.5cm}
  \begin{tikzpicture}[very thick,scale=0.7]
\tikzstyle{every node}=[circle, draw=black, fill=white, inner sep=0pt, minimum width=4pt];
     \node [draw=white, fill=white] (a) at (-0.5,1) {\tiny $p_1$};             
            \node [draw=white, fill=white] (a) at (-0.5,-0.5) {\tiny $p_2$};                
        \path (0,-0.5) node (p1) {} ;
       \path (0,1) node (p2) {} ;
       \draw[dotted](p1)--(p2);
 \draw(-2,0.25)--(2,0.25);

\draw[dotted](0.1,-0.5)--(0.1,0.25);
\draw[dotted](-0.1,1)--(-0.1,0.25);
         
                \node [draw=white, fill=white] (a) at (-1.5,0.6) {\tiny $\ell_1$};

        \node [draw=white, fill=white] (a) at (0,-1.6) {(b)};
      \end{tikzpicture} 
                  \hspace{0.5cm}
  \begin{tikzpicture}[very thick,scale=0.7]
\tikzstyle{every node}=[circle, draw=black, fill=white, inner sep=0pt, minimum width=4pt];
     \node [draw=white, fill=white] (a) at (0.4,1) {\tiny $p_1$};             
            \node [draw=white, fill=white] (a) at (-0.4,-0.5) {\tiny $p_2$};                
        \path (0,-0.5) node (p2) {} ;
       \path (0,1) node (p1) {} ;
       \draw[dotted](p1)--(p2);
 \draw(-1.2,1.45)--(1.2,-0.95);

\draw[dotted](p1)--(-0.4,0.6);
\draw[dotted](p2)--(0.4,-0.1);
         
                \node [draw=white, fill=white] (a) at (-1.5,0.6) {\tiny $\ell_1$};

        \node [draw=white, fill=white] (a) at (0,-1.6) {(c)};
      \end{tikzpicture} 
      \end{center}
     \vspace{-0.5cm}
\caption{A $\mathcal{PH}$-graph with two point vertices and one line vertex (a), and two realisations as a point-line framework in the plane, where in (b) the direction of the line is perpendicular to the edge joining the points. The framework in (b) is infinitesimally flexible, whereas the one in (c) is not.}
\label{fig:simpleex2}
\end{figure}

On the other hand, the following lemma shows that a bar-joint or point-hyperplane framework with $t$-fold extrusion symmetry retains the $t$-fold extrusion symmetry under a general affine transformation. We will use this lemma in Section~\ref{sec:linpush}, where we will consider finite flexibility. 

Recall that an affine transformation
$A_{A,v}$ in $\mathbb{R}^d$ is determined by an invertible matrix $A \in \mathbb{R}^d\times \mathbb{R}^d$ and a vector $b \in \mathbb{R}^d$. 

\begin{lem}\label{affine-invariance-hyperplane}
Let $G=H\square_{F_1,\ldots, F_t} K_2^{\square t}$  and let $(G,p,\ell)$ be a point-hyperplane framework with $t$-fold extrusion symmetry with extrusion directions $\tau_1,\dots,\tau_t$. Let $A_{A,b}$ be an affine transformation of  $\mathbb{R}^d$.  Then $(G,A_{A,b}p,A_{A,b}\ell)$ has $t$-fold extrusion symmetry with extrusion directions $A\tau_1,\dots,A\tau_t$.
\end{lem}

\begin{proof}
The affine transformation of a point $p \in \mathbb{R}^d$ is given by $A_{A,b}p =Ap+b$.  
The affine transformation of the hyperplane $\ell=(a,r)$ is determined by the affine transformation of the points on the hyperplane and is given by 
$A_{A,b}(a,r)=(aA^{-1},r+\langle a  A^{-1},b\rangle)$.  

We will show that if the conditions of Definition~\ref{def:extrsymph} are  satisfied by $(G,p,\ell)$ and vectors $\tau_1,\dots,\tau_t$ then they are also satisified by $(G,A_{A,b}p,A_{A,b}\ell)$ and vectors $A\tau_1,\dots,A\tau_t$.

The requirement (i) that
$$A_{A,b}p_{\theta(\gamma)((v_j,\be))}=
A_{A,b}p_{(v_j,\be)}+ \sum_{i\in X}A \tau_i - \sum_{g\in Y} A\tau_g,$$ follows immediately from the fact that $A$ is a linear operator in $\mathbb{R}^d$.

The requirement (ii) follows from the fact that parallel hyperplanes remain parallel under an affine transformation.

The requirement (iii) follows from the fact that if a displacement vector lies in a hyperplane then it also lies in the hyperplane after an affine transformation.

The property (iv) for $(G,p,\ell)$ says that for every $(w_j,\bm e)\in V_{\mathcal{H}}$  and every $\gamma\in \mathbb{Z}_2^t$ we have
 $$r_{\theta(\gamma)((w_j,\bm e))}-r_{(w_j,\bm e)}= \sum_{i\in X} \langle a_{(w_j,e)},\tau_i\rangle - \sum_{g\in Y} \langle a_{(w_j,e)},\tau_g\rangle,$$
 where $X$ is the set of indices in $\{1,\ldots, t\}$ for which $\be$ has an entry of $0$ and $\gamma$ has an entry of $1$, and  $Y$ is the set of indices in $\{1,\ldots, t\}$ for which both $\be$ and $\gamma$ have an entry of $1$. 

If $(a,r_1)$ and $(a,r_2)$ are two parallel hyperplanes then $r_1-r_2=A_{A,b}r_1-A_{A,b}r_2$.
%because $A_{A,v}$ is a linear operator on $\mathbb{P}^d(\mathbb{R})$.  
Hence $A_{A,b}r_{\theta(\gamma)((w_j,\bm e))}-A_{A,b}r_{(w_j,\bm e)}=\sum_{i\in X} \langle a_{(w_j,e)},\tau_i\rangle - \sum_{g\in Y} \langle a_{(w_j,e)},\tau_g\rangle=\sum_{i\in X} \langle a_{(w_j,e)}A^{-1},A\tau_i\rangle - \sum_{g\in Y} \langle a_{(w_j,e)}A^{-1},A\tau_g\rangle$. This is the required property (iv) for extrusion symmetry in the framework $(G,A_{A,b}p,A_{A,b}\ell)$ with extrusion directions $A\tau_1,\dots,A\tau_t$.
\end{proof}

\section{Detecting finite flexibility and local redundancy} \label{sec:finiteflex}

\subsection{A general method for detecting finite motions}

In this section we present a general method for detecting finite motions (and local redundancy -- see Definition~\ref{localred}) in point-hyperplane frameworks. Our approach is analogous to the one presented  in \cite{BSfinite} for symmetric bar-joint frameworks. (See also~\cite{guestfow,kangguest}.) 
Let $G=(V_{\mathcal P}\cup V_{\mathcal H}, E)$ be a \emph{decorated} $\mathcal P \mathcal H$-graph with $|V_{\mathcal P}|=n$ and $|V_{\mathcal H}|=k$ and let $E_{\mathcal{HH}\parallel}$ and $E_{\mathcal{HH}\not\parallel}$ be the partition sets of $E_\mathcal{HH}$ induced by the decoration (recall the definition in Section~\ref{sec:parcon}). We fix an ordering of the edges of $E \setminus E_{\mathcal{HH}\parallel}$ so that the edges in $E_{\mathcal P \mathcal P}$ come first, followed by the edges in $E_{\mathcal P \mathcal H}$ and finally the edges in $E_{\mathcal H \mathcal H\not\parallel}$.

Let 
$\mathbb{W}_G=\{(p,\ell)\in \mathbb{R}^{dn}\oplus \mathbb{R}^{(d+1)k}:a_i\parallel a_j \forall ij \in E_{\mathcal{HH}\parallel}\}$.  It is straightforward to verify that $\mathbb{W}_G$ is a linear subspace of $\mathbb{R}^{dn}\oplus\mathbb{R}^{(d+1)k}$.

We then define the measurement map %$f_{G}:\mathbb{R}^{dn+(d+1)k}\to \mathbb{R}^{|E|+(d-2)|E_{\mathcal{HH}\parallel}|+k}$  
$f_{G}:\mathbb{W}_G\to \mathbb{R}^{|E|+k}$  by 
\begin{displaymath}
f_{G}\big(p_{1},\ldots,p_{n},\ell_1,\ldots, \ell_k\big)=\big(\ldots, \|p_{i}-p_{j}\|^2,\ldots, \langle p_{i'}, a_{j'}\rangle-r_{j'}, \ldots, \langle a_{i''}, a_{j''}\rangle, 
%f_{a_i'',a_j''},
\ldots,\langle a_j, a_j\rangle\ldots\big)\textrm{, }
\end{displaymath}
where $p_{i}\in \mathbb{R}^{d}$ for all $i=1,\ldots,n$,  $\ell_j=(a_j,r_j)\in \mathbb{R}^{d+1}$ for all $j=1,\ldots,k$, and  where $\{i,j\}\in E_{\mathcal P \mathcal P}$,   $\{i',j'\}\in E_{\mathcal P \mathcal H}$ and $\{i'',j''\}\in E_{\mathcal {HH}\not\parallel}$.  

Let $G$ be a decorated point-hyperplane graph and let $K_{n+k}$ be the decorated complete point-hyperplane graph on the vertices of $V(G)$, where the decoration of $K_{n+k}$ is inherited from the decoration of $G$. 
%We say that $(X,p,\ell)$ is the completion of $(G,p,\ell)$ 
The partition of the edges in $E_\mathcal{HH}(K_{n+k})$ into $E_{\mathcal{HH}\parallel}(K_{n+k})$ and $E_{\mathcal{HH}\not\parallel}(K_{n+k})$ is well-defined and we have $E_{\mathcal{HH}\not\parallel}(G) \subseteq E_{\mathcal{HH}\not\parallel}(K_{n,k})$.

$f_{G}^{-1}\big(f_{G}(p,\ell)\big)$ is the set of all configurations $(q,m)$  of $n$ points and $k$ hyperplanes in 
$\mathbb{W}_G$ with the property that corresponding pairs of points $\{p_i,p_j\}$ and $\{q_i,q_j\}$  where $ij \in E_\mathcal{PP}$ have the same length, corresponding pairs of points and hyperplanes $\{p_i,\ell_j\}$ and $\{q_i, m_j \}$ where $ij \in E_{\mathcal P \mathcal H}$ have the same distance and corresponding pairs of hyperplanes $\{\ell_i ,\ell_j\}$ and $\{m_i ,m_j\}$ have the same angle if $ij \in E_{\mathcal{HH}\not\parallel}$ or are parallel if $ij \in E_{\mathcal{HH}\parallel}$. We have $f^{-1}_{K_{n+k}}\big(f_{K_{n+k}}(p,\ell)\big)\subseteq f_{G}^{-1}\big(f_{G}(p,\ell)\big)$ because $E(G) \subseteq E(K_{n+k})$.

\begin{defn}
\label{localred}
An edge $e$ is said to be \emph{locally redundant in $\mathbb{W}_G$} at $(p,\ell)$ if there is a neighborhood $N_{(p,\ell)}$  in $\mathbb{W}_G$
such that $f^{-1}_G (f_G(p,\ell))\cap N_{p,\ell}=f^{-1}_{G \setminus e}((f_{G\setminus e}(p,\ell)) \cap N_{p,\ell}$.
%emph{$\mathscr{A}$-preserving flex} of a point-hyperplane framework $(G,p,\ell)$ with $(p,\ell)\in \mathscr{A}$ is a differentiable path $x:[0,1]\to \mathscr{A}$ such that $x(0)=(p,\ell)$ and $x(t)\in \tilde{f}_{G}^{-1}\big(\tilde{f}_{G}((p,\ell))\big)\setminus \tilde{f}_{K_{n+k}}^{-1}\big(\tilde{f}_{K_{n+k}}((p,\ell))\big)$ for all $t\in (0,1]$.}
\end{defn}

Let $\mathscr{A}$ be a smooth submanifold of $\mathbb{W}_G$
%$\mathbb{R}^{dn+(d+1)k}$. associated with a framework $(G,p,\ell)$ and its completion $(X,p,\ell)$ 
and let $\tilde{f}_{G}:\mathscr{A}\to \mathbb{R}^{|E|+k}$ denote the restriction of the edge function $f_{G}$ to $\mathscr{A}$, and $\tilde{f}_{K}:\mathscr{A}\to \mathbb{R}^{\binom{n+k}{2}}$ denote the restriction of $f_K$ to $\mathscr{A}$. The Jacobian matrices of $\tilde{f}_{G}$ and $\tilde{f}_K$, evaluated at a point $(p,\ell)\in \mathscr{A}$, are denoted by $d\tilde{f}_{G}(p,\ell)$ and $d\tilde{f}_K(p,\ell)$, respectively.

\begin{defn}\label{symregulardef}
\emph{An element $(p,\ell)\in \mathscr{A}$ is said to be an \emph{$\mathscr{A}$-regular point of $G$} if there exists a neighborhood $N_{(p,\ell)}$ of $(p,\ell)$ in $\mathscr{A}$ so that $\textrm{rank } \big(d\tilde{f}_{G}((p,\ell))\big)=\textrm{max }\{\textrm{rank } \big(d\tilde{f}_{G}((q,m))\big)|\, (q,m)\in N_{(p,\ell)}\}$. An \emph{$\mathscr{A}$-regular point of $K_{n+k}$} is defined analogously.}
\end{defn}

\begin{defn}
\label{symmpreservflex}
\emph{An \emph{$\mathscr{A}$-preserving flex} of a point-hyperplane framework $(G,p,\ell)$ with $(p,\ell)\in \mathscr{A}$ is a differentiable path $x:[0,1]\to \mathscr{A}$ such that $x(0)=(p,\ell)$ and $x(T)\in \tilde{f}_{G}^{-1}\big(\tilde{f}_{G}((p,\ell))\big)\setminus \tilde{f}_{K_{n+k}}^{-1}\big(\tilde{f}_{K_{n+k}}((p,\ell))\big)$ for all $T\in (0,1]$.}
\end{defn}

Note that every point-hyperplane framework $(G,p,\ell)$ implies a well-defined affine subspace $\mathbb{W}_G$, a well-defined partition of $E_\mathcal{HH}$ and a well-defined edge map $f_G$ using the decoration of $G$ implied by $(G,p,\ell)$. Together with a subspace $\mathscr{A} \subset \mathbb{W}_G$,  $(G,p,\ell)$ also implies well-defined maps $\tilde{f}_G$ and $\tilde{f}_{K_{n+k}}$. 

\begin{lem}
\label{hs1}
Let $(G,p,\ell)$ be a point-hyperplane framework, let $\mathbb{W}_G$ be the subspace defined above and let $\mathscr{A}$ be an affine subspace of $\mathbb{W}_G$.

If $(p,\ell) \in \mathscr{A}$ is an $\mathscr{A}$-regular point of $G$, then there exists a neighborhood $N_{(p,\ell)}$ of $(p,\ell)$ in $\mathscr{A}$ such that $\tilde{f}_{G}^{-1}\big(\tilde{f}_{G}((p,\ell))\big)\cap N_{(p,\ell)}$ is a smooth manifold of dimension $\textrm{dim }\mathscr{A}-\textrm{rank }\big(d\tilde{f}_{G}((p,\ell))\big)$.
\end{lem}

\begin{proof} The result follows immediately from Proposition 2 in \cite{AR}. 
\end{proof}

\begin{thm}
\label{flexthm1}
Let $(G,p,\ell)$ be a point-hyperplane framework  
with $(p,\ell)\in \mathscr{A}$. If $(p,\ell)$ is an $\mathscr{A}$-regular point of $G$ 
and the points and hyperplanes of  $(G,p,\ell)$ affinely span  $\mathbb{R}^{d}$, then
\begin{itemize}
\item[(i)] $\textrm{rank }\big(d\tilde{f}_{G}((p,\ell))\big)= \textrm{rank }\big(d\tilde{f}_{K_{n+k}}((p,\ell))\big)$ if and only if $(G,p,\ell)$ has no $\mathscr{A}$-preserving flex;
\item[(ii)] $\textrm{rank }\big(d\tilde{f}_{G}((p,\ell))\big)< \textrm{rank }\big(d\tilde{f}_{K_{n+k}}((p,\ell))\big)$ if and only if $(G,p,\ell)$ has an $\mathscr{A}$-preserving flex.
\end{itemize}
\end{thm}

\begin{proof} Since  the points and hyperplanes of  $(G,p,\ell)$ affinely span  $\mathbb{R}^{d}$,   $(p,\ell)$ is an $\mathscr{A}$-regular point of $K_{n+k}$. So since  $(p,\ell)$ is an $\mathscr{A}$-regular point of both $G$ and $K_{n+k}$, it follows from Lemma \ref{hs1} that there exist neighborhoods $N_{(p,\ell)}$ and $N'_{(p,\ell)}$ of $(p,\ell)$ in $\mathscr{A}$ so that $\tilde{f}_{G}^{-1}\big(\tilde{f}_{G}((p,\ell))\big)\cap N_{(p,\ell)}$ is a manifold of dimension $\textrm{dim }\mathscr{A}-\textrm{rank }\big(d\tilde{f}_{G}((p,\ell))\big)$ and $\tilde{f}_{K_{n+k}}^{-1}\big(\tilde{f}_{K_{n+k}}((p,\ell))\big)\cap N'_{(p,\ell)}$ is a manifold of dimension $\textrm{dim }\mathscr{A}-\textrm{rank }\big(d\tilde{f}_{K_{n+k}}((p,\ell))\big)$. Since  $\tilde{f}_{K_{n+k}}^{-1}\big(\tilde{f}_{K_{n+k}}((p,\ell))\big)\cap N''_{(p,\ell)}$ is a submanifold of $\tilde{f}_{G}^{-1}\big(\tilde{f}_{G}((p,\ell))\big)\cap N''_{(p,\ell)}$, where $N''_{(p,\ell)}=N_{(p,\ell)}\cap N'_{(p,\ell)}$, it follows that \begin{displaymath}\textrm{rank }\big(d\tilde{f}_{G}((p,\ell))\big)\leq \textrm{rank }\big(d\tilde{f}_{K_{n+k}}((p,\ell))\big)\textrm{.}\end{displaymath} Clearly, $\textrm{rank }\big(d\tilde{f}_{G}((p,\ell))\big)= \textrm{rank }\big(d\tilde{f}_{K_{n+k}}((p,\ell))\big)$ if and only if there exists a neighborhood $N^*_{(p,\ell)}$ of $(p,\ell)$ in $\mathscr{A}$ such that $\tilde{f}_{K_{n+k}}^{-1}\big(\tilde{f}_{K_{n+k}}((p,\ell))\big)\cap N^*_{(p,\ell)}=\tilde{f}_{G}^{-1}\big(\tilde{f}_{G}((p,\ell))\big)\cap N^*_{(p,\ell)}$. Therefore, if $\textrm{rank }\big(d\tilde{f}_{G}((p,\ell))\big)= \textrm{rank }\big(d\tilde{f}_{K_{n+k}}((p,\ell))\big)$, then there does not exist an $\mathscr{A}$-preserving flex of $(G,p,\ell)$.\\\indent If $\textrm{rank }\big(d\tilde{f}_{G}((p,\ell))\big)< \textrm{rank }\big(d\tilde{f}_{K_{n+k}}((p,\ell))\big)$, then every neighborhood of $(p,\ell)$ in $\mathscr{A}$ contains elements of $\tilde{f}_{G}^{-1}\big(\tilde{f}_{G}((p,\ell))\big)\setminus\tilde{f}_{K_{n+k}}^{-1}\big(\tilde{f}_{K_{n+k}}((p,\ell))\big)$, and hence, by the same argument as in the proof of Proposition 1 in \cite{AR}, there exists an $\mathscr{A}$-preserving flex of $(G,p,\ell)$. This completes the proof. 
\end{proof}

\begin{defn}
\label{symmpreservinfflex}
\emph{An \emph{$\mathscr{A}$-preserving infinitesimal motion} of a point-hyperplane framework $(G,p,\ell)$ with $(p,\ell)\in \mathscr{A}$ is an element of the kernel of $d\tilde{f}_{G}((p,\ell))$. %(i.e., an infinitesimal motion of $(G,p,\ell)$ which is tangent to $\mathscr{A}$). 
A non-trivial $\mathscr{A}$-preserving infinitesimal motion is called an $\mathscr{A}$-preserving infinitesimal flex.}
\end{defn}

\begin{thm}
\label{flexthm2}
Let $(G,p,\ell)$ be a point-hyperplane framework with $(p,\ell)\in \mathscr{A}$, where the points and hyperplanes affinely span  $\mathbb{R}^{d}$. If $(p,\ell)$ is an $\mathscr{A}$-regular point of $G$ and there exists an $\mathscr{A}$-preserving infinitesimal flex of $(G,p,\ell)$, then there also exists an $\mathscr{A}$-preserving finite flex of $(G,p,\ell)$.
\end{thm}

\begin{proof}  Let $K_{n+k}$ be a complete decorated graph on $V(G)$, where the decoration is inherited from $G$. Since the points and hyperplanes of $(G,p,\ell)$ affinely span  $\mathbb{R}^{d}$, the kernel of $df_{K_{n+k}}((p,\ell))$ is the space of all trivial infinitesimal motions of $(G,p,\ell)$ and the kernel of $d\tilde{f}_{K_{n+k}}((p,\ell))$ is the space of all $\mathscr{A}$-preserving trivial infinitesimal motions of $(G,p,\ell)$. So, since  $(G,p,\ell)$ has an infinitesimal flex which is $\mathscr{A}$-preserving, we have $\textrm{nullity }\big(d\tilde{f}_{G}((p,\ell))\big)>\textrm{nullity }\big(d\tilde{f}_{K_{n+k}}((p,\ell))\big)$, and hence \begin{equation*}\label{rankinproof}\textrm{rank }\big(d\tilde{f}_{G}((p,\ell))\big)<\textrm{rank }\big(d\tilde{f}_{K_{n+k}}((p,\ell))\big)\textrm{.}\end{equation*} The result now follows from Theorem \ref{flexthm1}.
\end{proof}

\begin{cor}
If,  in addition to the conditions in Theorem~\ref{flexthm2},  $(G,p,\ell)$ has a self-stress and 
%all infinitesimal flexes of $G(q,m)$ are in $\mathscr{A}$  for all  $(q,m)$ sufficiently close to $(G,p,\ell)$ (is "
all infinitesimal motions of $(G,p,\ell)$ are  $\mathscr{A}$-preserving, then there is an edge $e \in E(G)$ which is locally redundant in $\mathbb{W}_G$.
\end{cor}
\begin{proof}
%If all infinitesimal flexes of $(G,p,\ell)$ are in $\mathscr{A}$ then by the definition of $\mathscr{A}$,  all infinitesimal flexes of $(G,q,m)$ are in $\mathscr{A}$ for all $(q,m) \in \mathscr{A}$.
%The existence of an edge $e$ which is locally redundant in $\mathscr{A}$ follows from the % Need to show that $e$ is locally redundant in $\mathbb{W}_G$.TO BE COMPLETED.
There is an edge $e$ and a neighborhood $N'_{(p,\ell)}$  in $\mathscr{A}$ such that $N'_{(p,\ell)} \cap \tilde{f}_G^{-1}(\tilde{f}_G(p,\ell))=N'_{(p,\ell)}\cap \tilde{f}_{G\setminus e}^{-1}(\tilde{f}_{G\setminus e}(p,\ell))$ by Lemma~\ref{hs1}.
%$f $f^{-1}_G (f_G(q,m))=f^{-1}_{G \setminus e}((f_{G\setminus e}(q,m))$
There is also a neighborhood $N''_{(p,\ell)}$ in $\mathbb{W}_G$ such that $N''_{(p,\ell)}\cap f_G^{-1}(f_G(p,\ell)))\subseteq \mathscr{A}$ because all infinitesimal motions of $(G,p,\ell)$ are   $\mathscr{A}$-preserving.  Every sub-neighborhood of $N_{(p,\ell)}$ of $N''_{(p,\ell)}$ has the property $N_{(p,\ell)}\cap f_G^{-1}(f_G(p,\ell))\subseteq \mathscr{A}$ so we can choose $N_{(p,\ell)}$ small enough that $N_{(p,\ell)} \cap f_G^{-1}(f_G(p,\ell))=N_{(p,\ell)}\cap f_{G\setminus e}^{-1}(f_{G\setminus e}(p,\ell))$.
%There is also a neighborhood $N'_{(p,\ell)}$ of $(p,\ell)$ in $\mathbb{W}_B$ such that $ f_G^{-1}(f_G(q, m))) \subseteq \mathscr{A}$ for all $(q,m) \in \N'$ because all infinitesimal flexes of $(G,q,m)$ are in  $\mathscr{A}$.  
%Hence $e$ is locally redundant  in $W_G$ because $N_{p,\ell} \cap N'_{p,\ell}$ is a neighborhood of $(p,\ell)$ in $W_G$.
%We can choose a sub-neighborhood $N\subset N''\subseteq \mathbb{W}_B$ such that $N\cap \mathscr{A} \subseteq N''$ which shows that $e$ is locally redundant in $\mathbb{W}_B$.
\end{proof}

\subsection{Finite motions of extrusion-symmetric frameworks}

We now use the results above to show that, under suitable regularity assumptions, the infinitesimal flexes of a $t$-fold extrusion-symmetric point-hyperplane framework $(G,p,\ell)$  detected with our symmetry-adapted counts in Section~\ref{sec:fowler-guest} extend to  finite flexes. 

To this end, we fix an irreducible representation $\rho_i$ of $\mathbb{Z}_2^t$ and consider the affine space $(p,\ell)+X^{(\rho_{i})}$, where $X^{(\rho_{i})}$ is the space of all vectors in $\mathbb{R}^{d|V|+(d+1)|V_{\mathcal{H}}|}$ that are $\rho_i$-symmetric. (Recall Definition~\ref{def:sy} for bar-joint frameworks. The definition is analogous for point-hyperplane frameworks, as discussed in Section~\ref{subsec:blockdecompthyper}.)

Recall from Corollary~\ref{cor:blockph} and the discussion in Section~\ref{pinh} that under suitable conditions/pinning, the point-hyperplane rigidity matrix of $(G,p,\ell)$ can be transformed into the block form
\begin{equation*}
\widetilde{R}(G,p,\ell)
=\left(\begin{array}{ccc}\widetilde{R}_{0}(G,p,\ell) & & \mathbf{0}\\ & \ddots & \\\mathbf{0} &  &
\widetilde{R}_{r-1}(G,p,\ell) \end{array}\right)\textrm{.}
\end{equation*}
If we let $\tilde{f}^{(i)}_{G}((p,\ell))= f_{G}|_{(p,\ell)+X^{(\rho_{i})}}$, then we clearly have $\textrm{rank }\widetilde{R}_{i}(G,p,\ell) = \textrm{rank } d\tilde{f}^{(i)}_{G}((p,\ell))$.

We define a point $(q,m)$ to be a \emph{regular point of $G$ in $(p,\ell)+X^{(\rho_{i})}$} if there exists a 
  neighborhood $N_{(q,m)}$ of $(q,m)$  in $(p,\ell)+X^{(\rho_{i})}$  so that $\textrm{rank } d\tilde{f}^{(i)}_{G}((q,m))= \textrm{max}\{\textrm{rank }d\tilde{f}^{(i)}_{G}((q',m'))|\,(q',m')\in  N_{(q,m)}\}$. Clearly,   $(q,m)$ is a regular point of the complete graph $K_{|V|}$ in $(p,\ell)+X^{(\rho_{i})}$   if the points and hyperplanes of $(G,q,m)$ affinely span  $\mathbb{R}^{d}$,  because in that case $(G,q,m)$ only has trivial infinitesimal motions.

We have the following corollary of Theorem~\ref{flexthm2}.

\begin{thm}
\label{flexthm2ex}
Let $(G,p,\ell)$ be a point-hyperplane framework with $t$-fold extrusion symmetry, where the points and hyperplanes affinely span  $\mathbb{R}^{d}$. If $(p,\ell)$ is a regular point of $G$ in $(p,\ell)+X^{(\rho_{i})}$ and there exists a $\rho_i$-symmeric  infinitesimal flex of $(G,p,\ell)$, then there also exists a finite flex of $(G,p,\ell)$.
\end{thm}
\begin{proof} Since $(G,p,\ell)$ has a $\rho_i$-symmetric infinitesimal flex, we have $$\textrm{rank } \widetilde{R}_{i}(G,p,\ell)=\textrm{rank }d\tilde{f}^{(i)}_{G}((p,\ell)) <\textrm{rank }  d\tilde{f}^{(i)}_{K_{|V|}}((p,\ell)).$$ 
The result now follows from Theorem~\ref{flexthm1}.
\end{proof} 

A  simple, but useful observation 
  is that for the trivial irreducible representation $\rho_0$ of $\mathbb{Z}_2^t$, which assigns $1$ to each element of $\mathbb{Z}_2^t$, the $t$-fold extrusion symmetry of a point-hyperplane framework is preserved by making a `linear push'  along a $\rho_0$-symmetric infinitesimal motion. In fact, by the definition of $X^{(\rho_{0})}$, if $(G,p,\ell)$ has $t$-fold extrusion symmetry, and $(q,m)$ is an element of  $X^{(\rho_{0})}$, then $(G,p+q,\ell+m)$  still has $t$-fold extrusion symmetry with the same extrusion directions as $(G,p,\ell)$ because it still satifies the conditions in Definition \ref{def:extrsymph}.
  Thus,    the block-decomposition of the point-hyperplane rigidity matrix is preserved by making a linear push  along an infinitesimal motion in $X^{(\rho_{0})}$. This is clearly not true in general for other irreducible representations of $\mathbb{Z}_2^t$.

It follows that we may define a regular point of $G$ in $(p,\ell)+X^{(\rho_{0})}$ to be a point $(q,m)$ with the property that there exists a   neighborhood $N_{(q,m)}$ of $(q,m)$  in $(p,\ell)+X^{(\rho_{0})}$  so that $\textrm{rank }  \widetilde{R}_{0}(G,q,m)= \textrm{max}\{\textrm{rank } \widetilde{R}_{0}(G,q',m')|\,(q',m')\in  N_{(q,m)}\}$. So for checking $(q,m)$ for regularity in $(p,\ell)+X^{(\rho_{0})}$, we may simply consider the block matrix $\widetilde{R}_{0}(G,q,m)$ and compare its rank to the rank of the block matrices $\widetilde{R}_{0}(G,q',m')$.

It is easy to check that  the configurations of the frameworks in Examples~\ref{ex:1}, \ref{ex:2} and \ref{ex:3} are  regular points of $G$ in $(p,\ell)+X^{(\rho_{0})}$ (for their respective  extrusion symmetry). Thus, we may conclude that the detected fully-symmetric infinitesimal flexes are in fact finite.  The configuration of the cube framework in Example~\ref{ex:4}, however, is not a regular point of $G$ in $(p,\ell)+X^{(\rho_{0})}$ for $1$-fold extrusion symmetry (recall the discussion of this example in Section~\ref{sec:countph}). Thus, further analysis (e.g. a direct geometric analysis of the structure) is needed to conclude that the detected infinitesimal flexes are in fact finite.

Finally, we return to the example in Figure~\ref{fig:ex3}(a) and show that the infinitesimal flex is in fact finite. 
For simplicity, we label the vertices of each of the triangles $H_i$ as $a_i,b_i,c_i$, where $i=1,2,3$. and we denote by $G$ the underlying graph of the framework shown in Figure ~\ref{fig:ex3}(a). As we have already seen, for $i=1,2,3$, each of the frameworks corresponding to the subgraphs $\overline{H}_i$ of $G$ induced by the vertices in $V(G) \setminus V(H_i)$ has a simple extrusion symmetry and a finite motion, and hence  a locally redundant constraint. We can choose three independent locally redundant constraints, such as the edges $\{c_1,c_2\}$, $\{c_2,c_3\}$ and $\{c_1,c_3\}$ (since each of these occurs in only one of the $\overline{H}_i$). The subgraph of $G$ induced by $V(G) \setminus\{c_1,c_2,c_3 \}$  remains and has an extrusion symmetry (where the triangle $b_1,b_2,b_3$ is an extrusion of the triangle $a_1,a_2,a_3$) and hence a remaining locally redundant constraint. So we have 4 independent locally redundant constraints overall. The simple scalar counting gives $2|V(G)|-|E(G)|-3=-3$ and hence after removing 4 locally redundant edges we have a net finite freedom count of 1 which must be  the simultaneous rotation of the three triangles described in Section~\ref{sec:fowler-guest1}.

\section{The numerical linear push algorithm} \label{sec:linpush}
The numerical linear push numerical algorithm for detecting finite flexes applies most easily to frameworks which have a single infinitesimal motion.  We therefore remove the trivial infinitesimal motions from a framework by restricting the coordinates of a minimal subset of vertices to have fixed values.  For example, for a bar-joint framework in dimension $d=2$ we could specify that vertex $v_1$ has $x$ and $y$ coordinates $0$ and the vertex $v_2$ has at least  $x$ coordinate $0$.

Let $P=(c_1,\dots,c_{d(d+1)/2})$ be a minimal set of coordinates such that (i) the set includes all the coordinates of at least one point vertex and (ii) the restriction %$\bar{f}$
of the measurement map to constant values for the coordinates $c_1,\dots, c_{d(d+1)/2}$ gives a rigidity matrix which has no trivial motions. Here,  a minimal set means that there is no proper subset of $P$ which satisfies these conditions.  We say that these frameworks are \emph{minimally pinned} with respect to the coordinates in $P$ and for minimally pinned frameworks \emph{we assume that the measurement maps $f$ and $\tilde{f}$ incorporate this restriction}.

Suppose that $(\dot{p},\dot{\ell})$ is the only infinitesimal flex in a minimally pinned point-hyperplane framework.  Theorem \ref{flexthm2} provides the basis for a numerical algorithm which gives a sufficient condition that this infinitesimal flex extends to a finite flex.  The algorithm determines when there is an affine (i.e.  linear) subspace  $\mathscr{A} \subset \mathbb{R}^{dn+(d+1)k}$ which satisfies the conditions of Theorem \ref{flexthm2}.

We say that a framework $(G,p,\ell)$ has a finite flex which is \emph{linearly detectable by pinning} if 

\begin{itemize}
\item[(1)] the kernel of its rigidity matrix is generated by trivial infinitesimal motions together with a single infinitesimal flex, i.e.  the vector space of infinitesimal motions of the framework has dimension $d(d+1)/2+1$, and

\item[(2)] there exists a minimal pinning of the framework and an affine subspace $\mathscr{A} \subset \mathbb{R}^{dn+(d+1)k}$ such that $(p,\ell)$ is an $\mathscr{A}$-regular point and the infinitesimal flex $(\tilde{\dot{p}},\tilde{\dot{\ell}})$ generates the kernel of $d\tilde{f}(p,\ell)$.  
\end{itemize}

Theorem \ref{flexthm2} shows that this flex is indeed finite.  We also say that the infinitesimal flex extends to a finite flex which is linearly detectable by pinning.

The linear push algorithm assumes that we can determine the rigidity matrix $R(G,q,m)$ at any point $(q,m) \in\mathbb{R}^{dn+(d+1)k}$ but does not require that we can construct frameworks which are equivalent to $(G,p,\ell)$.

Note that $d\tilde{f}(q,m)$ can be obtained by writing $df(q,m)$ in a coordinate basis which contains an orthogonal basis for $\mathscr{A}$ and defining the columns of $d\tilde{f}(q,m)$ to be the columns of $df(q,m)$ which correspond to this coordinate basis of $\mathscr{A}$.   We can check whether $(p,m)$ is an $\mathscr{A}$-regular point by selecting an orthogonal basis for $\mathscr{A}$,  selecting a random point $(q,m) \in \mathscr{A}$ using this basis and checking that rank$(d\tilde{f}(q,m))$=rank$(d\tilde{f}(p,\ell))$.  This is a valid test because almost all points in $\mathscr{A}$ are $\mathscr{A}$-regular,  there are $\mathscr{A}$-regular points in every neighbourhood of $(p,\ell)$, %are dense in $\mathscr{A}$ in eveery nei 
and rank$(d\tilde{f}(q,m))$ is maximal on $\mathscr{A}$-regular points.  We can check whether an infinitesimal flex is $\mathscr{A}$-preserving by checking whether it is tangent to $\mathscr{A}$.

The iterative algorithm computes a vector space $\mathscr{B}$ where  $\mathscr{A}=(p,\ell)+\mathscr{B}$ as follows.

\begin{algorithmic}
\State Input: A minimally pinned point-hyperplane framework $(G,p,\ell)$  which has a single infinitesimal flex $(\dot{p},\dot{\ell})$.
\State $v \gets (\dot{p},\dot{l}),\ \mathscr{B} \gets \langle v \rangle,\ \mathcal{A} \gets (p,\ell)+\mathcal{B},\ (q,m)\gets(p,\ell)$\;
\While{$\rank(df(q,m))\leq\rank(df(p,\ell))$ and $X(q,m) \not\subseteq \mathscr{B}$}
\State  $(q,m) \gets$ a random point in $\mathscr{A}$ (i.e. an $\mathscr{A}$-regular point) 
%\If{$\rank(df(q,m))>\rank(df(p,\ell))$}
%\State exit,  with the determination that $(\dot{p},\dot{\ell})$ does not extend to a linearly detectable flex. 
%\Else
%{\If {$X(q,m) \subseteq \mathscr{B}$ (where $X(q,m)$ is the  null space of $df(q,m)$)}
%\State exit: with the determination that the extension is linearly detectable in the affine space $\mathscr{A}$ and the framework has a locally redundant edge. 
%\Else{
\State  $v \gets$ the unit vector which generates $X(q,m)$ 
\State $\mathscr{B} \gets \mathscr{B} \cup \langle v \rangle$ 
\State $\mathscr{A} \gets (p,\ell)+\mathscr{B}$
%}
%\EndIf
%}
%\EndIf
%exit,  with the determination that $(\dot{p},\dot{\ell})$ does not extend to a linearly detectable flex. 
\EndWhile
\If{$\rank(df(q,m))>\rank(df(p,\ell))$}
\State $(\dot{p},\dot{\ell})$ does not extend to a linearly detectable finite flex. 
\Else
%{\If {$X(q,m) \subseteq \mathscr{B}$ (where $X(q,m)$ is the  null space of $df(q,m)$)}
\State $(\dot{p},\dot{\ell})$ extends to a linearly detectable finite flex in the affine space $\mathscr{A}$ and the framework has a locally redundant edge. 
\EndIf
\end{algorithmic}

%\textcolor{red}{I think the while loop needs a stopping condition. So how about: while $\rank(df(q,m))\leq \rank(df(p,\ell))$ and $X(q,m)$ is not contained in $\mathscr{B}$, choose random point for $(q,m)$ and assign to $v$ the unit vector.... Then do the two if statements.}

%Input: A minimally pinned point-hyperplane framework $(G,p,\ell)$ which has a single infinitesimal flex $(\dot{p},\dot{\ell})$.

%Initialise : 

%$v=(\dot{p},\dot{\ell})$,  $\mathscr{B}=\langle v\rangle$, $\mathscr{A}=(p,l)+\mathscr{B}$,  $(q,m)=(p,\ell)$

%Loop start: 
%Reset   
%$(q,m)$ to a random point in $\mathscr{A}$ (i.e. to an $\mathscr{A}$-regular point) 

%If rank$(df(q,m))>$rank$(df(p,\ell))$
%exit,  with the determination that $(\dot{p},\dot{\ell})$ does not extend to a linearly detectable flex. 

%Otherwise if $X(q,m) \subseteq \mathscr{B}$ (where $X(q,m)$ is the  null space of $df(q,m)$)
%exit: with the determination that the extension is linearly detectable in the affine space $\mathscr{A}$.

%Otherwise:  reset 
%$v$ to the unit vector which generates $X(q,m)$,  reset $\mathscr{B} \to \mathscr{B}  \cup \langle v\rangle$, reset $\mathscr{A}=(p,\ell)+\mathscr{B}$ and iterate to loop start.

The algorithm must terminate after at most $nd+(d+1)k$ steps because $\mathscr{A}  \subseteq \mathbb{R}^{nd+(d+1)k}$
and the dimension of $\mathscr{A}$ increases by at least one on each iteration.

\begin{lem} \label{lin-push}
Let $G$ be an isostatic point-hyperplane graph in $\mathbb{R}^d$ and let $(G,p,\ell)$ be a minimally pinned framework on $G$ whose rigidity matrix has a kernel with dimension one.  
The extension of the infinitesimal flex of $(G,p,\ell)$ to a finite flex is linearly detectable if and only if the linear push algorithm terminates with this determination.
\end{lem}
\begin{proof}
Suppose that the linear push algorithm terminates with the determination that the extension is linearly detectable.  The point $(q,m)$ in the algorithm at termination is $\mathscr{A}$-regular because $X(q,m) \subseteq \mathscr{B}$ and rank$(df(q,m))$ is maximal.
%ts projection on each of the axes of an orthogonal basis of $\mathscr{A}$ is $\mathscr{a}$-regular.  
Since rank$(df(p,l)) \geq$ rank$(df(q,m))$ by the termination condition and $(q,m)$ is $\mathscr{A}$-regular we have rank$(df(p,l))=$ rank$(df(q,m))$.  The infinitesimal flex $v$ is tangent to $\mathscr{A}$ by construction so it extends to a finite flex in $\mathscr{A}$.

Conversely, suppose that the extension of the infinitesimal flex $(\dot{p},\dot{\ell})$  to a finite flex is linearly detectable in an affine space $\mathscr{A}_{min}=(p,\ell)+\mathscr{B}_{min}$ which is minimal with respect to containment.  Let $\mathscr{B}_s$, $(q_s,m_s)$ and  $v_s$ be $\mathscr{B}$, $(q,m)$ and $v$ respectively at the start of iteration step $s$.   This means that $\mathscr{B}_1=\langle(\dot{p}_1,\dot{\ell}_1)\rangle$, $(q_1,m_1)=(p,\ell)$, $v_1=(\dot{p},\dot{\ell})$ and $\mathscr{B}_s=\langle(\dot{p},\dot{\ell}), v_1,\dots,v_{s-1}\rangle$.
We will show that $\mathscr{B}_s \subseteq \mathscr{B}_{min}$ after each iteration and hence the linear push algorithm must terminate with the determination that $(\dot{p},\dot{\ell})$ is linearly detectable.  Clearly $\mathscr{B}_1 \subseteq \mathscr{B}_{min}$ because $(\dot{p},\dot{ \ell}) \in \mathscr{B}_{min}$. Assume for induction that $\mathscr{B}_s \subseteq \mathscr{B}_{min}$.  This implies that $\mathscr{B}_{s+1} \subseteq \mathscr{B}_{min}$ because otherwise $v_s \not\in \mathscr{B}_{min}$ and $\textrm{rank}(d\tilde{f}(q_{s+1},m_{s+1})) >\textrm{rank}(d\tilde{f}(p,\ell))$ (where $\tilde{f}$ is the restriction of $f$ to $\mathscr{A}_{min}$) which is a contradiction since $(q_{s+1},m_{s+1})$ is an $\mathscr{A}_s$-regular point and $\mathscr{A}_s \subseteq \mathscr{A}_{min}$ by the induction hypothesis.
\end{proof}

The linear push algorithm gives a sufficient condition that a single infinitesimal flex in a minimally pinned framework extends to a finite flex when a specific set of vertices are chosen as the pinning vertices.  Many of the simple  frameworks on isostatic  graphs which are not rigid do indeed have a set of vertices such that the infinitesimal flex with these vertices pinned extends to a linearly detectable finite flex.  However it may be that not all choices for the minimal set of pinning vertices lead to a linearly detectable finite flex.

For example, the $2$-dimensional bar-joint framework $(K_{3,3},p)$ on the complete  bipartite graph $K_{3,3}$ in which the points of the vertices in one partite set lie on a line and the points of the vertices in the other partite set lie on an orthogonal line has a 4-dimensional space of infinitesimal motions,  all of which extend to finite flexes \cite{bolker}.  However numerical calculation of the rigidity matrix shows that the extension to a finite flex is linearly detectable only if a pair of non-adjacent vertices is chosen as the pinning vertices.

The situation for a framework with extrusion symmetry is more favourable however.  As shown in Lemma~\ref{affine-invariance-hyperplane}, a point-hyperplane framework with $t$-fold extrusion symmetry retains the $t$-fold extrusion symmetry under a general affine transformation and hence has  a linearly dectectable finite flex whichever vertices are used to pin it. 

\begin{lem} \label{lin-push-any-pin}
Let $G=H\square K_2^{\square t}$  and let $(G,p,\ell)$ be a point-hyperplane framework with $t$-fold extrusion symmetry.  %Let  $P$ be a minimal pinning set for $(H,p)$.   
If a minimal pinning $P$ of $(G,p,\ell)$ has a single infinitesimal flex $(\dot{p},\dot{\ell})$,  
 then this flex extends to a linearly detectable finite flex.

\end{lem}
\begin{proof}
$(G,p,\ell)$ has an infinitesimal flex $(\dot{p_s},\dot{\ell_s})$
which is $\rho_0$-symmetric (i.e.  fully-symmetric) by the remark following Example~\ref{ex:2}.  Let $\mathscr{A}_0=(p,\ell)+X^{(\rho_{0})}$ where $X^{(\rho_{0})} \in \mathbb{R}^{nd+k(d+1)}$ is the space of all vectors which are $\rho_0$-symmetric. The point $(p,\ell)$ is $\mathscr{A}_0$-regular and the infinitesimal flex $(\dot{p_s},\dot{\ell_s})$ is $\mathscr{A}_0$-preserving.

 Since $P$ is a minimal pinning set of coordinates for $(G,p,\ell)$ we claim that there is an infinitesimal rigid motion $(\dot{p}_r,\dot{\ell}_r)$ such that  $(\dot{p}_s+\dot{p}_r,\dot{\ell}_s+\dot{\ell}_r)$ is zero on all the coordinates in $P$ .  Let $u$ be a point vertex all of whose coordinates are in $P$.  Then we can choose a translation such that $\dot{p}_r(u)=-\dot{p}_s(u)$  and infinitesimal rotations round $p(u)$ so that the remaining coordinates satisfy $\dot{c}^s_i =\dot{c}^r_i$.  This shows that $(\dot{p},\dot{\ell})=(\dot{p}_s+\dot{p}_r,\dot{\ell}_s+\dot{\ell}_r)$ since this is the only  infinitesimal flex of the framework $(G,p,\ell)$ which is pinned on $P$.

Recall that if $(\dot{p_r},\dot{\ell_r})$ is an infinitesimal rigid motion of $(G,p,\ell)$ then $(G,p+\alpha\dot{p_r}, \ell+\alpha \dot{\ell_r})=(G,A_{\alpha}p,A_{\alpha}\ell)$,  where $A_{\alpha}$ is an affine transformation,  for all $\alpha \in \mathbb{R}$.  The frameworks $(G,p+\alpha\dot{p_s},\ell+\alpha\dot{\ell_s})$ have $t$-fold extrusion symmetry for all $\alpha \in \mathbb{R}$ because $(\dot{p_s},\dot{\ell_s})$ is fully-symmetric with respect to the extrusion directions $\tau$.  This implies that the frameworks $(G, p+\alpha\dot{p},\ell+\alpha\dot{\ell })$ all have $t$-fold extrusion symmetry 
by Lemma \ref{affine-invariance-hyperplane} because $(G, p+\alpha\dot{p},\ell+\alpha\dot{\ell })$ is an affine transform of $(G, p+\alpha\dot{p_s},\ell+\alpha\dot{\ell_s})$.

Let $\mathscr{A}_{\alpha}=A_{\alpha}\mathscr{A}_0$.  All frameworks 
$(G,p+\alpha\dot{p}, \ell+\alpha \dot{\ell})$ for $\alpha \in \mathbb{R}$ have an infinitesimal flex and hence $(p,\ell)$ is an $\mathscr{A}_{\alpha}$-regular point. $(\dot{p},\dot{\ell})$ is tangent to $\mathscr{A}_{\alpha}$ by construction which shows that $(\dot{p},\dot{\ell})$ extends to a linearly detectable flex by Theorem \ref{flexthm2}.
\end{proof}

The linear push algorithm is easily extended to the case when $(G,p,\ell)$ has only one state of self-stress but has $m >1$ infinitesimal flexes (so that the graph $G$ is generically under-constrained and has fewer than $d|V|-d(d+1)/2$ edges).  In this case if all %$dN-d(d+1)/2-E+1$ 
infinitesimal flexes of $(G,p,\ell)$ extend to finite flexes then each edge with a non-zero stress coefficient is a locally redundant edge. The numerical algorithm is extended by simply initialising $\mathscr{B}$ to the vector space generated by all the infinitesimal flexes of the pinned framework $(G,p,\ell)$.

The algorithm can also be extended to the case that $(G,p,\ell)$ has more than one state of self-stress by deleting stressed edges from the framework until there is only one state of self-stress  remaining. If all of the infinitesimal flexes of this reduced framework extend to finite flexes then the remaining state of self-stress identifies the edges which are locally redundant in the reduced framework. These edges are also locally redundant in the original framework since local redundance in a sub-framework clearly implies local redundance in the containing framework.  A more detailed analysis of the extended linear push algorithm is the subject of continuing research.

%The definition of a linearly detectable flex in this section is quite restrictive since it assumes a minimally pinned framework which has a single infinitesimal flex. For many practical applications (such as geometric constraint solving in CAD) it would be desirable to remove both of these restrictions. We are investigating enhancements to the numerical linear push algorithm which will achieve this.

\section*{Acknowledgements}  

We would like to thank Patrick Fowler and  Simon Guest for helpful discussions during research meetings in Cambridge and Lancaster in 2019.

\bibliographystyle{abbrv}
\def\lfhook#1{\setbox0=\hbox{#1}{\ooalign{\hidewidth
  \lower1.5ex\hbox{'}\hidewidth\crcr\unhbox0}}}

\end{document}